\documentclass[11pt]{article}
\usepackage{amssymb,amsmath,amsthm,latexsym}
\usepackage[cp1251]{inputenc}
\usepackage{dsfont}
\usepackage{amssymb}
\usepackage{indentfirst}
\usepackage{graphicx}
\usepackage{caption}
\usepackage{multicol}
\usepackage{enumitem}
\usepackage{hyperref}
\usepackage{tikz-cd}
\hypersetup{citecolor=red,}   
\usepackage{amsfonts, array, hhline} 
\usepackage{color}
\usepackage{lmodern}
\usepackage[left=1in,right=1in,top=1in,bottom=1.5in]{geometry}
\usepackage{mathrsfs}
\usepackage[title]{appendix}

\theoremstyle{definition}
\makeatletter
\newtheorem*{rep@theorem}{\rep@title}
\newcommand{\newreptheorem}[2]{%
\newenvironment{rep#1}[1]{%
 \def\rep@title{#2 \ref{##1}}%
 \begin{rep@theorem}}%
 {\end{rep@theorem}}}
\makeatother

\numberwithin{equation}{section}

\newtheorem{dfn}{Definition}[section]

\newtheorem{thm}[dfn]{Theorem}
\newtheorem{lm}[dfn]{Lemma}

\newtheorem{crl}[dfn]{Corollary}

\newreptheorem{lm}{Lemma}

\newtheorem{thmA}{Theorem}

\newtheorem{thmM}{Theorem}

\newreptheorem{thmM}{Theorem}
\newtheorem{lmA}{Lemma}

\newreptheorem{lmA}{Lemma}

\theoremstyle{remark}
\newtheorem{rmk}[dfn]{Remark}

\newcommand{\e}{\varepsilon}
\newcommand{\pt}{\partial}
\newcommand{\mc}[1]{\mathcal{#1}}
\newcommand{\mf}[1]{\mathfrak{#1}}
\newcommand{\ms}[1]{\mathscr{#1}}
\newcommand{\ol}[1]{\overline{#1}}
\newcommand{\ra}{\rightarrow}
\newcommand{\R}{\mathbb{R}}

\newcommand{\A}{\mathbb{A}d\mathbb{S}}
\renewcommand{\H}{\mathbb{H}}
\newcommand{\E}{\mathbb{E}}
\renewcommand{\S}{\mathbb{S}}
\newcommand{\RP}{\mathbb{RP}}
\newcommand{\len}{{\rm len}}

\newcommand{\conv}{{\rm conv}}
\newcommand{\clconv}{{\rm conv}}
\newcommand{\cl}{{\rm cl}}

\newcommand{\ct}{\textsc{ct}}
\newcommand{\inter}{{\rm int}}

\renewcommand{\hat}{\widehat}
\renewcommand{\tilde}{\widetilde}

\DeclareMathOperator{\sys}{sys}

\DeclareFontFamily{U}{mathb}{}
\DeclareFontShape{U}{mathb}{m}{n}{<-5.5> mathb5 <5.5-6.5> mathb6 
<6.5-7.5> mathb7 <7.5-8.5> mathb8 <8.5-9.5> mathb9 <9.5-11> mathb10 
<11-> mathb12}{}
\DeclareSymbolFont{mathb}{U}{mathb}{m}{n}
\DeclareFontSubstitution{U}{mathb}{m}{n}

\DeclareMathSymbol{\bdia}{\mathbin}{mathb}{"0C}

\title{Polyhedral surfaces in anti-de Sitter (2+1)-spacetimes}
\author{Roman Prosanov}
\date{}

\AtEndDocument{\bigskip{\footnotesize
\par
  \textsc{Autonomous University of Barcelona, Department of Mathematics, Building C Science Faculty, 08193 Bellaterra (Barcelona), Spain} \par
  \textit{E-mail}: \texttt{roman.prosanov@univie.ac.at}
}}

\begin{document}

\maketitle

\begin{abstract}
We first prove that given a Fuchsian representation $\rho_\circ: \pi_1S \ra {\rm PSL}(2,\R)$, where $S$ is a closed oriented surface of genus $\geq 2$, any
hyperbolic cone-metric on $S$ with  cone-angles $>2\pi$ isometrically embeds as a future-convex
bent Cauchy surface in a globally hyperbolic maximal Cauchy compact (GHMC) anti-de Sitter
(2+1)-spacetime whose left representation is $\rho_\circ$. Second, we show that given any two such cone-metrics, there exists a GHMC anti-de Sitter
(2+1)-spacetime in which the cone-metrics embed simultaneously, one as a future-convex bent Cauchy surface and one as a past-convex. Furthermore, in both cases we establish that such a spacetime and embeddings are unique provided that the cone-metrics are sufficiently small.
\end{abstract}

\section{Introduction}

\subsection{Motivation}

The motivation for this paper is twofold. First, we want to take another step in the investigation of anti-de Sitter (2+1)-spacetimes, interest in which arises from various perspectives. One source of interest to anti-de Sitter spacetimes is their role in the AdS/CFT correspondence, see, e.g., the fundamental paper of Witten~\cite{Wit2}. On the other hand, (2+1)-dimensional gravity is an important test ground for quantum gravity theories, see, e.g., the book of Carlip~\cite{Car} for an introduction to the subject. 

The (2+1)-spacetimes that satisfy the Einstein equations necessarily have constant sectional curvature. Up to scaling, the sectional curvature is $1$, $0$ or $-1$, and the respective spacetimes are called \emph{de Sitter, Minkowski} or \emph{anti-de Sitter}. Here we focus on a special type of (2+1)-spacetimes, called \emph{GHMC} spacetimes, see the definitions below. Anti-de Sitter geometry is a Lorentzian cousin of hyperbolic geometry. Due to a special structure of the isometry group of the model anti-de Sitter space in dimension 3, there are deep connections between the geometry of GHMC anti-de Sitter (2+1)-spacetimes and Teichm\"uller theory, see, e.g., the pioneering article~\cite{Mes} of Mess. GHMC (2+1)-spacetimes are homeomorphic to $S \times \R$, where $S$ is a closed oriented surface. We consider the case of surfaces of genus $\geq 2$. In~\cite{Mes} Mess showed that GHMC anti-de Sitter (2+1)-spacetimes are classified by their holonomy, which belongs to $\mc T \times \mc T$, the square of the Teichm\"uller space of $S$. 

In the present article we study two Alexandrov--Weyl-type problems in this context. In particular, we prove that a GHMC anti-de Sitter (2+1)-spacetime is uniquely determined by the intrinsic geometry of two its spacelike slices, provided that the slices are convex in the opposite directions, are bent and their intrinsic metrics are small enough in some sense. Here bent is a weakening of the notion of polyhedrality, see Section~\ref{sec:bentsurf} for details. We also prove that all pairs of concave hyperbolic cone-metrics on $S$ are realized as the intrinsic metrics of such pairs of slices. This resolves the existence part of Question 3.6 in~\cite{BBD+} and makes a progress towards its uniqueness part. To our knowledge, the only previously known global rigidity result of such type was the case when the intrinsic metrics of the slices have constant curvatures $\kappa^+,\kappa^- < -1$ satisfying $\kappa^+=-\kappa^-/(\kappa^-+1)$. This follows from the work~\cite{BMS} of Bonsante--Mondello--Schlenker. In particular, no smooth counterpart to our rigidity result is known.

The second goal of this paper is to exhibit the significance of the projective nature of anti-de Sitter spacetimes. The main benefit of considering anti-de Sitter geometry as a subgeometry of projective geometry is the possibility to use geometric transitions. The initial interest in geometric transitions comes from the study of degenerations of geometric structures; see, e.g., the thesis of Hodgson~\cite{Hod} and the article~\cite{Por} of Porti in the Riemannian case and the articles~\cite{Dan, RS, Dia} of Danciger, Riolo--Seppi and Diaf in the case of changing signature. In the Riemannian case geometric transition was notably used as one of the ingredients in the proofs of the geometrization theorem for 3-orbifolds~\cite{BLP, CHK} due to Boileau--Leeb--Porti and Cooper--Hodgson--Kerckhoff. In our article we rely on geometric transitions from anti-de Sitter geometry to Minkowski and co-Minkowski geometry. In particular, the proofs of the main results are based on a recent solution to similar problems in Minkowski geometry~\cite{FP} due to Fillastre and the author. We refer to the survey~\cite{FS5} of Fillastre--Seppi on geometric transitions between projective subgeometries. 

\subsection{Statement of the results}
\label{sec:statement}

We refer to the book~\cite{One} of O'Neill as a main reference on Lorentzian geometry.
For us, a \emph{spacetime} is a connected, oriented and time-oriented Lorentzian manifold. A \emph{Cauchy hypersurface} in a spacetime is a hypersurface $\Sigma$ such that every inextensible causal curve intersects $\Sigma$ exactly once. A spacetime is called \emph{globally hyperbolic} (abbreviated as GH) if it admits a Cauchy hypersurface. All Cauchy hypersurfaces are homeomorphic to each other. For a Cauchy hypersurface $\Sigma$ Geroch~\cite{Ger} proved that the spacetime admits a parameterization $\Sigma \times \R$, where every fiber $\Sigma \times \{r\}$ is a Cauchy hypersurface. A GH spacetime is \emph{Cauchy compact} if its Cauchy hypersurfaces are compact. A GH spacetime $\Omega$ is \emph{maximal} if every isometric embedding $\Omega \ra \Omega'$ into another GH spacetime that sends some Cauchy hypersurface of $\Omega$ to a Cauchy hypersurface of $\Omega'$ is onto. A globally hyperbolic maximal Cauchy compact spacetime is abbreviated as GHMC. From now on we focus on dimension (2+1). In such case, $\Sigma \cong S$ where $S$ is a closed oriented surface of genus $k$. For anti-de Sitter spacetimes, $k \geq 1$. However, the case $k=1$ is somewhat exceptional and we focus on the case $k\geq 2$.  

Let $\A^3$ be anti-de Sitter 3-space, see the definition in Section~\ref{sec:geometries}. Denote the identity component of its isometry group by $G_-$. Every GHMC anti-de Sitter (2+1)-spacetime has a holonomy representation $\rho: \pi_1S \ra G_-$ defined up to conjugation by $G_-$. 

Let $G={\rm PSL}(2, \R)$, the identity component of the isometry group of the hyperbolic plane $\H^2$. The special feature of anti-de Sitter geometry in dimension 3 is that there is a canonical isomorphism $G_-\cong G \times G$. For a holonomy representation $\rho: \pi_1S \ra G_-$, let $\rho^l,$ $\rho^r: \pi_1S \ra G$ be its left and right projections with respect to $G_- \cong G \times G$. In~\cite{Mes} Mess proved that both $\rho^l,$ $\rho^r$ are \emph{Fuchsian}, i.e., discrete and faithful orientation-preserving representations $\pi_1S \ra G$. Furthermore, Mess showed that each pair of Fuchsian representations is realized as a holonomy representation of a unique GHMC anti-de Sitter (2+1)-spacetime. 

Let $\Sigma$ be a convex Cauchy surface in a GHMC anti-de Sitter (2+1)-spacetime $\Omega$. The first anti-de Sitter version of the Alexandrov--Weyl problem that we consider here studies the intrinsic metric of $\Sigma$ and investigates up to which degree this intrinsic metric prescribes the pair $(\Omega, \Sigma)$. Dimensional considerations show that the missing information matches the size of half the holonomy. Thus, one can conjecture that one can determine $(\Omega, \Sigma)$ from the intrinsic metric on $\Sigma$ and half the holonomy. Every convex Cauchy surface is either \emph{future-convex} or \emph{past-convex}, depending on the direction in which it is convex with respect to the time-orientation of $\Omega$.

We focus on the polyhedral side of things. If $\Sigma$ is polyhedral, then the intrinsic metric is a concave hyperbolic cone-metric, i.e., it is locally isometric to the hyperbolic plane $\H^2$ except at finitely many points, where it is isometric to hyperbolic cones with cone-angles $>2\pi$. However, such metrics admit convex isometric realizations in anti-de Sitter (2+1)-spacetimes that are not polyhedral in the naive sense. They can be additionally bent along geodesic laminations. We call a surface \emph{bent} if it satisfies the respective weak notion of polyhedrality, see Section~\ref{sec:bentsurf} for a precise definition. Bent surfaces provide the right setting for the polyhedral anti-de Sitter versions of the Alexandrov--Weyl problem.

We can now formulate the first main result of our article:

\begin{thmA}
\label{main}
Let $S$ be a closed oriented surface of genus $\geq 2$, $V \subset S$ a finite non-empty set, $\rho_\circ: \pi_1 S \ra G$ a Fuchsian representation and $d$ a concave hyperbolic cone-metric on $(S, V)$. Then there exist a GHMC anti-de Sitter (2+1)-spacetime $\Omega\cong S\times \R$ whose left representation is $\rho_\circ$ as well as a future-convex bent isometric embedding $(S, d) \rightarrow \Omega$. Furthermore, there exists a non-empty open set $U=U(\rho_\circ)$ in the space of cone-metrics on $(S, V)$ for which the realization is unique.
\end{thmA}

Here $U$ is a ``strong neighborhood of zero'' in the space of cone-metrics, which we will specify further on. In other words, we prove the uniqueness part provided that $d$ is ``sufficiently small'' is some strong sense. 

Note that if $V$ is empty, so $d$ is just a hyperbolic metric, then it is classical that in such case there exists a unique such $\Omega$ and a unique future-convex bent isometric embedding $\phi: (S, d) \rightarrow \Omega$. Indeed, it follows from a combination of the Kerckhoff--Thurston earthquake theorem~\cite{Ker} and of observations of Mess~\cite{Mes} that there exists a unique $\Omega$, whose left representation is $\rho_\circ$ and whose intrinsic metric of the future-convex boundary of the convex core is $d$. See Section~\ref{sec:domain} for a definition of the convex core and Section~\ref{sec:holonaway} for a connection between earthquakes and anti-de Sitter geometry. On the other hand, it is straightforward to deduce from the definition of bent surface that the image of any such embedding $\phi$ of a hyperbolic surface must coincide with the future-convex boundary of the convex core.

Further dimensional considerations show that if we have two convex Cauchy surfaces $\Sigma^\pm \subset \Omega$, then the pair of their intrinsic metrics, in principle, has enough amount of data to prescribe the triple $(\Omega, \Sigma^+, \Sigma^-)$. For technical reasons one needs to assume that $\Sigma^\pm$ are convex in the different directions, i.e., one is future-convex and the other one is past-convex. In such case they bound a totally convex subset inside $\Omega$. We now formulate the second main result of our article:

\begin{thmA}
\label{main2}
Let $S$ be a closed oriented surface of genus $\geq 2$, $V^\pm \subset S$ two finite non-empty sets and $d^\pm$ concave hyperbolic cone-metrics on $(S, V^\pm)$. Then there exist a GHMC anti-de Sitter (2+1)-spacetime $\Omega\cong S\times \R$ as well as future-/past-convex bent isometric embeddings $(S, d^\pm) \rightarrow \Omega$ respectively. Furthermore, there exist a non-empty open set $U$ in the space of pairs of cone-metrics on $(S, V^\pm)$ for which the realization is unique.
\end{thmA}

Here if $V^\pm$ are empty, then the existence was shown by Diallo in~\cite[Appendix A]{BDMS}. The uniqueness is fully open.

\subsection{Proof ideas}

The initial setup for the proofs of the both results is the \emph{continuity method} introduced by Weyl~\cite{Wey} and Alexandrov~\cite{Ale}. We mostly focus our exposition on the proof of Theorem~\ref{main}. The proof of Theorem~\ref{main2} is quite similar, except few details, which we will mention in the end of the section.

Fix a pair $(S, V)$ and a Fuchsian representation $\rho_\circ: \pi_1S \ra G$. By Mess~\cite{Mes}, the space of GHMC anti-de Sitter (2+1)-spacetimes whose left representation is $\rho_\circ$ is parameterized by $\mc T$, which is the space of Fuchsian representations $\pi_1S \ra G$ up to conjugation. Every convex bent surface is uniquely determined by the position of its vertices. Hence, the space of future-convex bent surfaces in such spacetimes with vertices marked by $V$ can be parameterized by a finite-dimensional manifold $\mc P_-^s=\mc P_-^s(\rho_\circ, V)$, which is a fibration over $\mc T$. (The meaning of the superscript ``s'' will be clarified in Section~\ref{sec:bentsurf}.)

On the other hand, there is a natural space $\mc D_-^s=\mc D_-^s(S, V)$ of concave hyperbolic cone-metrics on $(S, V)$ up to isotopy. This is also a finite-dimensional manifold. By considering the intrinsic metric of a bent surface, one defines the \emph{intrinsic metric map}
\[\mc I_-^s: \mc P_-^s \ra \mc D_-^s,\]
which is continuous. We prove that it is surjective and show that there is a subset $U \subset \mc P_-^s$ such that for every $x \in U$, $\mc I_-^s(x)$ has only preimage in $\mc P_-^s$. We conjecture that $\mc I_-^s$ is a homeomorphism. 

In order to describe the main ingredient of the proof, we need to turn to the Minkowski side of things. A GHMC Minkowski spacetime is either future-complete or past-complete. In~\cite{Mes} Mess parameterized future-complete GHMC Minkowski spacetimes by $T\mc T$, the tangent bundle of $\mc T$. A version of the Alexandrov--Weyl problem can be formulated if instead of the left representation one prescribes the \emph{linear part} of the holonomy (i.e., the base point in the Mess parameterization) and instead of a hyperbolic cone-metric one prescribes a Euclidean cone-metric.

This problem was solved in~\cite{FP} by Fillastre--Prosanov. The proof is similarly based on the continuity method. We fix $(S, V)$ and $\rho_\circ$. Let $\mc P_0^s=\mc P_0^s(\rho_\circ, V)$ be the space of future-convex polyhedral surfaces in GHMC Minkowski (2+1)-spacetimes. (We highlight that in the Minkowsi case every bent surface is polyhedral, as follows from~\cite{FP}.) Let $\mc D_0^s$ be the space of concave Euclidean cone-metrics on $(S, V)$ up to isotopy. There is the intrinsic metric map
\[\mc I_0^s: \mc P_0^s \ra \mc D_0^s.\]
The following result was shown in~\cite{FP}: 

\begin{thmM}
\label{minkowski}
$\mc I_0^s$ is a $C^1$-diffeomorphism. 
\end{thmM}

Both spaces $\mc P_0^s$ and $\mc D_0^s$ have a natural $\R$-action by scaling and $\mc I_0^s$ is $\R$-equivariant. Denote the respective $\R$-quotients by $\S(\mc P_0^s)$ and $\S(\mc D_0^s)$ and the induced map by $\S(\mc I_0^s)$. 

Minkowski geometry can be considered as the infinitesimal version of anti-de Sitter geometry. Our further construction is a development of this observation. There is a special topological end of the space $\mc P_-^s$ corresponding to a ``fully degenerate configuration''. Using geometric transitions we blow-up $\mc P_-^s$ at this end, obtaining the space $\mc P_\vee^s$, which is $\mc P_-^s \cup \S(\mc P_0^s)$ endowed with a natural topology of a manifold with boundary. Similarly we obtain a blow-up $\mc D_\vee^s$ of $\mc D_-^s$, which is $\mc D_-^s\cup\S(\mc D_0^s)$. The maps $\mc I_-^s$ and $\S(\mc I_0^s)$ glue together into a map
\[\mc I_\vee^s: \mc P_\vee^s \ra \mc D_\vee^s,\]
which is $C^1$ near $\pt\mc P_\vee^s$. The proof of Theorem~\ref{main} follows easily from Theorem~\ref{minkowski} and two main lemmas:

\begin{lmA}
\label{ml}
The differential of $\mc I_\vee^s$ is non-degenerate on $\pt\mc P_\vee^s$.
\end{lmA}

\begin{lmA}
\label{ml2}
The map $\mc I_\vee^s$ is proper.
\end{lmA}

The proof of Theorem~\ref{main2} follows the same pattern. The necessary Minkowski result, see Theorem~\ref{minkowski2}, was also basically established by Fillastre--Prosanov in~\cite{FP}. The main difference is the proof of the properness of the respective induced metric map at the blow-up, which requires additional tools. 

Our work is divided into two parts, corresponding to the proofs of Theorems~\ref{main} and~\ref{main2}.
In the first part, we construct the blow-ups and prove Lemma~\ref{ml} in Section~\ref{sec:blowups}. In Section~\ref{sec:proof} we deduce Theorem~\ref{main} from the main lemmas. In Section~\ref{sec:properness} we obtain Lemma~\ref{ml2}. In the second part, in Section~\ref{sec:changes} we describe the necessary changes in the setup for the proof of Theorem~\ref{main2}. In Section~\ref{sec:prop2} we establish the required properness result in this context. We finish the paper with two appendices containing some results that we use, which might be of independent interest. In Appendix~\ref{sec:intmet} we study intrinsic metrics of general convex surfaces in $\A^3$ and derive results on their convergence. In Appendix~\ref{sec:busfel} we obtain a Busemann--Feller-type lemma for GHMC anti-de Sitter (2+1)-spacetimes.

\subsection{Context}

In~\cite{Wey} Weyl asked whether any smooth Riemannian metric of positive curvature can be realized as the intrinsic metric of the boundary of a unique smooth convex body in Euclidean 3-space $\E^3$. This problem has two parts: the realization part and the rigidity part (the uniqueness). Weyl formulated a version of continuity method and implemented a part of it. Several geometers made contributions in its further developments, culminating in a positive resolution of the realization part by Nirenberg~\cite{Nir}. The rigidity part is due to Cohn-Vossen~\cite{CV} in the analytic class and to Herglotz~\cite{Her} in the smooth class.  

In~\cite{Ale} Alexandrov formulated and proved a polyhedral version of the Weyl problem. Furthermore, in the same paper he searched for a common generalization of the smooth and polyhedral cases. This led him to develop the notion of what now is known as \emph{Alexandrov space} and what now belongs to one of the cornerstones of modern geometry. In~\cite{Ale} Alexandrov proved the realization part of the problem in this generalized context. The rigidity part for general convex bodies was later supplied by Pogorelov in~\cite{Pog}. Note that from the combination of works of Alexandrov and Pogorelov another solution to the original smooth Weyl problem follows. 

The Alexandrov--Weyl problem was generalized in multiple directions. The one that is the most relevant to our paper is the direction of hyperbolic geometry. It was observed already by Alexandrov that his proofs from~\cite{Ale} generalize directly to convex bodies in hyperbolic 3-space $\H^3$. Pogorelov in~\cite{Pog} developed a method to prove the rigidity for general convex bodies in $\H^3$. Their works also lead to a resolution of the smooth problem in $\H^3$. Curiously, a direct proof of either realization or rigidity part for smooth convex bodies in $\H^3$ is not known. 

The hyperbolic version of the problem has an interesting further generalization that the Euclidean version lacks. Convex bodies in either $\E^3$ or $\H^3$ have the trivial topology of the 3-ball and all the original works use this significantly (namely, that the boundary has positive Euler characteristic). Starting from the 70s, the geometrization program of Thurston~\cite{Thu} implied that in some sense ``most'' compact 3-manifolds are hyperbolic, which generated a lot of interest in the study of hyperbolic 3-manifolds. In particular, ``most'' compact 3-manifolds with non-empty boundary can be given a hyperbolic metric with convex boundary. It is natural to formulate an analogue of the Alexandrov--Weyl problem for such manifolds. It has required a considerable development of the existing techniques to deal with the case of non-trivial topology. The smooth realization part was proven by Labourie in~\cite{Lab}, the smooth rigidity part was established by Schlenker in~\cite{Sch}. The polyhedral counterpart was obtained by Prosanov in~\cite{Pro2}, though the rigidity was shown only under a genericity assumption. The case of general convex boundary is open, though some progress was obtained in~\cite{Slu, Pro}. 

GHMC anti-de Sitter (2+1)-spacetimes exhibit properties similar to the so-called \emph{quasi-Fuchsian} hyperbolic 3-manifolds; see, e.g., the paper~\cite{KS} of Krasnov--Schlenker for a simultaneous treatment. As it was mentioned in Section~\ref{sec:statement}, there are two ways to formulate an Alexandrov--Weyl problem for them. In the first one, we prescribe half the holonomy and the intrinsic metric of a convex Cauchy surface. The realization part of the smooth version  was obtained by Tamburelli in~\cite{Tam} and the rigidity part was established by Chen--Schlenker in~\cite{CS}. It is interesting to note that the methods of the present paper, dealing with the polyhedral case, are very different from the techniques of the mentioned articles. We note that polyhedral surfaces in GHMC anti-de Sitter (2+1)-spacetimes have a curious interpretation via flippable tilings on hyperbolic surfaces, see~\cite{FS6} by Fillastre--Schlenker, which are polyhedral analogues of the earthquakes. We, however, do not pursue this perspective.

Mess classified in~\cite{Mes} the GHMC anti-de Sitter and Minkowski (2+1)-spacetimes. The classification of the de Sitter ones was finished by Scannell in~\cite{Sca}. Our type of the Alexandrov--Weyl problem can be formulated for all of them. Interestingly, each case requires different techniques. The smooth version for Minkowski spacetimes was obtained by Trapani--Valli in~\cite{TV}. As we already mentioned, the polyhedral version was established by Fillastre--Prosanov in~\cite{FP}. The problem remains open for de Sitter spacetimes.

Alternatively to prescribing a left/right part of the holonomy, one can drop half of the holonomy by restricting themselves to the diagonal of $\mc T \times \mc T$. In such case a smooth problem was resolved by Labourie--Schlenker in~\cite{LS} and a polyhedral one was done by Fillastre in~\cite{Fil}.

The second version of the Alexandrov--Weyl problem for GHMC (2+1)-spacetimes asks to prescribe the intrinsic metrics of two Cauchy surfaces that are convex in the opposite directions. For anti-de Sitter spacetimes, in the smooth case the realization part was shown by Tamburelli in~\cite{Tam}. For the rigidity part, except the mentioned work~\cite{BMS} of Bonsante--Mondello--Schlenker, a partial progress was obtained by Prosanov--Schlenker in~\cite{PS}, both for smooth and polyhedral surfaces. Namely, in~\cite{PS} the authors show that there exist an open subset $U$ in the configuration space of the triples $(\Omega, \Sigma^+, \Sigma^-)$ such that all triples in $U$ are uniquely determined inside $U$ by the intrinsic metrics of $\Sigma^\pm$. However, a priori, these intrinsic metrics still might be realized by triples outside of $U$. This is the main difference in flavor between the results of~\cite{PS} and our rigidity results: Here we establish the global rigidity for subsets of cone-metrics.

For Minkowski and de Sitter spacetimes such a version requires embedding the surfaces in two different spacetimes that have the same holonomy, one spacetime is future-complete and one is past-complete. For Minkowski spacetimes this was proven in the smooth case in~\cite{Smi} by Smith and in the polyhedral case in~\cite{FP} by Fillastre--Prosanov. For de Sitter spacetimes one needs to restrict to the quasi-Fuchsian ones. In this setting the smooth case was obtained by Schlenker in~\cite{Sch} and the polyhedral case by Prosanov in~\cite{Pro3}.

There is a lot of other research related to the Alexandrov--Weyl problem that we are not able to mention here. However, there is one topic that we need to include. This is the question whether quasi-Fuchsian hyperbolic 3-manifolds or GHMC anti-de Sitter (2+1)-spacetimes are determined by the bending laminations of their convex cores (the definition can be found in Section~\ref{sec:bentsurf}) and which measured laminations arise as such. In the hyperbolic case the realization part was obtained by Bonahon--Otal in~\cite{BO} and the rigidity part was proven in~\cite{DS} by Dular--Schlenker. There was a previous progress on the local rigidity for small bending laminations due to Bonahon in~\cite{Bon3}. In~\cite{Ser} Series proved a compactness result for small laminations, establishing their global rigidity. In the anti-de Sitter case the paper~\cite{BS3} by Bonsante--Schlenker showed the realization part and showed the global rigidity of small laminations. This needs to be mentioned because the papers~\cite{Bon3} and~\cite{BS3} are the main inspirations for our techniques here. They introduce the blow-up on the respective deformation spaces and use it to obtain rigidity results. However, their constructions of blow-up are different from ours and do not require geometric transitions. Instead they rely on a study of the infinitesimal geometry of the Teichm\"uller space. The compactness results for small laminations from~\cite{Ser} and~\cite{BS3} are the main prototypes for the compactness result for small metrics required for Theorem~\ref{main2} and established in Section~\ref{sec:prop2}. 

For some other aspects of the Alexandrov--Weyl-type problems, we refer to the recent surveys~\cite{Sch2} and~\cite{Sch3} of Schlenker.

\vskip+0.2cm

\textbf{Acknowledgments.} I am very grateful to Fran\c{c}ois Fillastre, Jean-Marc Schlenker, Andrea Seppi and Graham Smith for useful discussions. This work was done at three universities: the University of Vienna, TU Wien and the Autonomous University of Barcelona. I am very thankful to all of them for comfortable working conditions. I am very grateful to Michael Eichmair, Ivan Izmestiev and Joan Porti for their support.

This research was funded in whole by the Austrian Science Fund (FWF)\\ https://doi.org/10.55776/ESP12 and https://doi.org/10.55776/J4955. For open access purposes, the author has applied a CC BY public copyright license to any author-accepted manuscript version arising
from this submission.

\vskip+0.8cm
\begin{center}
{\LARGE \textit{Part 1}}
\end{center}

The goal of this part is to prove Theorem~\ref{main}.

\section{Notation}

In this section we give a brief list of the notation that we use the most in Part 1. We do not give the definitions here, since they are given in the respective sections. For the whole manuscript $S$ is a closed oriented surface of genus $k\geq 2$, $V \subset S$ is a finite set of size $n\geq 1$, $G:={\rm PSL(2,\R)}$, $\rho_\circ: \pi_1S \ra G$ is a fixed Fuchsian representation. 
\begin{itemize}
\item $\tilde V \subset \tilde S$ is the preimage of $V$ in $\tilde S$ equipped with the $\pi_1S$-action; we fix some lift $V \ra \tilde V$, which we call canonical;

\item $\mc D_0=\mc D_0(V)$ is the space of Euclidean cone-metrics on $(S, V)$; 

\item $\mc D_0^c=\mc D_0^c(V) \subset \mc D_0$ is the subset of concave Euclidean cone-metrics;

\item $\mc D_0^s=\mc D_0^s(V) \subset \mc D_0^c$ is the subset of concave Euclidean cone-metrics with $V(d)=V$;

\item $\mc D_-=\mc D_-(V)$ is the space of hyperbolic cone-metrics on $(S, V)$;

\item $\mc D_-^c=\mc D_-^c(V) \subset \mc D_-$ is the subset of concave hyperbolic cone-metrics;

\item $\mc D_-^s=\mc D_-^s(V) \subset \mc D_-^c$ is the subset of concave hyperbolic cone-metrics with $V(d)=V$;

\item $\mc D_\vee^c=\mc D_\vee^c(V)$ is the blow-up of the space of concave hyperbolic cone-metrics;

\item $G_0 \cong G \ltimes \R^{2,1}$ and $G_- \cong G \times G$ are the identity components of the isometry groups of Minkowski 3-space and anti-de Sitter 3-space;

\item $\mc R$ is the Teichm\"uller component of the representation variety of $\pi_1S$ in $G$;

\item $\mc T$ is the Teichm\"uller space of $S$, which we mostly consider as the $G$-quotient of $\mc R$; note that we will have the same standard notation $\rho$ for elements of both $\mc R$ and $\mc T$; similarly, we will use $\tau$ for an element of both $T_{\rho_\circ}\mc R$ and $T_{\rho_\circ}\mc T$; the meaning should be each time clear from the context;

\item for $\tau \in T_{\rho_\circ}\mc R$, $\iota_\tau=\iota_{\rho_\circ,\tau}: \pi_1S \ra G_0$ is the representation obtained by twisting $\rho_\circ$ by $\tau$;

\item for $\rho \in \mc R$, $\theta_\rho=\theta_{\rho_\circ, \rho}: \pi_1S \ra G_-$ is the representation $(\rho_\circ, \rho)$;

\item for $\tau \in T_{\rho_\circ}\mc R$, $\tilde\Omega_\tau^+=\tilde\Omega_{\iota_\tau}^+ \subset \R^{2,1}$ is the future-complete domain of discontinuity of $\iota_\tau$;

\item for $\tau \in T_{\rho_\circ}\mc T$, $\Omega_\tau^+=\Omega_{\iota_\tau}^+$ is the $\iota_\tau$-quotient of $\tilde\Omega_\tau$; 

\item for $\rho \in \mc R$, $\tilde\Omega_\rho=\tilde\Omega_{\theta_\rho} \subset \A^3$ is the domain of discontinuity of $\theta_\rho$; $\tilde C_\rho=\tilde C_{\theta_\rho}$ is its convex core; $\tilde\Omega^\pm_\rho=\tilde\Omega^\pm_{\theta_\rho}$ are the future-convex and past-convex components of $\tilde\Omega_\rho\backslash\tilde C_\rho$;

\item for $\rho \in \mc R$, $\Lambda_\rho=\Lambda_{\theta_\rho} \subset \pt\A^3$ is the limit set of $\theta_\rho$;

\item for $\rho \in \mc T$, $\Omega_\rho=\Omega_{\theta_\rho}$ is the quotient of $\tilde\Omega_\rho$; $C_\rho=C_{\theta_\rho}$ is its convex core; $\Omega^\pm_\rho=\Omega^\pm_{\theta_\rho}$ are the future-convex and past-convex components of $\Omega_\rho\backslash C_\rho$;

\item $\tilde{\mc P}_0=\tilde{\mc P}_0(\rho_\circ, V):=T_{\rho_\circ}\mc R \times (\R^{2,1})^V$; we denote its elements by $(\tau, \tilde f)$ where $\tilde f: V \ra \R^{2,1}$; it will be helpful to consider $\tilde f$ as a $\pi_1S$-equivariant map $\tilde f: \tilde V \ra \R^{2,1}$ using the canonical lift $V \rightarrow \tilde V$;

\item $\tilde{\mc P}_0^w=\tilde{\mc P}_0^w(\rho_\circ, V) \subset \tilde{\mc P}_0$ is the subset of $(\tau, \tilde f)$ where $\tilde f(\tilde V) \subset \tilde\Omega_\tau^+$;

\item $\tilde{\mc P}_0^c=\tilde{\mc P}_0^c(\rho_\circ, V) \subset \tilde{\mc P}^w_0$ is the subset of $(\tau, \tilde f)$ in a convex position;

\item $\tilde{\mc P}_0^s=\tilde{\mc P}_0^s(\rho_\circ, V) \subset \tilde{\mc P}^s_0$ is the subset of $(\tau, \tilde f)$ in a strictly convex position;

\item $\mc P_0^s \subset \mc P_0^c \subset \mc P_0^w$ are the quotients of the respective spaces above by the $G_0$- and $\pi_1S$-actions; their elements are $(\tau, f)$, $\tau \in T_{\rho_\circ}\mc T$, $f: V \ra \Omega_\tau$;

\item $\tilde{\mc P}_-=\tilde{\mc P}_-(\rho_\circ, V):=\mc R \times (\A^3)^V$; we denote its elements by $(\rho, \tilde f)$ where $\tilde f: V \ra \A^3$; it will be helpful to consider $\tilde f$ as a $\pi_1S$-equivariant map $\tilde f: \tilde V \ra \A^3$ using the canonical lift $V \rightarrow \tilde V$;

\item $\tilde{\mc P}_-^w=\tilde{\mc P}_-^w(\rho_\circ, V) \subset \tilde{\mc P}_-$ is the subset of $(\rho, \tilde f)$ where $\tilde f(\tilde V) \subset \tilde\Omega_\rho^+$;

\item $\tilde{\mc P}_-^c=\tilde{\mc P}_-^c(\rho_\circ, V) \subset \tilde{\mc P}^w_-$ is the subset of $(\rho, \tilde f)$ in a convex position;

\item $\tilde{\mc P}_-^s=\tilde{\mc P}_-^s(\rho_\circ, V) \subset \tilde{\mc P}^s_-$ is the subset of $(\rho, \tilde f)$ in a strictly convex position;

\item $\mc P_-^s \subset \mc P_-^c \subset \mc P_-^w$ are the quotients of the respective spaces above by the $G_-$- and $\pi_1S$-actions; their elements are $(\rho, f)$, $\rho \in \mc T$, $f: V \ra \Omega_\rho$;

\item $\mc P_\vee^c=\mc P_\vee^c(\rho_\circ, V)$ is the blow-up of the space $\mc P_-^c$;

\item $\mc I_0: \mc P_0^c \ra \mc D_0^c$ is the Minkowski intrinsic metric map;

\item $\mc I_-: \mc P_-^c \ra \mc D_-^c$ is the anti-de Sitter intrinsic metric map;

\item $\mc I_\vee: \mc P_\vee^c \ra \mc D_\vee^c$ is the blow-up of $\mc I_-$.
\end{itemize}


\section{Construction of blow-ups}
\label{sec:blowups}

\subsection{Cones and blow-ups}
\label{sec:cone}

Consider an $m$-dimensional real vector space $X$, denote its origin by $o$ and the space of rays from $o$ by $\S(X)$, called the \emph{spherization} of $X$. We have a projection $\sigma: (X \backslash o) \ra \S(X)$. A subset $C \subset (X \backslash o)$ is a \emph{cone} if for every $x \in C$ and $t \in \R_{> 0}$, we have $t x \in C$. Note that we require the origin to not belong to $C$ and do not require $C$ to be convex. Denote $\sigma(C)$ by $\S(C)$. A map $\phi: C^1 \ra C^2$ between cones is \emph{coned} if $\phi(t x)=t \phi(x)$ for all $x \in C^1$, $t \in \R_{>0}$. For $A \subset X$, a cone spanned by $A$ is the smallest cone containing $A\backslash o$.

For an open cone $C$ let $\kappa: \S(C) \ra C$ be a smooth section of $\sigma$. We define the \emph{blow-up} $C_\vee$ of $C$ by
\[C_\vee:=\{(x,t): x \in {\rm im}(\kappa), t \in [0, +\infty)\}\subset C \times [0, +\infty)\]
endowed with the induced topology. Clearly, its topology is independent on $\kappa$. There are natural identifications $\inter(C_\vee) \cong C$ and $\pt C_\vee \cong \S(C)$, which we will often use implicitly. We are now interested in a criterion, when a smooth map $\phi: C^1 \ra C^2$, extending continuously to send the origin to the origin, extends to a smooth map $\phi_\vee: C_\vee^1 \ra C_\vee^2$, which we then call the \emph{blow-up} of $\phi$. First we need a technical lemma.

\begin{lm}
\label{techn2}
Let $f$ and $g$ be smooth functions on $A \times [0, \e)$, where $A$ is a domain in $\R^m$. By $f^{(k)}, g^{(k)}: A \times [0, \e) \ra \R$ we denote the $k$-th derivatives in the last variable. Suppose that there is $k \geq 0$ such that for all $x \in A$ we have $f(x,0)=g(x,0)=\ldots=f^{(k)}(x, 0)=g^{(k)}(x,0)=0$, but $g^{(k+1)}(x, 0) \neq 0$. Furthermore, assume that $g\neq 0$ on $A \times (0, \e)$. Then the function $f/g$ admits a smooth extension to $A \times [0, \e)$.
\end{lm}

This follows from the Taylor theorem. Now we can establish our criterion.

\begin{lm}
\label{smooth}
Let $\phi: C^1 \ra C^2$ be a smooth map, extending continuously to send the origin of $C^1$ to the origin of $C^2$. Let the map $\grave \phi:(x, t) \mapsto \phi(tx)$ be smooth on $C^1 \times [0, +\infty)$. Then $\phi$ admits a smooth extension $\phi_\vee: C_\vee^1 \ra C_\vee^2$.
\end{lm}

\begin{proof}
With the help of some map $\kappa: \S(C^1) \ra C^1$ consider $C_\vee^1$ as a subset of $C^1 \times [0, +\infty)$. Define a map $\psi: C^1 \times (0, +\infty) \ra C^2 \times (0, +\infty)$ by
\[\psi(x, t):=\left(\frac{\phi(tx)}{t}, t\right).\]
Due to Lemma~\ref{techn2}, it admits a continuous extension to a smooth map $\psi: C^1 \times [0, +\infty) \ra C^2 \times [0, +\infty)$. There is a natural submersion $\chi: C^2 \times [0, +\infty) \ra C_\vee^2$: we send $(x, t)$ to $tx$ when $t>0$ and send $(x, 0)$ to $\sigma(x)\in\S(C^2) \cong \pt C_\vee^2$. Clearly, the restriction of $\chi\circ\psi$ to $C^1 \subset C^1_\vee \subset C^1 \times [0,+\infty)$ is $\phi$. Its restriction to $C^1_\vee$ gives the desired extension of $\phi$. 
\end{proof}

We remark that here we used only a simple partial case of Lemma~\ref{techn2}, but we will use Lemma~\ref{techn2} again in a slightly deeper situation.

A \emph{coned manifold} is a manifold with an atlas with charts in cones, and coned transition maps. A coned manifold has a natural smooth action of $\R_{>0}$. To a coned manifold one can associate its \emph{spherization} $\S(M)$. We naturally define coned maps between coned manifolds. For a coned map $\phi: M_1 \ra M_2$, $\S(\phi)$ is the respective map $\S(M_1) \ra \S(M_2)$. For a smooth map $\phi: M_1 \ra M_2$ sometimes we can define its blow-up $\phi_\vee$.


A subset $C \subset (X \backslash o)$ is a \emph{lower cone}, if for every $x \in C$ and $t \in \R_{> 0}$, $t\leq 1$, we have $t x \in C$. Every lower cone $C$ spans a cone, which can be then used to define the blow-up $C_\vee$ of $C$. For $A \subset X$ a lower cone spanned by $A$ is the smallest lower cone containing $A \backslash o$.


%

\subsection{The spaces of cone-metrics}
\label{sec:metrdef}

We recall some basic notions from the respective sections of~\cite{Pro3, Pro2, FP}. Let $\mc H=\mc H(V)$ be the group of self-homeomorphisms of $S$ fixing $V$ and isotopic to the identity. Let $\mc H^\sharp=\mc H^\sharp(V)$ be its normal subgroup of the ones that are isotopic to the identity relative to $V$. Define $\mc B=\mc B(V):=\mc H/\mc H^\sharp$, the \emph{pure braid group} of $(S, V)$. 

A \emph{triangulation} $\ms T$ of $(S, V)$ is a collection of simple disjoint arcs with endpoints in $V$ that cut $S$ into triangles. Two triangulations are \emph{equivalent} (resp. \emph{weakly equivalent}) if they differ by $h \in \mc H^\sharp$ (resp. by $h \in \mc H$). The set of edges of a triangulation $\ms T$ is denoted by $E(\ms T)$. 

A \emph{hyperbolic cone-metric} (resp. \emph{Euclidean cone-metric}) on $(S, V)$ is locally isometric to the hyperbolic plane $\H^2$ (resp. the Euclidean plane $\E^2$) except possibly at some points of $V$, where it is locally isometric to a hyperbolic cone (resp. a Euclidean cone). We say a \emph{cone-metric} meaning either hyperbolic or Euclidean cone-metric. The set of cone-points of a cone-metric $d$, i.e., those where the cone-angle is $\neq 2\pi$, is denoted by $V(d) \subseteq V$. The \emph{curvature} $\kappa_d(v)$ of $v \in V$ in $d$ is $2\pi$ minus the cone-angle of $v$ in $d$. We call $d$ \emph{concave} if $\kappa_d \in (-\infty, 0]^V$. For a triangulation $\ms T$ of $(S, V)$ we say that a cone-metric $d$ is \emph{$\ms T$-triangulable} (resp. \emph{weakly $\ms T$-triangulable}) if there exists a triangulation equivalent (resp. weakly equivalent) to $\ms T$ that is geodesic in $d$. 

Let $\mc D_-=\mc D_-(V)$ be the set of hyperbolic cone-metrics on $(S, V)$ up to isometries belonging to $\mc H$. Let $\mc D_-^\sharp=\mc D_-^\sharp(V)$ be the set of hyperbolic cone-metrics on $(S, V)$ up to isometries belonging to $\mc H^\sharp$. We define similarly the sets $\mc D_0$, $\mc D_0^\sharp$. The group $\mc B$ acts on $\mc D_-^\sharp$, $\mc D_0^\sharp$ and the sets $\mc D_-$, $\mc D_0$ are the respective quotients. For a cone-metric $d$ and a triangulation $\ms T$ the property that $d$ is $\ms T$-triangulable (resp. weakly $\mc T$-triangulable) holds by definition for its class in $\mc D_-^\sharp$ or $\mc D_0^\sharp$ (resp. in $\mc D_-$ or $\mc D_0$). If $d$ is $\ms T$-triangulable, then a respective geodesic triangulation is unique. However, this does not hold for weak equivalence. This technical nuance is the main reason why we sometimes invoke $\mc D_-^\sharp$ and $\mc D_0^\sharp$ in this paper, as we will be mostly using $\mc D_-$ and $\mc D_0$. We note that every cone-metric is $\ms T$-triangulable for some $\ms T$, see, e.g.,~\cite{Izm}. 

For a triangulation $\ms T$ let $\mc D_-^\sharp(\ms T)$ and $\mc D_0^\sharp(\ms T)$ be the subsets of $\ms T$-triangulable cone-metrics. There are the edge-lengths charts $\phi^\ms T_-:\mc D_-^\sharp(\ms T) \ra \R^{E(\ms T)}$, $\phi^\ms T_0:\mc D_0^\sharp(\ms T) \ra \R^{E(\ms T)}$ sending $\mc D_-^\sharp(\ms T)$ and $\mc D_0^\sharp(\ms T)$ injectively onto the open polyhedral cone in $\R^{E(\ms T)}$ defined by the triangle inequalities. These charts endow $\mc D_-^\sharp$ and $\mc D_0^\sharp$ with structures of smooth manifolds of dimension $3(n-k)$, where $n=|V|$ and $k$ is the genus of $S$. Furthermore, for $\mc D_0^\sharp$ the transition maps are coned and endow $\mc D_0^\sharp$ with the structure of a coned manifold. In the case of $\mc D_-$ the intersections of charts are not subcones and the transition maps are not coned, so a coned structure and a spherization are not defined for $\mc D_-^\sharp$.

Let $\mc C=\mc C(V)$ be the set of conformal structures on $S$ up to conformal maps belonging to $\mc H^\sharp$. 
It has a natural topology of a smooth manifold of dimension $2n-3k$, see~\cite{FM}.
Every cone-metric defines a conformal structure on $S$. Consider the map $\mc U_-: \mc D_-^\sharp \ra \mc C \times \R^V$, which sends a cone-metric $d$ to the respective conformal structure and to the tuple $\kappa_d$ of the curvatures of $V$. The work of Troyanov~\cite{Tro2} implies that $\mc U_-$ is a diffeomorphism onto the domain that is defined by the conditions $\kappa_d \in (-\infty, 2\pi)^V$ and $\sum_{v \in V}\kappa_d(v)<2\pi(2-2k)$. The group $\mc B$ acts on $\mc D_-^\sharp$ equivariantly with respect to $\mc U_-$, where on the image it acts on $\mc C$ and preserves $\kappa_d$. As a subgroup of the pure mapping class group of $(S, V)$, it acts properly discontinuously on $\mc C$, see~\cite{FM}. By definition, it also acts freely. Hence, it endows $\mc D_-$ with the structure of a smooth manifold of dimension $3(n-k)$ so that $\mc D_-^\sharp \ra \mc D_-$ is a covering (actually, a universal covering). For $\mc D_0^\sharp$ we have a map $\mc U_0: \mc D_0^\sharp \ra \mc C \times \R^V \times \R_{>0}$. The last component here stands for the area of a cone-metric. The work of Troyanov~\cite{Tro} implies that $\mc D_0^\sharp$ is sent diffeomorphically onto the hypersurface defined by the conditions $\kappa_d \in (-\infty, 2\pi)^V$ and $\sum_{v \in V}\kappa_d(v)=2\pi(2-2k)$. The group $\mc B$ again acts on $\mc D_0$ equivariantly with respect to $\mc U_0$, where on the image it acts on $\mc C$ and preserves $\kappa_d$ and the area. Hence, it also acts freely and properly discontinuously. The coned structure of $\mc D_0$ is expressed in the scaling of the area. Thereby, $\mc D_0$ is endowed with a structure of a coned manifold of dimension $3(n-k)$.

We denote by $\mc D_{-}^s \subset \mc D_-$ and $\mc D_{0}^s \subset \mc D_0$ the (open) subsets of concave cone-metrics $d$ with $V(d)=V$. We denote by $\mc D_{-}^c \subset \mc D_-$ and $\mc D_{0}^c \subset \mc D_0$ the (closed) subsets of concave cone-metrics. 

\subsection{Blow-up of the space of cone-metrics}
\label{sec:metr}


Pick a triangulation $\ms T$ of $(S, V)$ and the charts $\phi^\ms T_-:\mc D_-^\sharp(\ms T) \ra \R^{E(\ms T)}$, $\phi^\ms T_0:\mc D_0^\sharp(\ms T) \ra \R^{E(\ms T)}$ from Section~\ref{sec:metrdef}.
The images of both $\phi^\ms T_0$, $\phi^\ms T_-$ are the same open cone in $\R^{E(\ms T)}$, which we denote by $\Phi^{\ms T}$. Now define $\mc D_\vee^\sharp(\ms T):=\mc D_-^\sharp(\ms T)\cup\S(\mc D_0^\sharp(\ms T))$ and define a map $\phi^\ms T_\vee: \mc D_\vee^\sharp(\ms T) \ra  \Phi^\ms T_\vee$, which coincides with $\phi^\ms T_-$ on $\mc D_-^\sharp(\ms T)$ and with $\S(\phi^\ms T_0)$ on $\S(\mc D_0^\sharp(\ms T))$.
We claim

\begin{lm}
\label{metrics}
The maps $\phi^\ms T_\vee$ equip $\mc D_-^\sharp\cup\S(\mc D_0^\sharp)$ with a topology of a smooth manifold with boundary. 
\end{lm}
We denote this manifold by $\mc D_\vee^\sharp$. We will need

\begin{lm}
\label{techn1}
Let $ABC$ and $A'B'C'$ be two hyperbolic triangles, whose respective side-length differ by $t$, $0<t<1$, so that $A'B'C'$ is smaller. Then the respective angles of $A'B'C'$ are strictly bigger than the respective angles of $ABC$.
\end{lm}

This was shown in~\cite[Lemma 2.3.9]{Pro}.

\begin{proof}[Proof of Lemma~\ref{metrics}.]
Let us see what we need to do. The system of charts $\phi^\ms T_\vee$ is an atlas on $\mc D_\vee^\sharp$ with values in smooth manifolds with boundary $\Phi^\ms T_\vee$. We need to understand the behavior of the transition maps around the boundary points. Let $\ms T$ and $\ms T'$ be two triangulations. Denote by $\mc D_0^\sharp(\ms T, \ms T')$, $\mc D_-^\sharp(\ms T, \ms T')$ the sets of cone-metrics that are both $\ms T$- and $\ms T'$-triangulable, denote $\mc D_-^\sharp(\ms T, \ms T')\cup\S(\mc D_0^\sharp(\ms T, \ms T'))$ by $\mc D_\vee^\sharp(\ms T, \ms T')$, and define $\Phi^{\ms T, \ms T'}_\vee:=\phi^\ms T_\vee(\mc D_\vee^\sharp(\ms T, \ms T'))$. The first thing is to check that $\Phi^{\ms T, \ms T'}_\vee$ is a submanifold with boundary of $\Phi^{\ms T}_\vee$, i.e. that for every point of $\Phi^{\ms T, \ms T'}_\vee \cap \pt \Phi^{\ms T}_\vee$, its neighborhood in $\Phi^{\ms T}_\vee$ belongs to $\Phi^{\ms T, \ms T'}_\vee$. Second, one needs to show that the transition map $\phi^{\ms T, \ms T'}_\vee: \Phi^{\ms T, \ms T'}_\vee \ra \Phi^{\ms T', \ms T}_\vee$ is smooth at the boundary points. It is enough to consider the case of $\ms T$ and $\ms T'$ that differ by a flip in one quadrilateral $Q$. Note that $d \in \mc D_0^\sharp(\ms T, \ms T')$ or $d \in \mc D_-^\sharp(\ms T, \ms T')$ if and only if $Q$ is strictly convex in $d$. 

For the moment we identify $\Phi^{\ms T}_\vee$ with $\mc D_\vee^\sharp(\ms T)$. Pick an arbitrary triangle of $\ms T$ with the edge-lengths $a$, $b$ and $c$ and the angle $\alpha$ opposite to the $a$-edge, which are all considered as functions on $\Phi^{\ms T}$. Recall the hyperbolic cosine law
\[\cos \alpha=\frac{\cosh b\cosh c-\cosh a}{\sinh b\sinh c}.\]
Consider the function $\grave\alpha$ on $\Phi^{\ms T} \times (0, +\infty)$ defined by $\grave\alpha(x,t):=\alpha(tx)$. By Lemma~\ref{techn2}, it extends smoothly to $\Phi^{\ms T} \times [0, +\infty)$. Hence, as in the proof of Lemma~\ref{smooth}, $\alpha$ extends to a smooth function on $\Phi^{\ms T}_\vee$, which we continue to denote by $\alpha$. An easy computation shows that on $\pt\Phi^{\ms T}_\vee$ it is equal to the respective angle in the respective class of metrics in $\S(\mc D_0^\sharp(\ms T))$.

The set $\Phi^{\ms T, \ms T'}_\vee \cap \pt \Phi^{\ms T}_\vee$ is an open subset of $\pt \Phi^{\ms T}_\vee$. Pick $x \in \Phi^{\ms T, \ms T'}_\vee \cap \pt \Phi^{\ms T}_\vee$. Consider the foliation of $\Phi^{\ms T}_\vee$ by the curves of the form $tx$, $t \in \R_{\geq 0}$. Observe that from Lemma~\ref{techn1}, every angle of every triangle of $\ms T$ is strictly monotonous along every such curve. Then all the curves that start sufficiently close to $x$ belong to $\Phi^{\ms T, \ms T'}_\vee$. On the other hand, clearly they span a neighborhood of $x$ in $\Phi^{\ms T}_\vee$. This shows that $\Phi^{\ms T, \ms T'}_\vee$ is a submanifold with boundary. Denote by $\breve\Phi^{\ms T, \ms T'}_\vee$ the subset of $\Phi^{\ms T, \ms T'}_\vee$ spanned by these curves. Its interior corresponds to an open cone $\breve\Phi^{\ms T, \ms T'}$ in $\R^{E(\ms T)}$.

Now we  need to verify the smoothness of the transition map. It is enough to check it on $\breve\Phi^{\ms T, \ms T'}_\vee$. Return to the quadrilateral $Q$. Denote the edge-lengths of $Q$ by $a$, $b$, $c$ and $d$ in this order, and the lengths of the diagonals by $e$ and $f$, where $e$ is the length of an edge of $\ms T$ and this edge passes in the corner between the edges of lengths $a$ and $d$, and $f$ is the length of an edge of $\ms T'$. Let $\alpha$ be the angle between the edges of lengths $a$ and $d$, and let it be split by the $e$-diagonal into the angles $\alpha_1$ and $\alpha_2$; $\alpha_1$ is adjacent to the $a$-edge. We consider all these as functions on $\breve\Phi^{\ms T, \ms T'}$. We have
\[\cosh f =\cosh a\cosh d-\sinh a \sinh d \cos \alpha , \]
\[\alpha=\alpha_1+\alpha_2.\]

Let $\grave \alpha_1$ and $\grave \alpha_2$ be the respective functions on $\breve\Phi^{\ms T, \ms T'} \times (0,+\infty)$. Since they extend to smooth functions on $\breve\Phi^{\ms T, \ms T'} \times [0,+\infty)$, it follows that the respective function $\grave f$ also extends to a smooth function on $\breve\Phi^{\ms T, \ms T'} \times [0,+\infty)$. Denote the restriction of $\phi^{\ms T, \ms T'}_\vee$ to $\breve\Phi^{\ms T, \ms T'}$ by $\breve\phi^{\ms T, \ms T'}$. By Lemma~\ref{smooth}, it has a smooth blow-up $\breve\phi^{\ms T, \ms T'}_\vee$ on $\breve\Phi^{\ms T, \ms T'}_\vee$. An easy computation shows that on $\pt \breve\Phi^{\ms T, \ms T'}_\vee\cong \S(\mc D_0^\sharp(\ms T, \ms T'))$ the map $\breve\phi^{\ms T, \ms T'}_\vee$ coincides with $\phi^{\ms T, \ms T'}_\vee$. The proof is finished.

\end{proof}

The group $\mc B$ acts on both spaces $\mc D_-^\sharp$ and $\S(\mc D_0^\sharp)$ by diffeomorphisms, freely and properly discontinuously. We need to see that these actions combine into a properly discontinuously action on $\mc D_\vee^\sharp$ by diffeomorphisms. Pick $h \in \mc H\backslash \mc H^\sharp$. By definition, for a triangulation $\ms T$ the triangulation $\ms T':=h_*(\ms T)$ is not equivalent to $\ms T$. The class of $h$ in $\mc B$ induces a diffeomorphism $\R^{E(\ms T)} \ra \R^{E(\ms T')}$, which sends $\Phi^\ms T$ onto $\Phi^{\ms T'}$. From this and from our construction, it is clear that the class of $h$ induces a self-diffeomorphism of $\mc D_\vee^\sharp$ and that the proper discontinuity of the actions of $\mc B$ on $\mc D_-^\sharp$ and $\S(\mc D_0^\sharp)$ implies the proper discontinuity of the action on $\mc D_\vee^\sharp$. We denote the obtained quotient by $\mc D_\vee$.

From the viewpoint of the Troyanov parameterization, $\mc D_\vee^\sharp$ is diffeomorphic to the domain in $\mc C \times \R^V$ given by $\kappa_d \in (-\infty, 2\pi)$ and $\sum_v \kappa_d(v)\leq 2\pi(2-2k)$. However, we will not pursue this viewpoint.

Define $\mc D_{\vee}^s:=\mc D_{-}^s \cup \S(\mc D_{0}^s)\subset \mc D_\vee$. Lemma~\ref{techn1} implies that $\mc D_{\vee}^s$ is a submanifold with boundary. Define also  $\mc D_{\vee}^c:=\mc D_{-}^c \cup \S(\mc D_{0}^c)\subset \mc D_\vee$.

\subsection{Projective geometry of spacetimes of constant curvature}

\subsubsection{Projective geometries and geometric transition}
\label{sec:geometries}

We will employ the theory of geometric transition from Minkowski and co-Minkowski geometries to anti-de Sitter geometry. The geometric transition from co-Minkowski geometry to anti-de Sitter geometry was introduced and popularized by Danciger~\cite{Dan}, with co-Minkowski geometry having the name half-pipe geometry in~~\cite{Dan}. For us, the main role will be played by the transition from Minkowski to anti-de Sitter, though we will also need to employ the one from co-Minkowski. We will rely on the exposition of Fillastre--Seppi~\cite{FS5}.

%
%
%

We work with $\R^4$ as well as with its projectivization $\RP^3$. Consider on $\R^4$ the quadratic form 
\[q(x):=x_1^2+x_2^2-x_3^2-x_4^2\]
and let $b$ be the associated bilinear form.
Define anti-de Sitter 3-space, $\A^3$, as the projective quotient of the quadric
\[\{x \in \R^4: q(x)=-1\},\]
define Minkowski 3-space, $\R^{2,1}$, as the projective quotient of the degenerate quadric
\[\{x \in \R^4: x_4^2=1\},\]
and define co-Minkowski 3-space, $^*\R^{2,1}$, as the projective quotient of the degenerate quadric
\[\{x \in \R^4: x_1^2+x_2^2-x_3^2=-1\}.\]
We consider $\RP^3$ oriented, which induces an orientation on all the mentioned domains.

Let $G_-$ be the identity component of ${\rm PO}(2,2)$. Note that it is isomorphic to $G \times G$. Define $G_0$ to be the subgroup of ${\rm PGL}(4, \R)$ consisting of the projectivized matrices of the form
\[\begin{pmatrix}
\raisebox{-0.2\height}[0pt][0pt]{\LARGE$A$} & \begin{matrix}
t_1 \\ t_2 \\ t_3
\end{matrix}\\
\begin{matrix}
0 & 0 & 0
\end{matrix}
& 1
\end{pmatrix}\]
where $A \in G$. For $g \in {\rm PGL}(4,\R)$ we define its \emph{dual} by $g^*:=(g^a)^{-1}$, where $g^a$ is the adjoint to $g$ with respect to $b$. Then the elements of $G_-$ are self-dual. Define $G_0^*$ to be the subgroup of ${\rm PGL}(4, \R)$ consisting of the duals to $g \in G_0$. In other words, it consists of the projectivized matrices of the form
\[\begin{pmatrix}
\raisebox{-0.2\height}[0pt][0pt]{\LARGE$A$} & \begin{matrix}
0 \\ 0 \\ 0
\end{matrix}\\
\begin{matrix}
t_1 & t_2 & t_3
\end{matrix}
& 1
\end{pmatrix}\]
where $A \in G$. The both groups $G_0$ and $G_0^*$ are isomorphic to $G \ltimes \R^3$.

In the pairs $(\A^3, G_-)$, $(\R^{2,1}, G_0)$ and $(^*\R^{2,1}, G_0^*)$ the corresponding Lie groups act smoothly and transitively on the corresponding spaces. Thus, they are geometries in the sense of Thurston~\cite{Thu}. (Furthermore, one can say projective geometries, to emphasize that the spaces are domains in $\RP^3$ and the groups are subgroups of ${\rm PGL}(4, \R)$.) We consider $\A^3$ and $\R^{2,1}$ endowed with the Lorentzian metrics induced by $b$, which are invariant with respect to the corresponding groups. The metric of $\A^3$ has constant sectional curvature $-1$, while the metric of $\R^{2,1}$ has constant sectional curvature 0. We fix a time-orientation on $\A^3$ and $\R^{2,1}$ so that the anti-de Sitter future directions for $p \in \A^3 \cap \R^{2,1}$ are future also for the Minkowski metric. The space $^*\R^{2,1}$ does not have a $G_0^*$-invariant non-degenerate pseudo-Riemannian metric. A natural metric on it is degenerate and is not induced by $b$, but we anyway will not use it. 

We will employ the basics of convex geometry in $\RP^3$. Recall that a subset $C \subset \RP^3$ is \emph{convex} if it is contained in an affine chart and is convex there. It is \emph{properly convex} if its closure is convex. We also say that a subset of $\RP^3$ is \emph{closely convex} if it is the closure of a convex subset. For $X \subset \RP^3$ we denote by $\cl(X)$ and $\conv(X)$ the closure and the closed convex hull of $X$ in $\RP^3$, where the latter means the inclusion minimal closely convex set containing $X$. We also make a convention that when we speak about the boundary of $X$, we mean it in $\RP^3$ in the sense of general topology, unless we indicate otherwise. In particular, the boundary of a not full-dimensional closed subset of $\RP^3$ is meant the subset itself. When we speak about the Hausdorff convergence, we mean it on the closed subsets of $\RP^3$.

We will rely on the projective duality in $\RP^3$ with respect to $q$. For a point $p \in \RP^3$, $p^* \subset \RP^3$ stands for the dual plane to $p$.
Let $C$ be a closely convex subset of $\RP^3$.
It determines a closed convex cone $K \subset \R^4$. Define the \emph{dual} $C^* \subset \RP^3$ of $C$ as the projective quotient of the set
\[\{x \in \R^4: b(x, x')\geq 0,~~~\forall x' \in K\}.\]
For an alternative definition, we say that a plane intersects $C$ transversely, if it intersects the projective span of $C$ transversely and intersects the relative interior of $C$. 
Then $C^*$ is exactly the set of points dual to the planes that do not intersect $C$ transversely. Note that $C^*$ is also closely convex and that the duality is polar, i.e., $C^{**}=C$. Note that if $C \subset \RP^3$ is closely convex and $g \in {\rm PGL}(4,\R)$, then $(gC)^*=g^*C^*$.

We will need a special interpretation for $\pt \A^3$. To this purpose, let $\rm Mat(2)$ be the space of $2\times 2$ real matrices. Consider an isomorphism 
\[\R^4 \ra \rm Mat(2),\]
\[(x_1, x_2, x_3, x_4) \mapsto
\begin{pmatrix}
x_1+x_3 & x_2+x_4 \\
x_2-x_4 & x_3-x_1
\end{pmatrix}.\]
Note that under this isomorphism we get
\[\pt \A^3\cong \{[A] \in \mathbb P(\rm{Mat}(2)): {\rm rank}(A)=1\}.\]
We use this to construct a diffeomorphism
\begin{equation}
\label{boundaryads}
\pt \A^3 \ra \RP^1 \times \RP^1,
\end{equation}
\[[A] \mapsto ({\rm im}(A), {\rm ker}(A)).\]
Recall that $\pt \A^3$ has a conformal Lorentzian structure, see, e.g.,~\cite[Section 2.2]{BS2}. Via identification~(\ref{boundaryads}), the sets $\{p\} \times \mathbb {RP}^1$ and $\mathbb {RP}^1 \times \{p\}$ correspond to the lightlike lines in $\pt \A^3$.

We will employ two natural charts on $\RP^3$. The first is
\begin{equation}
\label{chart}
\{x \in \RP^3: x_4 \neq 0\} \ra \R^3,
\end{equation}
\[[x_1, x_2, x_3, x_4] \mapsto \left(\frac{x_1}{x_4}, \frac{x_2}{x_4}, \frac{x_4}{x_4}\right)=:(y_1, y_2, y_3).\]
We note that $\R^{2,1}$, considered as a subset of $\RP^3$, is exactly the domain of this chart. Because of this, we will call it \emph{the Minkowski chart}. We fix the base point $o:=[0,0,0,1] \in \RP^3$, which corresponds to the origin of the chart. Via identification~(\ref{boundaryads}), the diagonal $\{(p,p): p \in \mathbb{RP}^1\}$ corresponds to $\pt \A^3 \cap o^*\cong \RP^1$.

The other chart that we will use is
\begin{equation}
\label{dualchart}
\{x \in \RP^3: x_3 \neq 0\} \ra \R^3,
\end{equation}
\[[x_1, x_2, x_3, x_4] \mapsto \left(\frac{x_1}{x_3}, \frac{x_2}{x_3}, \frac{x_4}{x_3}\right)=:(z_1, z_2, z_3).\]

The set $^*\R^{2,1}$ embeds in this chart, and because of this we will call it \emph{the co-Minkowski chart}. 
Recall that the hyperbolic plane $\H^2$ is frequently defined as 
\[\{y \in \R^{2,1}: y_1^2+y_2^2-y_3^2=-1,~~~y_3>0\},\]
where we use the coordinates of the Minkowski chart. We, however, prefer to call $\H^2$  the radial projection of this set from $o$ to $o^*$. Note that the intersection of $o^*$ with the co-Minkowski chart is the $z_1z_2$-coordinate plane. Then, as a point set, $\H^2$ is the open disk around the origin in this plane, which is the unit disk when the plane is endowed with the standard Euclidean metric. Then the co-Minkowski chart provides us the identification $^*\R^{2,1}\cong\mathbb \H^2 \times \R$.


The relative boundary $\pt_r \H^2$ coincides with $\pt \A^3 \cap o^*$ and hence gets identified with $\RP^1$. On the other hand, it is a unit Euclidean circle $\S^1$. Given a tangent vector $\xi$ at $p \in \pt_r \H^2$, we associate to it the point $(p, a) \in \pt^*\R^{2,1}$, where $a$ is the oriented length of $\xi$ in the Euclidean metric, where the clockwise direction is considered positive. This produces an identification 
 $\pt ^*\R^{2,1} \cong T\RP^1$, which we will rely on.

%

Let $g_t \in {\rm PGL}(4,\R)$ be the projectivization of the matrix 
\[\begin{pmatrix}
1/t & 0 & 0 & 0\\
0 & 1/t & 0 & 0\\
0 & 0 & 1/t & 0\\
0 & 0 & 0 & 1
\end{pmatrix}\]
Minkowski geometry $(\R^{2,1}, G_0)$ is the limit of $g_t(\A^3, G_-)$ as $t \ra 0$: $\cl(\R^{2,1})$ is the Hausdorff limit of $g_t\cl(\A^3)$, and $G_0$ is the Chabauty limit of $g_tG_-g_t^{-1}$ as subgroups of ${\rm PGL}(4,\R)$. In the Minkowski chart, $g_t$ just acts as the homothety from $o$ with the coefficient $1/t$. We have a natural identification $T_{o}\A^3 \cong \R^{2,1}$. Under this identification, if $p_t: [0,1] \ra \A^3$ is a differentiable curve with $p_0=o$, then $g_tp_t$ converges to $\dot p_0$ as $t \ra 0$, where $\dot p_0$ is considered as a point in  $\R^{2,1}\subset \RP^3$. 

The dual $g^*_t \in {\rm PGL}(4,\R)$ is the projectivization of the matrix 
\[\begin{pmatrix}
1 & 0 & 0 & 0\\
0 & 1 & 0 & 0\\
0 & 0 & 1 & 0\\
0 & 0 & 0 & 1/t
\end{pmatrix}\]
Then co-Minkowski geometry $(^*\R^{2,1}, G_0^*)$ is the limit of $g_t^*(\A^3, G_-)$ as $t \ra 0$.

\subsubsection{Anti-de Sitter and Minkowski spacetimes}


An anti-de Sitter (resp. Minkowski) (2+1)-spacetime has a $(\A^3, G_-)$-structure (resp. a $(\R^{2,1}, G_0)$-structure) in the sense of Thurston~\cite{Thu}. Thereby, geodesic segments in such spacetimes are segments of projective lines in local charts. By a convex surface in such spacetimes we mean an embedded orientable surface that is everywhere locally convex in the same direction in local charts. A convex surface $\Sigma$ is called \emph{future-convex} if locally the future cone of every point of $\Sigma$ is on the convex side. Similarly we define a past-convex surface. 

Recall from the introduction the notion of GHMC spacetimes and recall that in~\cite{Mes} Mess classified anti-de Sitter and Minkowski GHMC (2+1)-spacetimes. We now need to go to further details of this classification. We first mention the case of the anti-de Sitter ones. Let $\Omega$ be such a spacetime homeomorphic to $S \times \R$. Pick a pair $(\theta, m)$ of a holonomy $\theta: \pi_1S \ra G_-$ and a $\theta$-equivariant developing map $m: \tilde\Omega \ra \A^3$, preserving the orientation and time-orientation. Recall that $G_-\cong G \times G$, hence $\theta$ can be represented as $(\rho^l, \rho^r)$, where $\rho^l,\rho^r: \pi_1S \ra G$. Mess proved that both $\rho^l, \rho^r$ are discrete and faithful, hence belong to $\mc R$. Furthermore, $m$ is an embedding onto a convex domain in $\A^3$. On the other hand, for a given $(\rho^l, \rho^r) \in \mc R\times \mc R$ Mess constructed a unique maximal convex domain of discontinuity in $\A^3$. We will recall the details of this construction in the next section. Given that a pair of a holonomy and of a developing map is defined up to action of $G_-$, this gives a parameterization of marked isometry classes of GHMC anti-de Sitter (2+1)-spacetimes by $\mc T \times \mc T$. We note that such spacetimes are always timelike incomplete, i.e., no timelike geodesic can be extended infinitely in either direction.

Now consider the Minkowski situation. Let $\Omega$ be such a spacetime and $\theta: \pi_1S\ra G_0$ be its holonomy. Recall that $G_0\cong G \ltimes \R^{2,1}$. Mess proved that the projection $\rho: \pi_1S \ra G$ of $\theta$ is discrete and faithful, hence belongs to $\mc R$. It is twisted by a \emph{$\rho$-cocycle} $\tau: \pi_1S \ra \R^{2,1}$, which for all $\gamma_1, \gamma_2 \in \pi_1S$ satisfies
\[\tau(\gamma_1\gamma_2)=\rho(\gamma_1)\tau(\gamma_2)+\tau(\gamma_1).\]
We recall that there is a canonical identification $\R^{2,1}\cong\mf {so}(2,1)\cong\mf{sl}(2, \R)$, see, e.g.,~\cite{FS}. Furthermore, the space of $\mf{sl}(2,\R)$-valued $\rho$-cocycles is naturally identified with $T_\rho\mc R$, see, e.g.,~\cite{Gol}. We perceive $\tau$ as an element of $T_\rho\mc R$. Furthermore, Mess showed that the associated developing map $m: \tilde\Omega \ra \R^{2,1}$ is also an embedding onto a convex domain in $\R^{2,1}$. For every $(\rho, \tau) \in T\mc R$ he constructed two maximal convex domains of discontinuity in $\R^{2,1}$, one future-complete and one past-complete. Here a spacetime is \emph{future-complete} if every timelike geodesic extends infinitely in the future. Similarly one defines past-complete. This parameterizes marked isometry classes of future-/past-complete Minkowski GHMC (2+1)-spacetimes by $T\mc T$. We note that if two representations to $G_0$ are conjugated by $x \in \R^{2,1}$, then the respective cocycles differ by a $\rho$-coboundary, which is a cocycle $\tau: \pi_1S \ra \R^{2,1}$ satisfying for all $\gamma \in \pi_1S$
\[\tau(\gamma)=\rho(\gamma)x-x.\]

\subsubsection{Domains of discontinuity}
\label{sec:domain}

Here we describe the construction of domains of discontinuity, as some parts of the construction will be of use to us. We start from the anti-de Sitter situation. We refer to the excellent exposition of Bonsante--Seppi~\cite{BS2}. We will always have $\rho^l=\rho_\circ$ and will vary only $\rho^r$.
Pick $\rho \in \mc R$ and define $\theta_\rho: \pi_1S \ra G_- \cong G \times G$ by $\theta_\rho:=(\rho_\circ, \rho)$.

Let $\mc{QS}$ be the space of \emph{quasisymmetric homeomorphisms} $h: \RP^1 \ra \RP^1$. The exact definition of quasisymmetry and the topology on $\mc{QS}$ are a bit technical and not much relevant for us, we refer for them to~\cite[Chapter 16]{GL},~\cite{GS},~\cite[Chapter III]{Leh}. (We recall that the space of \emph{normalized} quasisymmetric homeomorphisms, i.e., those that fix 0, 1 and $\infty$, is frequently called the \emph{universal Teichm\"uller space}, as it contains all classical Teichm\"uller spaces.) Three facts will be relevant for us: that $\mc{QS}$ has a structure of a (complex) Banach space; that its topology is stronger than the topology of uniform convergence; and that there exists a smooth embedding $\mc R \times \mc R \hookrightarrow \mc{QS}$, where $(\rho^l, \rho^r) \in \mc R\times \mc R$ is sent to a unique $h \in \mc{QS}$ such that
for every $\gamma \in \pi_1S$, for the extensions of $\rho^l(\gamma)$ and $\rho^r(\gamma)$ to $\RP^1\cong\pt_r\H^2$, we have
\begin{equation}
\label{conjt}
\rho^l(\gamma)=h^{-1}\rho^r(\gamma) h.
\end{equation}
In our case, we have $\rho^l=\rho_\circ$, so we restrict the embedding above to $\mc R \hookrightarrow \mc{QS}$. For $\rho \in \mc R$ denote the respective homeomorphism by $h_\rho$. Equation~(\ref{conjt}) then turns into
\begin{equation}
\label{conj}
\rho_\circ(\gamma)=h_\rho^{-1}\rho(\gamma) h_\rho.
\end{equation}

Define $\Lambda_\rho \subset \pt\A^3$ to be the graph of $h_\rho$ via the identification $\pt\A^3 \cong \RP^1 \times \RP^1$ given by~(\ref{boundaryads}). From~\cite[Lemma 4.5.2]{BS2}, it is a achronal with respect to the causal structure of $\pt\A^3$ and is contained in an affine chart. Define $\tilde C_\rho:=\conv(\Lambda_\rho)\cap \A^3$. Because $\Lambda_\rho$ is achronal, one can see that 
\[\conv(\Lambda_\rho)=\tilde C_\rho\cup\Lambda_\rho=\cl(\tilde C_\rho).\]
Define $\tilde\Omega_\rho$ to be the interior of $\conv(\Lambda_\rho)^*$. Note that $\tilde C_\rho \subset \tilde\Omega_\rho\subset \A^3$ and $\pt\tilde\Omega_\rho\cap\pt\A^3=\Lambda_\rho$.
By a combination of~\cite[Proposition 4.6.4 and Proposition 5.4.4]{BS2}, $\theta_\rho$ acts freely and properly discontinuously on $\tilde\Omega_\rho$ and it is a maximal convex domain in $\A^3$ with this property (actually, maximal in $\RP^3$ with this property). The space $\Omega_\rho:=\tilde\Omega_\rho/\theta_\rho(\pi_1S)$ is a GHMC spacetime. For $\rho_1, \rho_2 \in \mc R$ different by conjugation, $\Omega_{\rho_1}$ and $\Omega_{\rho_2}$ are marked isometric, hence we can use the notation $\Omega_\rho$ for $\rho\in\mc T$.


\begin{lm}
\label{limset}
The set $\Lambda_\rho$ is the \emph{limit set} for $\theta_\rho$ in $\tilde\Omega_\rho$, i.e., for every $p \in \tilde\Omega_\rho$, the set of accumulation points of the $\theta_\rho$-orbit of $p$ is exactly $\Lambda_\rho$. 
\end{lm}

\begin{proof}
Let $\Lambda_\rho(p)$ be the limit set of $p \in \tilde\Omega_\rho$. Clearly, it is closed and $\theta_\rho$-invariant. By a result of Barbot~\cite[Theorem 10.13]{Bar3}, $\Lambda_\rho(p) \supset \Lambda_\rho$. Suppose that there is $q \in \Lambda_\rho(p)\backslash \Lambda_\rho$. Since $\theta_\rho$ acts properly discontinuously on $\tilde\Omega_\rho$, we have $q \in \pt\tilde\Omega_\rho$. Since $(\pt\tilde\Omega_\rho\backslash\Lambda_\rho) \subset \A^3$, we have $q \in \A^3$. Consider the plane $p^*$. This is a spacelike plane in $\A^3$ that is disjoint from $\cl(\tilde C_\rho)$. Hence, the maximal timelike distance between $p^*$ and $\cl(\tilde C_\rho)$ in the past from $p^*$ is some $a>0$. Since $\theta_\rho$ acts by isometries, for all the $\theta_\rho$-orbit of $p^*$ the timelike distance to $\cl(\tilde C)$ is $a$. On the other hand, $q^*$ is supporting to $\cl(\tilde C_\rho)$, hence the maximal timelike distance between $q^*$ and $\cl(\tilde C_\rho)$ is zero. Hence, $q^*$ cannot be an accumulation point for the orbit of $p^*$.
\end{proof}

When $\rho\neq \rho_\circ$, $\tilde C_\rho$ is full-dimensional and $\tilde \Omega_\rho$ is properly convex. The set $\Lambda_\rho$ divides $\pt \tilde C_\rho$ into two components, the future- and the past-convex ones, which we denote by $\pt^+\tilde C_\rho$ and $\pt^- \tilde C_\rho$ respectively. Similarly, $\Lambda_\rho$ divides $\pt\tilde \Omega_\rho$ into two components, which we also denote by $\pt^+\tilde\Omega_\rho$ and $\pt^-\tilde\Omega_\rho$. Next, we denote the connected components of the complement of $\tilde \Omega_\rho$ to $\tilde C_\rho$ by $\tilde \Omega^+_\rho$ and $\tilde\Omega^-_\rho$ respectively, where $\tilde \Omega^+_\rho$ is bounded between $\pt^+\tilde\Omega_\rho$ and $\pt^+\tilde C_\rho$. 

In the case $\rho=\rho_\circ$, $\tilde C_{\rho}$ coincides with $\H^2 \subset o^*$. Then $\tilde\Omega_{\rho}$ coincides with $^*\R^{2,1}$ as a set. In particular, it is convex, but not properly convex. Note that its intersection with the Minkowski chart is the union of the two open cones based at $o$ and spanned by $\H^2$. We consider then $\pt^+\tilde C_{\rho}$, $\pt^-\tilde C_\rho$ coinciding with $\tilde C_\rho$. As for $\pt^+\tilde\Omega_\rho$, $\pt^-\tilde\Omega_\rho$, we denote so the boundaries of the respective cones. The domains $\tilde \Omega^+_\rho$ and $\tilde\Omega^-_\rho$ are defined the same way as before.

We denote by $C_\rho\subset\Omega_\rho$ the projection of $\tilde C_\rho$. We define $\pt^\pm C_\rho$, $\Omega^\pm_\rho$ in an obvious way.


Now we pass to the construction in the Minkowski case. The initial description of Mess was quite different from his construction in the anti-de Sitter case. We, however, need to give a description that is similar to the anti-de Sitter one. To this purpose, we will employ the duality between $\R^{2,1}$ and $^*\R^{2,1}$. Pick $\tau \in T_{\rho_\circ}\mc T$. Define $\iota_\tau: \pi_1S \ra G_0$ to be the representation obtained by twisting $\rho_\circ$ by $\tau$. 
Denote by $\iota^*_\tau: \pi_1S \ra G_0^*$ the \emph{dual representation}.

Consider the identification $^*\R^{2,1} \cong \H^2 \times \R$. A continuous function $b: \H^2 \ra \R$ is called \emph{$\tau$-equivariant} if its graph is $\iota_\tau^*$-invariant as a subset of $^*\R^{2,1}$. From~\cite[Corollary 3.14]{BF}, there exists a function $a_\tau: \RP^1 \ra \R$ such that any $\tau$-equivariant function on $\H^2$ extends continuously to $\RP^1$ by $a_\tau$. It follows, in particular, that $a_\tau$ is the unique $\tau$-equivariant function on $\RP^1$, i.e., whose graph, which we denote by $\Lambda_\tau \subset \pt^*\R^{2,1}$, is $\iota_\tau^*$-invariant. Via the identification $\pt^*\R^{2,1}\cong T\RP^1$, $\Lambda_\tau$ determines a vector field $\xi_\tau$ on $\RP^1$. Recall that $\tau$ can be considered as a function $\tau: \pi_1S \ra \mf{sl}(2,\R)$ and the latter may be interpreted as the algebra of the Killing fields on $\H^2$. The $\tau$-equivariance of $a_\tau$ translates as the condition that for every $\gamma \in \pi_1S$, for the extensions of $\rho_\circ(\gamma)$ and $\tau(\gamma)$ to $\RP^1$, we have
\begin{equation}
\label{infconj}
\tau(\gamma)=\xi_\tau-\rho_\circ(\gamma)_*\xi_\tau,
\end{equation} 
where we perceive each $\tau(\gamma)$ as the extension of a Killing field. It follows that $\xi_\tau$ is a unique vector field on $\RP^1$ satisfying such condition.

Define $\tilde C_\tau:=\conv(\Lambda_\tau)\cap {^*\R^{2,1}}$ and define $\tilde\Omega_\tau$ to be the interior of $\conv(\Lambda_\tau)^*$. One can check that $\iota_\tau$ acts on it freely and properly discontinuously and that $\tilde\Omega_\tau$ is the maximal convex subset of $\RP^3$ with this property. However, it is not contained in $\R^{2,1}$. Its intersection with $\R^{2,1}$ consists of two convex domains, which we denote by $\tilde\Omega^+_\tau$ and $\tilde\Omega^-_\tau$. One can see that $\tilde\Omega_\tau=\tilde\Omega^+_\tau\cup \H^2\cup \tilde\Omega^-_\tau$ (recall that by $\H^2$ we mean a disk in $o^*$). In this case, $\tilde\Omega_\tau$ is properly convex if and only if $\tau$ is not a coboundary. We have $\pt_r\H^2 \subset \pt \tilde\Omega_\tau$ and, provided that $\tilde\Omega_\tau$ is properly convex, $\pt_r\H^2$ divides $\pt \tilde\Omega_\tau$ into two components, which we denote by $\pt^+\tilde\Omega_\tau$ and $\pt^-\tilde\Omega_\tau$ respectively. In the case when $\tilde\Omega_\tau$ is not properly convex, we use the same convention for $\pt^+\tilde\Omega_\tau$ and $\pt^-\tilde\Omega_\tau$ as in the anti-de Sitter situation.

We need to check that $\tilde\Omega^\pm_\tau$ are indeed the same domains that were described by Mess. Let $B^+_\tau$ be the set of points in $^*\R^{2,1}$ that are below the upper boundary component of $\tilde C_\tau$. This is a convex subset of $\RP^3$. The interior of the intersection of $\cl(B^+_\tau)^*$ with $\R^{2,1}$, which we denote by $B_\tau^{+*}$, is a convex future-complete $\iota_\tau$-invariant subset of $\R^{2,1}$. Furthermore, since $\cl(B^+_\tau)$ is the closed convex hull of points in $\pt^*\R^{2,1}$ (those that are below $\Lambda_\tau$), $\cl(B_\tau^{+})^*$ is the intersection of the future half-spaces of a set of lightlike planes in $\R^{2,1}$. Thus, it is what is called a \emph{regular domain} in the terminology of Bonsante~\cite{Bon}. By~\cite[Theorem 5.1]{Bon}, it is a unique regular domain. Since the output of the construction of Mess are also regular domains, the constructions produce the same result. We define the quotients $\Omega^\pm_\tau:=\tilde\Omega^\pm_\tau/\iota_\tau(\pi_1S)$. These are GHMC spacetimes. When $\tau_1,\tau_2 \in T_{\rho_\circ}\mc R$ differ by a coboundary, $\Omega_{\tau_1}$ and $\Omega_{\tau_2}$ are marked isometric, hence we can use the notation $\Omega_\tau$ for $\tau\in T_{\rho_\circ}\mc T$.

\subsubsection{Convergence of domains of discontinuity}

In this subsection, when we speak about convergence of closed subsets of $\RP^3$, we mean the Hausdorff convergence, unless we specify otherwise. 

Let $\rho_i \ra \rho$ in $\mc R$. Since the embedding $\mc R \hookrightarrow \mc{QS}$ is continuous, the respective homeomorphisms $h_{\rho_i}$ converge to $h_\rho$ in $\mc{QS}$. This particularly means that they converge to $h_\rho$ uniformly. From this it follows

\begin{lm}
\label{convlimsetbase}
The sets  $\Lambda_{\rho_i}$ converge to $\Lambda_\rho$.
\end{lm}

Since $\cl(\tilde C_\rho)=\conv(\Lambda_\rho)$, we get

\begin{crl}
\label{convccorebase}
The sets  $\cl(\tilde C_{\rho_i})$ converge to $\cl(\tilde C_\rho)$.
\end{crl}

Since $\cl(\tilde\Omega_\rho)$ is dual to $\cl(\tilde C_\rho)$, we obtain

\begin{crl}
\label{convdombase}
The sets  $\cl(\tilde\Omega_{\rho_i})$ converge to $\cl(\tilde\Omega_\rho)$.
\end{crl}

Here and in what follows we perceive all mentioned sets as subsets of $\RP^3$. We will rely on the following elementary principle.

\begin{lm}
\label{convcompl}
Let $C_i$ be a sequence of closely convex subsets of $\RP^3$ converging to a closely convex subset $C$. Then $\pt C_i$ converge to $\pt C$ and $\cl(\RP^3\backslash C_i)$ converge to $\cl(\RP^3\backslash C)$.
\end{lm}

For the proof of the next lemma we need the following basic claim.

\begin{lm}
\label{subconvtoconv}
Let $X$ be a Hausdorff topological space and $\{x_i\}$ be a sequence in $X$ with the property that every its subsequence contains a further subsequence that converges to $x \in X$. Then $\{x_i\}$ converges to $x$.
\end{lm}

\begin{lm}
\label{convtemp1-}
The sets $\cl(\pt^\pm\tilde C_{\rho_i})$ converge to $\cl(\pt^\pm\tilde C_{\rho})$.
\end{lm}

\begin{proof}
We consider separately the cases $\rho=\rho_\circ$ and $\rho\neq\rho_\circ$. Consider first the former case, thus $\pt^\pm\tilde C_{\rho_\circ}=\tilde C_{\rho_\circ}$. Consider the co-Minkowski chart. Introduce the standard Euclidean metric on it. In this metric, $\cl(\tilde C_{\rho_\circ})$ is the closed unit disk in the $z_1z_2$ plane. Since $\cl(\tilde C_{\rho_i})$ converge to $\cl(\tilde C_{\rho_\circ})$, for all large enough $i$, $\cl(\tilde C_{\rho_i})$ belongs to the chart. Consider the orthogonal projection of $\cl(\pt^\pm\tilde C_{\rho_i})$ to the $z_1z_2$ plane. The images are continuous images of a 2-disk. Up to subsequence, they converge to a subset of $\cl(\tilde C_{\rho_\circ})$, while the images of the boundaries of the disks converge to the relative boundary of $\cl(\tilde C_{\rho_\circ})$. Then the images of the disks converge to $\cl(\tilde C_{\rho_\circ})$. Using Lemma~\ref{subconvtoconv}, the projections of $\cl(\pt^\pm\tilde C_{\rho_i})$ converge to $\cl(\tilde C_{\rho_\circ})$ for the initial sequence. This implies that $\cl(\pt^\pm\tilde C_{\rho_i})$ converge to $\cl(\pt^\pm\tilde C_{\rho})$.

Now we pass to the case $\rho \neq \rho_\circ$. First we need an interlude. Let $\psi_i$ be a sequence of simple closed curves in some ambient manifold converging in the Hausdorff sense to a simple closed curve $\psi$. Assume that all the curves are oriented. We say that $\psi_i$ converge to $\psi$ \emph{orientedly} if for a positive triple $p^1$, $p^2$, $p^3$ of distinct points on $\psi$ and for a sequence of triples $p^1_i$, $p^2_i$, $p^3_i$ on $\psi_i$, converging to $p^1$, $p^2$, $p^3$ respectively, all but finitely many triples are positive. One can observe that then it holds for any initial triple and any converging sequence of triples. 

Now consider a 2-sphere $S^2$ and assume that $\psi_i$ is a sequence of oriented simple closed Lipschitz curves on $S^2$ converging orientedly to an oriented simple closed Lipschitz curve $\psi$. Each curve divides $S^2$ into two domains. Orient $S^2$ and denote by $D^+$ the domain of $S^2$, for which $\psi$ is the boundary and for which at each point of $\psi$, where $\psi$ is differentiable, the direction along $\psi$ together with a direction outside the domain is positive. Similarly we define $D^+_i$. One can observe then that the closures of $D^+_i$ converge in the Hausdorff sense to the closure of $D^+$. 

Orient $\RP^3$ and recall that $\A^3$ is future-oriented. The future-orientation of $\A^3$ induces a future-orientation on $\pt\A^3$. Further, this induces an orientation on every spacelike curve on $\pt\A^3$ by demanding that at every point, where the curve is differentiable, the direction along the curve, a future direction along $\pt\A^3$ and a direction outwards $\A^3$ form a positive triple in $\RP^3$. In particular, this induces an orientation on all $\Lambda_\rho$. When $\rho_i$ converge to $\rho$, the convergence of $\Lambda_{\rho_i}$ to $\Lambda_\rho$ comes from the convergence of graphs under $\pt\A^3\cong \RP^1\times\RP^1$. Thus, this convergence is oriented.

Now return to our problem. Due to Corollary~\ref{convccorebase}, $\cl(\tilde C_{\rho_i})$ converge to $\cl(\tilde C_{\rho})$. Since these are closely convex sets, $\pt\tilde C_{\rho_i}$ converge to $\pt\tilde C_{\rho}$. Due to Lemma~\ref{convlimsetbase}, $\Lambda_{\rho_i}$ converge to $\Lambda_{\rho}$. Pick $p \in \inter(\tilde C_{\rho})$. For all large enough $i$, we have $p \in \inter(\tilde C_{\rho_i})$. We project $\pt\tilde C_{\rho_i}$ and $\pt\tilde C_{\rho}$ onto the sphere of directions from $p$, which we denote by $S^2$. Then $\Lambda_{\rho_i}$ and $\Lambda_{\rho}$ are homeomorphically projected onto simple closed Lipschitz curves on $S^2$, which we denote by $\psi_i$ and $\psi$, and $\pt^\pm\tilde C_{\rho_i}$, $\pt^\pm\tilde C_{\rho}$ are projected homeomorphically onto domains bounded by $\psi_i$ and $\psi$. Due to our observation and by construction, the closures of the former domains converge to the closure of the latter. Hence $\cl(\pt^\pm\tilde C_{\rho_i})$ converge to $\cl(\pt^\pm\tilde C_{\rho})$.
\end{proof}

Just in the same way one shows

\begin{lm}
\label{convtemp2-}
The sets $\cl(\pt^\pm\tilde \Omega_{\rho_i})$ converge to $\cl(\pt^\pm\tilde \Omega_{\rho})$.
\end{lm}

We need it for

\begin{crl}
\label{convdom2-}
The sets $\RP^3\backslash \tilde\Omega^+_{\rho_i}$ converge to $\RP^3\backslash \tilde\Omega^+_\rho$.
\end{crl}

\begin{proof}
Define $\tilde\Omega^{--}_\rho$ to be the future of $\pt^+\tilde C_\rho$ in $\tilde\Omega_\rho$. This is a convex domain, whose boundary is $\pt^+\tilde C_\rho \cup \Lambda_\rho\cup \pt^-\tilde\Omega_\rho$. We define similarly $\tilde\Omega^{--}_{\rho_i}$. Lemmas~\ref{convlimsetbase},~\ref{convtemp1-} and~\ref{convtemp2-} imply that $\pt\tilde\Omega^{--}_{\rho_i}$ converge to $\pt\tilde\Omega^{--}_\rho$. Hence,  $\tilde\Omega^{--}_{\rho_i}$ converge to $\tilde\Omega^{--}_\rho$. It remains to observe that 
\[\RP^3\backslash \tilde\Omega^+_{\rho_i}=\cl(\tilde\Omega^{--}_{\rho_i})\cup(\RP^3\backslash \tilde\Omega_{\rho_i}),\]
\[\RP^3\backslash \tilde\Omega^+_\rho=\cl(\tilde\Omega^{--}_\rho)\cup(\RP^3\backslash \tilde\Omega_\rho)\]
and use Corollary~\ref{convdombase} and Lemma~\ref{convcompl}.
\end{proof}

\subsubsection{Spaces of bent surfaces}
\label{sec:bentsurf}

Define $\tilde{\mc P}_-=\tilde{\mc P}_-(\rho_\circ, V):=\mc R \times (\A^3)^V$. We will denote its elements by $(\rho, \tilde f)$ where $\tilde f: V \ra \A^3$. Fix a lift $V \rightarrow \tilde V$ that we call canonical. Using it, we extend $\tilde f$ to a $\theta_\rho$-equivariant map $\tilde f: \tilde V \ra \A^3$. We will sometimes consider elements of $V$ as elements of $\tilde V$ via the canonical lift. For $v \in V$ we denote by $\tilde{V \backslash v}$ the preimage of $V\backslash v$ in $\tilde V$. Denote by $\tilde{\mc P}^w_- \subset \tilde{\mc P}_-$ the subset of $(\rho, \tilde f)$ such that $\tilde f(V) \subset \tilde\Omega_\rho$. Due to Lemma~\ref{convcompl}, $\tilde{\mc P}^w_-$ is open. We say that $\tilde f$ is in a \emph{(future-)convex position} if $\tilde f(V) \subset (\tilde\Omega^+_\rho \cup \pt^+\tilde C_\rho)$, $\tilde f$ is injective and for every $v \in V$ we have $\tilde f(v) \notin \inter(\clconv(\tilde f(\tilde{V \backslash v}))).$ We say that $\tilde f$ is in a \emph{strictly (future-)convex position} if $\tilde f(V) \subset \tilde\Omega^+_\rho$ and for every $v \in V$ we have $\tilde f(v) \notin \clconv(\tilde f(\tilde{V \backslash v}))$. Note that in Part 1 we deal only with maps in future-convex/strictly future-convex position, so we will omit the word future in this context. However, we will resume to use it in Part 2, where we will need similarly to define a \emph{past-convex position}. We denote the subset of $(\rho, \tilde f)$ when $\tilde f$ is in convex position by $\tilde{\mc P}_-^c\subset \tilde{\mc P}_-^w$ and the subset when $\tilde f$ is in strictly convex position by $\tilde {\mc P}_-^s \subset \tilde{\mc P}_-^c$. Due to Corollary~\ref{convdom2-}, the latter is open in $\tilde{\mc P}_-^w$, thereby it is a manifold of dimension $3(n-k)$. For $(\rho, \tilde f)\in \tilde{\mc P}_-^w$ define $\clconv(\tilde f):=\clconv(\tilde f(\tilde V))$. Due to Lemma~\ref{limset}, $\cl(\tilde C_\rho)\subset \clconv(\tilde f)$. The boundary of $\clconv(\tilde f)$ consists of $\Lambda_\rho$, a future-convex and a past-convex spacelike surfaces. We denote the future-convex one by $\Sigma(\tilde f)$. When $(\rho, \tilde f) \in \tilde{\mc P}_-^c$, the past-convex one is $\pt^-\tilde C_\rho$. 

In a similar way we define $\tilde {\mc P}_0=\tilde {\mc P}_0(\rho_\circ, V) := T_{\rho_\circ}\mc R \times (\R^{2,1})^V$ and denote its elements by $(\tau, \tilde f)$ where $\tilde f: \tilde V \ra \R^{2,1}$ is a $\iota_\tau$-equivariant map. Define $\tilde {\mc P}_0^w$ as the subset of those $(\tau, \tilde f)$ that $\tilde f(V) \subset \tilde\Omega_\tau^+$; define $\tilde{\mc P}_0^c$ as the subset of those that, in addition, are injective and for every $v \in V$ we have $\tilde f(v) \notin \inter(\clconv(\tilde f(\tilde{V \backslash v})))$; define $\tilde{\mc P}_0^s$ as the subset of those that for every $v \in V$ we have $\tilde f(v) \notin \clconv(\tilde f(\tilde{V \backslash v}))$.
Note that $\tilde{\mc P}_0$ naturally has a structure of a vector space and the other spaces are cones in it. We use the notation $\conv(\tilde f)$, $\Sigma(\tilde f)$ similarly as above. 

Lemmas~\ref{limset} and~\ref{convlimsetbase} imply

\begin{lm}
\label{convconv}
Let $(\rho_i, \tilde f_i) \ra (\rho, \tilde f)$ in $\tilde{\mc P}_-^w$. Then $\clconv(\tilde f_i) \ra \clconv(\tilde f)$. 
\end{lm}

Moreover, the same proof as the proof of Lemma~\ref{convtemp1-} imply

\begin{lm}
\label{convsurf}
Let $(\rho_i, \tilde f_i) \ra (\rho, \tilde f)$ in $\tilde{\mc P}_-^w$. Then $\cl(\Sigma(\tilde f_i)) \ra \cl(\Sigma(\tilde f))$. 
\end{lm}

The group $G_-$ acts on $\tilde{\mc P}_-$ from the left by conjugation on $\mc R$ and by isometries on $(\A^3)^V$. This action is free and properly discontinuous. Furthermore, $\pi_1S$ acts on $\tilde{\mc P}_-$ from the left fiberwise, via $\theta_\rho$ on $\{\rho\} \times (\A^3)^V$. This action is free and properly discontinuous on $\tilde{\mc P}_-^w$. These two actions commute and we denote the quotient by $\mc P_-^w$. Its elements are the pairs $(\rho, f)$ where $\rho\in\mc T$ and $f: V \ra \Omega_\rho$. Define the subsets $\mc P_-^s, \mc P_-^c$ in the obvious manner. We denote by $\conv(f)$, $\Sigma(f)$ the projections of $\conv(\tilde f)$, $\Sigma(\tilde f)$ for some lifts $(\rho, \tilde f) \in \tilde{\mc P}_-^w$. In the same way we define $\mc P_0^w \supset {\mc P}_0^c \supset {\mc P}_0^s$. They all have a coned structure; ${\mc P}_0^w$ and ${\mc P}_0^s$ are coned manifolds. Their elements are the pairs $(\tau, f)$ where $\tau\in T_{\rho_\circ}\mc T$ and $f: V \ra \Omega^+_\tau$.

One might expect that the surfaces $\Sigma(\tilde f)$ are (locally) polyhedral, i.e., around every point they coincide with a part of the boundary of a convex polyhedron in $\RP^3$ (by a polyhedron we mean the convex hull of finitely many points). Curiously, while it holds for $(\tau, \tilde f) \in \tilde{\mc P}_0^w$ (see~\cite[Lemma 2.7]{FP}), it does not hold for  $(\rho, \tilde f) \in \tilde{\mc P}_-^w$. The main reason for this is that $\pt^+\tilde C_\rho$ can be non-polyhedral. As a convex spacelike surface, it is endowed with the intrinsic metric, see details in Section~\ref{sec:intmet}. There is an isometry $\H^2 \ra \pt^+\tilde C_\rho$. The surface $\pt^+\tilde C_\rho$ is totally geodesic apart from a closed set of complete geodesics of $\A^3$. The preimage of this set is a geodesic lamination in $\H^2$ invariant with respect to a Fuchsian representation of $\pi_1S$, which we denote by $\rho^+ \in \mc R$. The data of how $\pt^+\tilde C_\rho$ is bent in $\A^3$ defines a transverse measure on the geodesic lamination. Denote the $\rho^+$-projection of the obtained measured lamination to $S$ by $\lambda^+$, which is a measured geodesic lamination on $S$. It is called the \emph{bending lamination} of $\pt^+C_\rho$. The measure of an isolated leaf of $\lambda^+$ is the exterior dihedral angle in $\A^3$ of the respective bending line of $\pt^+\tilde C_\rho$. Measured geodesic laminations on $S$ naturally form a PL-manifold $\mc{ML}=\mc{ML}(S)$, homeomorphic to a $(6k-6)$-dimensional ball. For an introduction to measured geodesic laminations we refer to~\cite{Bon2, Mar}, and for the details of this construction we refer to~\cite{Mes, BS2}. This is similar to the geometry of quasi-Fuchsian hyperbolic 3-manifolds, see, e.g.,~\cite{CEG, BO}.

From the work of Mess~\cite{Mes}, any measured geodesic lamination can appear as $\lambda^+$. In particular, it can have non-isolated leaves. In this case, $\pt^+\tilde C_\rho$ is non-polyhedral and $\Sigma(\tilde f)$ can be non-polyhedral as well. In particular, the image of $\tilde f$ can belong to $\pt^+\tilde C_\rho$, in which case $\Sigma(\tilde f) = \pt^+\tilde C_\rho$. Even if $\tilde f$ is in a strictly convex position, $\Sigma(\tilde f)$ still can have nonempty intersection with $\pt^+\tilde C_\rho$, in which case it can fail to be polyhedral. All non-polyhedrality, however, anyway comes only from the intersection with $\pt^+\tilde C_\rho$.

Let $\Sigma \subset \RP^3$ be a (locally) convex embedded surface. If $p \in \Sigma$ does not belong to the relative interior of any segment belonging to $\Sigma$, we call $p$ a \emph{vertex} of $\Sigma$. If it belongs to the relative interiors of two such segments with distinct tangents, $p$ is called \emph{regular}. Otherwise, it is called an \emph{edge-point}.  A \emph{face} of $\Sigma$ is the closure in $\Sigma$ of a connected component of the set of regular points. An \emph{edge} is the closure of a maximal segment in $\Sigma$ consisting from edge-points. We say that $\Sigma$ is \emph{bent} if the set of vertices is discrete. It is \emph{strictly polyhedral} if it is polyhedral and each face is isomorphic to a (compact affine) polygon. These notions are local and extend to convex surfaces in anti-de Sitter spacetimes (and in locally projective manifolds in general). For every $(\rho, \tilde f) \in \tilde{\mc P}_-^w$, $\Sigma(\tilde f)$ is a bent surface. It is strictly polyhedral if and only if the timelike distance between $\Sigma(\tilde f)$ and $\pt^+\tilde C_\rho$ is positive. A proof is the same as in the hyperbolic case, see~\cite[Corollary 3.19]{Pro2}. In such case, we also say that $\tilde f$ is strictly polyhedral. This notion extends to the elements of $\mc P_-^w$. The set of vertices of $\Sigma(f)$ will be denoted by $V(f)\subseteq V$. We have $(\rho, f) \in \mc P_-^s$ if and only if $V(f)=V$. We denote the respective subsets of strictly polyhedral elements by $\tilde{\mc P}_{-, sp}^c$, $\tilde{\mc P}_{-, sp}^s$, $\mc P_{-, sp}^c$, $\mc P_{-, sp}^s$. As for $(\tau, \tilde f) \in \tilde{\mc P}_0^c$, it was shown in~\cite[Lemma 2.7]{FP} that in fact $\tilde f$ is always strictly polyhedral.

Pick $\rho \in \mc T$ and consider $\Omega_\rho$. By construction, for every Cauchy surface $\Sigma \subset\Omega_\rho$ there exists a homeomorphism $\zeta: \Sigma \ra S$, which is defined up to isotopy. For $(\rho, f) \in \mc P_-^c$, $\zeta: \Sigma(f) \ra S$ can be chosen so that $\zeta \circ f$ is the identity on $V$. Then such $\zeta$ is chosen up to $h \in \mc H$. If $(\rho, f) \in \mc P_{-, sp}^c$ and such $\zeta$ is chosen for $\Sigma(f)$, it pushes forward the edges of $\Sigma(f)$ to a celluation of $(S, V)$. A \emph{celluation} of $(S, V)$ is defined similarly as a triangulation with the difference that now we allow cells with arbitrary number of vertices as faces and also allow them to contain some points of $V$ in the interior. The notions of equivalence and weak equivalence apply also to celluations. By a \emph{face celluation} of $(\rho,  f)$ we will mean a celluation of $(S, V)$ as above, which is then defined up to weak equivalence. We will abuse the terminology and say that a celluation $\ms C_1$ of $(S, V)$ is a subdivision of a celluation $\ms C_2$ if $\ms C_1$ is weakly equivalent to a subdivision of $\ms C_2$ in a straightforward sense. All these notions apply also to the elements of $\mc P_0^c$. It was shown in~\cite[Lemma 2.13]{FP} that for any $(\tau, f) \in \mc P_0^c$ there exists a neighborhood $U \ni (\tau, f)$ such that for every $(\tau', f') \in U$ the face celluation of $f'$ is a subdivision of the face celluation of $f$. The same proof works to prove the same claim for $(\rho, f) \in \mc P_{-, sp}^c$. We can also speak about face celluations of $\tilde f$ for $(\rho, \tilde f) \in  \tilde{\mc P}_{-, sp}^c$ or $(\tau, \tilde f) \in \tilde {\mc P}_0^c$, where we mean the respective decompositions of $(\tilde S, \tilde V)$ and the equivalences are $\pi_1S$-invariant.

\subsection{Blow-up of the space of bent surfaces}

\subsubsection{Convergence of domains of discontinuity at the blow-up}
\label{sec:convblowup}

Consider a continuous curve $\rho_t: [0,1] \ra \mc R$ with $\rho_0=\rho_\circ$, differentiable at $t=0$ with $\dot\rho_0=\tau \in T_{\rho_\circ}\mc R$. As $t\ra 0$, we have

\begin{lm}
\label{convlimset}
The sets $g_t^*\Lambda_{\rho_t}$ converge to $\Lambda_\tau$.
\end{lm}

\begin{proof}
Let $h_t:=h_{\rho_t}$ be the respective homeomorphisms conjugating $\rho_t$ to $\rho_\circ$ given by~(\ref{conj}). Since the embedding $\mc R \hookrightarrow \mc{QS}$ is smooth, one can differentiate the path $h_t$ at zero and get a vector field $\dot h_0$. By differentiating the condition 
$\rho_\circ=h_t^{-1}\rho_t h_t,$ we get that $\dot h_0$ satisfies~(\ref{infconj}). By uniqueness, it coincides with $\xi_\tau$. Now we pass to the co-Minkowski coordinate chart, and notice that in this chart $g_t^*$ acts by preserving $z_1, z_2$ and multiplying $z_3$ by $1/t$. Then for any $p \in \RP^1$, the $g_t^*$-images of $(p, h_t(p)) \in \RP^1\times\RP^1\cong \pt \A^3$ converge to $(p, \dot h_0(p)) \in T\RP^1 \cong \pt ^*\R^{2,1}$ as $t \ra 0$. Pick any $\R$-chart for $\RP^1$, so $h_t$, $h$, $\dot h_0$ become $\R$-valued functions defined on a domain in $\R$. Since the embedding $\mc R \hookrightarrow \mc{QS}$ is smooth, by considering the Fr\'echet derivative of $h_t$ at $t=0$ in the uniform topology, it follows that in the chart $(h_t-h_0)/t$ converge to $\dot h_0$ uniformly. This implies that $\Lambda_\tau$ is the Hausdorff limit of $g_t^*\Lambda_{\rho_t}$, as desired.
\end{proof}

By taking the convex hulls, we get

\begin{crl}
The sets $g_t^*\cl(\tilde C_{\rho_t})$ converge to  $\cl(\tilde C_\tau)$.
\end{crl}

By passing to the dual sets, we obtain

\begin{crl}
\label{convdom}
The sets $g_t\cl(\tilde\Omega_{\rho_t})$ converge to $\cl(\tilde\Omega_\tau)$.
\end{crl}

\begin{lm}
\label{convlimset2}
The sets $g_t\Lambda_{\rho_t}$ converge to $\Lambda_{\rho_\circ}$.
\end{lm}

\begin{proof}
Due to Lemma~\ref{convlimsetbase}, $\Lambda_{\rho_t}$ converge to $\Lambda_{\rho_\circ}$. However, $\Lambda_{\rho_\circ}$ is pointwise fixed by all $g_t$ and every point of $\Lambda_{\rho_\circ}$ has a basis of neighborhoods $\{U_i\}$ such that for all $t$ and $i$ we have $g_tU_i \subset U_i$. From this and from the equivalence of the Hausorff topology on the space of closed subsets of $\RP^3$ to the Vietoris topology, it follows that $g_t\Lambda_{\rho_t}$ converge to $\Lambda_{\rho_\circ}$. 
\end{proof}

\begin{crl}
\label{convccore2}
The sets $g_t\cl(\tilde C_{\rho_t})$ converge to  $\cl(\tilde C_{\rho_\circ})$.
\end{crl}

The next three results are obtained the same as Lemmas~\ref{convtemp1-},~\ref{convtemp2-} and Corollary~\ref{convdom2-}.

\begin{lm}
\label{convtemp1}
The sets $g_t\cl(\pt^-\tilde C_{\rho_t})$ converge to $\cl(\tilde C_{\rho_\circ})$.
\end{lm}

\begin{lm}
\label{convtemp2}
The sets $g_t\cl(\pt^+\tilde \Omega_{\rho_t})$ converge to $\cl(\pt^+\tilde \Omega_\tau)$.
\end{lm}



\begin{crl}
\label{convdom2}
The sets $g_t(\RP^3\backslash \tilde\Omega^{+}_{\rho_t})$ converge to $\RP^3\backslash\tilde\Omega^+_\tau$.
\end{crl}

\subsubsection{$K$-surfaces foliation}
\label{sec:ksurf}

We will rely on an important theorem of Barbot--Beguin--Zeghib~\cite{BBZ}:

\begin{thm}
\label{BBZ}
For every $\rho \in \mc T$ there exists a smooth foliation of $\Omega^+_\rho$ by Cauchy surfaces of constant Gauss curvature. 
\end{thm}

By the Gauss equation, the sectional curvature at a point of a surface with the Gauss curvature $K$ is $-K-1$. The leaves of the foliation are strictly convex. By strictly convex we mean a smooth convex surface with non-degenerate shape operator.

Note that in~\cite{BBZ} the authors state only that the foliation is continuous. However, in another paper~\cite{BBZ2} they prove that $\Omega_\rho$ is foliated by CMC-surfaces (of constant mean curvature), and there they show that this foliation is smooth (in fact, analytic). It was observed that if $\Sigma$ is a CMC-surface of mean curvature $H$, then its normal evolution in past at time $H/2+1+\sqrt{H^2/4+1}$ is a future-convex $K$-surface with $K=H(H+\sqrt{H^2/4+1})/2$, see, e.g.,~\cite[Proposition 7.1.4]{BS2}. Thus, the smoothness of the $K$-surface foliation also follows. 

Pick $\rho \in \mc R$. Let $L \subset \Omega_\rho$ be a leaf of the foliation from Theorem~\ref{BBZ} and $\tilde L \subset\tilde\Omega_\rho$ be its preimage. From Lemma~\ref{limset}, $\cl(\tilde L)=\tilde L \cup \Lambda_\rho$. Pick $p \in \Lambda_\rho$, let $\Pi$ be the tangent plane at $p$ to $\pt\A^3$. We claim

\begin{lm}
\label{duallimset}
For any $p_i \ra p$, $p_i \in \tilde L$, the supporting planes $\Pi_i$ at $p_i$ to $\tilde L$ converge to $\Pi$.
\end{lm}

\begin{proof}
Consider the dual surface $\tilde L^*$, which consists from the points dual to the supporting planes to $\tilde L$. Then it is a strictly past-convex $\theta_\rho$-invariant surface in $\tilde \Omega^-_\rho$. From Lemma~\ref{limset}, $\cl(\tilde L^*)=\tilde L^* \cup \Lambda_\rho$. Let $q_i \in \tilde L^*$ be dual to $\Pi_i$. Up to subsequence, they converge to $p' \in \Lambda_\rho$. Let $\Pi'$ be the tangent plane at $p'$ to $\pt \A^3$. Since $(p')^*=\Pi'$, the respective subsequence of $\Pi_i$ converge to $\Pi'$. But then $\Pi'$ must pass through $p$. Since $\Lambda_\rho$ is achronal, $p'=p$, hence $\Pi'=\Pi$.  Using Lemma~\ref{subconvtoconv}, we get the desired result.
\end{proof}

Define now $\tilde L^+=\tilde L$ and pick a leaf $\tilde L^-$ of a similar foliation of $\tilde\Omega_\rho^-$. Then $\cl(\tilde L^+)$ and $\cl(\tilde L^-)$ bound a convex set $C$ with $\pt C=\tilde L^+ \cup \Lambda_\rho \cup \tilde L^-$. Lemma~\ref{duallimset} implies that at every point of $\Lambda_\rho$, $C$ has a unique supporting plane. Since $\tilde L^\pm$ are strictly convex, we have

\begin{crl}
\label{ksurfunion}
$\pt C$ is $C^1$ and touches $\A^3$ along $\Lambda_\rho$. 
\end{crl}

\subsubsection{Construction of the blow-up}
\label{sec:blowupconstr}

The space $\tilde{\mc P}_0$ is a real vector space and $\tilde{\mc P}_0^s$ is an open cone in it. The representation $\theta_{\rho_\circ}$ fixes a point in $\A^3$, which we assume to be $o$.
Let $o_-=(\rho_\circ, \tilde f) \in \tilde{\mc P}_-$ be the configuration with $\tilde f(\tilde V)=o$. We have the identification $T_{o}\A^3 \cong \R^{2,1}$. This produces an identification $T_{o_-}\tilde{\mc P}_- \cong \tilde{\mc P}_0$. 

\begin{lm}
\label{weaktransition}
Let $x_t=(\rho_t, \tilde f_t): [0, 1] \ra \tilde{\mc P}_-$ be a $C^1$-curve with $x_0=o_-$ and $\dot x_0=(\tau, \tilde f)\in\tilde{\mc P}_0^w$. Then for all small enough $t>0$ we have $x_t \in \tilde{\mc P}_{-}^w$. 
\end{lm}

\begin{proof}
Let $p_t: [0,1] \ra \A^3$ be a $C^1$-curve with $p_0=o$ and $\dot p_0 \in \tilde\Omega^+_\tau$. Due to Corollary~\ref{convdom}, $g_t(\RP^3\backslash \tilde\Omega_{\rho_t})$ converge to $\RP^3\backslash\tilde\Omega_\tau$. Thereby, $g_t p_t \in g_t\tilde\Omega_{\rho_t}$ for all small enough $t$, which is equivalent to $p_t \in \tilde\Omega_{\rho_t}$.
From this we see that for all small enough $t$, for each $v \in V$ we have $\tilde f_t(v) \in \tilde \Omega_{\rho_t}$.
\end{proof}

Our main technical result is

\begin{lm}
\label{transition}
Let $x_t=(\rho_t, \tilde f_t): [0, 1] \ra \tilde{\mc P}_-$ be a $C^1$-curve with $x_0=o_-$ and $\dot x_0=(\tau, \tilde f)\in\tilde{\mc P}_0^s$. Then for all small enough $t>0$ we have $x_t \in \tilde{\mc P}_{-, sp}^s$ and the face celluation of $\tilde f_t$ is a subdivision of the face celluation of $\tilde f$.
\end{lm}

Note that we do not mean that the face celluations of $\tilde f_t$ are weakly equivalent for all small enough $t$.

\begin{proof}
By the same argument as in the proof of Lemma~\ref{weaktransition}, only using Corollary~\ref{convdom2} instead of Corollary~\ref{convdom}, we get that for all small enough $t$, for each $v \in V$ we have $\tilde f_t(v) \in \tilde \Omega^+_{\rho_t}$.
It remains to show that for all small enough $t$, $\tilde f_t$ is in a strictly convex position, is strictly polyhedral and that its face celluation is a subdivision of the face celluation of $\tilde f$. We first describe the proof idea. Suppose for simplicity that the face celluation of $\tilde f$ is a triangulation. Then this triangulation together with the positions of points $\tilde f_t(\tilde V)$ allow us to define a simplicial surface $F_t$. We will first observe that, provided $t$ is small enough, $F_t$ is locally convex. We will need then to show that $F_t$ is globally convex, i.e., $F_t \subset \clconv(F_t)$, or, equivalently, every locally supporting plane is globally supporting. We will use some ideas of Stoker, who showed that a closed smooth locally convex surface in $\R^3$ is necessarily globally convex, see~\cite{Sto}. However, this fact is notably false for non-closed surfaces. But what helps us is that, provided $t$ is small enough, $F_t$ is spacelike for anti-de Sitter geometry, which then restricts its global behavior.

Now we pass to the details.
Consider the Minkowski chart. We perturb it slightly so that $\cl(\tilde\Omega^+_{\rho_\circ})$ is contained in the domain of the perturbed chart. In what follows we will consider only small enough $t$ so that $\cl(\tilde\Omega^+_{\rho_t})$ is also contained there, which is possible, since $\cl(\tilde\Omega^+_{\rho_t}\cup\tilde C_{\rho_t})$ are properly convex sets converging to $\cl(\tilde\Omega^+_{\rho_\circ})$ as $t \rightarrow 0$, due to Lemmas~\ref{convtemp1-} and~\ref{convtemp2-}. We consider the perturbed chart as a vector space with the orientation induced from $\RP^3$, and equip it with a Euclidean metric. 
For a plane $\Pi$ we say an \emph{orientation of $\Pi$} for a choice which half-space with respect to $\Pi$ to call positive, and which to call negative. For an oriented triple of distinct points $p_1, p_2, p_3 \in \Pi$, we say that their order \emph{induces an orientation on $\Pi$}, by calling positive the half-space towards which the vector $(p_2-p_1) \times (p_3-p_1)$ points. Here and in what follows the cross product and the scalar product are Euclidean.


We make two observations. First, we claim that for every quadruple $\tilde v_1, \tilde v_2, \tilde v_3, \tilde v_4 \in \tilde V$ such that $\tilde f(\tilde v_1)$, $\tilde f(\tilde v_2)$, $\tilde f(\tilde v_3)$, $\tilde f(\tilde v_4)$ are affinely independent, for all small enough $t$, the points $\tilde f_t(\tilde v_1)$, $\tilde f_t(\tilde v_2)$, $\tilde f_t(\tilde v_3)$, $\tilde f_t(\tilde v_4)$ are also affinely independent, and $\tilde f_t(\tilde v_4)$ lies in the half-space of the same sign with respect to the plane spanned by $\tilde f_t(\tilde v_1), \tilde f_t(\tilde v_2), \tilde f_t(\tilde v_3)$, as the sign of the half-space with respect to the plane spanned by $\tilde f(\tilde v_1), \tilde f(\tilde v_2), \tilde f(\tilde v_3)$ containing $\tilde f(\tilde v_4)$, where we consider the planes oriented by the order of points $\tilde v_1, \tilde v_2, \tilde v_3$. Indeed, we consider the function
\[\zeta(t):=\langle \tilde f_t(\tilde v_4)-\tilde f_t(\tilde v_1), (\tilde f_t(\tilde v_2)-\tilde f_t(\tilde v_1))\times(\tilde f_t(\tilde v_3)-\tilde f_t(\tilde v_1))\rangle.\]
Note that $\zeta(0)=\zeta'(0)=\zeta''(0)=0$, but 
\[\zeta'''(0):=\langle \tilde f(\tilde v_4)-\tilde f(\tilde v_1), (\tilde f(\tilde v_2)-\tilde f(\tilde v_1))\times(\tilde f(\tilde v_3)-\tilde f(\tilde v_1))\rangle\neq 0.\]
Thus, for all small enough $t$, the function $\zeta(t)$ has the same sign as $\zeta'''(t)$, and our claim holds.

Second, we notice that the plane spanned by $\tilde f_t(\tilde v_1)$, $\tilde f_t(\tilde v_2)$, $\tilde f_t(\tilde v_3)$ converges to the plane passing through $o$ that is parallel in the Minkowski chart to the plane spanned by $\tilde f(\tilde v_1)$, $\tilde f(\tilde v_2)$, $\tilde f(\tilde v_3)$ as $t \ra 0$. Note that this means that for all small enough $t$ such a plane is spacelike for $\A^3$. Further, the union of any two planes that correspond to adjacent faces in the face celluation of $\tilde f$ is future-convex for all small enough $t$.

For the moment, we suppose that the face celluation of $\tilde f$ is a triangulation $\ms T$ of $(\tilde S, \tilde V)$. 
We consider $\tilde S$ oriented so that the positive normals to faces of $\tilde f$ point to the concave side. For $t>0$, to every triangle $T$ of $\ms T$ we associate an oriented plane $\Pi_t(T)$ spanned by the respective points of $\tilde f_t(\tilde V)$. We assume that $t$ is small enough so that (1) for every $T$ and every $\tilde v \in \tilde V$ that is adjacent to at least one vertex of $T$ in $\ms T$, $\tilde f_t(\tilde v)$ is in the negative half-space with respect to $\Pi_t(T)$. Due to the $\theta_{\rho_t}$-invariance, it is enough to check this only for finitely many cases, hence this indeed holds for all small enough $t$ because of the first observation above. Moreover, we set $t$ small enough so that (2) every $\Pi_t(T)$ is spacelike for $\A^3$, does not intersect $\tilde C_{\rho_t}$, and for every $\Pi_t(T)$ and $\Pi_t(T')$ of adjacent $T$ and $T'$, the intersection of the negative half-spaces is future-convex for $\A^3$. Because of the $\theta_{\rho_t}$-invariance, again it is enough to check this for finitely many cases, 
hence it is indeed true for all small enough $t$ because of the second observation above.
We extend $\tilde f_t$ to a simplicial map $F=F_t: \tilde S \ra \A^3$ with respect to $\mc T$. Due to assumption (1), $F$ is a PL-immersion, i.e., is locally injective. (The local injectivity is non-trivial only at vertices, where it means that the links are embedded.) Furthermore, $F$ is locally convex with respect to the orientation, i.e., the Euclidean dihedral angle of every edge, determined by the orientations of the adjacent faces, is less than $\pi$. Next, assumptions (1) and (2) together mean that $F$ is locally future-convex, which implies that for each face the past directions with respect to $\A^3$ are positive. Observe that assumption (1) concerns only with polyhedral geometry, while assumption (2) is about anti-de Sitter geometry. 



Let $\S^2$ be the Euclidean unit sphere (in the perturbed Minkowski chart). Consider the limit set $\Lambda=\Lambda_{\rho_t} \subset \pt \A^3$. Look at the Euclidean Gauss map on $\pt \A^3$, sending a point on $\pt \A^3$ to the exterior unit normal to $\pt \A^3$. Its restriction to $\Lambda$ is a homeomorphism onto a Jordan curve $J \subset \S^2$. Denote the components of $\S^2 \backslash J$ by $J^+$ and $J^-$. Consider the foliations of $\tilde\Omega^\pm_{\rho_t}$ by $K$-surfaces, given by Theorem~\ref{BBZ}. We denote the respective leaves by $\tilde L_K^\pm$. Due to Corollary~\ref{ksurfunion}, the surface $\ol L_K:=\tilde L^+_K \cup \Lambda \cup \tilde L^-_K$ is a strictly convex $C^1$-surface, which touches $\pt \A^3$ along $\Lambda$.
Thereby, the Euclidean Gauss map on $\ol L_K$ also sends $\Lambda$ onto $J$. Assume that the notation $J^+$, $J^-$ is chosen so that the Euclidean Gauss map on $\tilde L^+_K$ has values in $J^+$.

Let $\Pi$ be a spacelike plane disjoint with $\tilde C_{\rho_t}$ that intersects $\tilde\Omega^+_{\rho_t}$. Let $\nu \in \S^2$ be its Euclidean normal, directing in the past with respect to $\A^3$. Then $\Pi$ is tangent to a unique $\tilde L^+_K$ for some $K$. Thus, $\nu \in J^+$. 

Let $(\tilde S, \ms T^*)$ be a $\pi_1S$-invariant celluation of $\tilde S$ dual to $\ms T$. We now apply the Euclidean Gauss map to the parameterized surface $F$. Recall that $F$ is simplicial, immersed and locally convex. We consider its Gauss map as a map $G: (\tilde S, \ms T^*) \ra \S^2$, sending the topological dual cell of each vertex homeomorphically onto the respective geometric cell in $\S^2$, respecting the vertices. Then $G$ is a local homeomorphism. Due to assumption (2) on $t$ and the observation just above, $G$ values in $J^+$. 

Now define $\tilde S^\bdia$ as the abstract union $\tilde S \cup \Lambda$. We define a topology on $\tilde S^\bdia$ using the map $F$: open sets are of the form $(U\cap \Lambda)\cup F^{-1}(U)$ for open $U \subset \RP^3$. Then $\tilde S^\bdia$ is compact. We extend $G$ to $\tilde S^\bdia$ by sending $s \in \Lambda$ to the exterior normal to $\pt\A^3$ at $s$. We claim that this extension is continuous. Indeed, let $p_i \in \tilde S$ converge to $s \in \Lambda$. Up to subsequence, we can assume that $p_i$ belong to the orbit of a single cell $C$ of $\ms T^*$. It is enough to assume that $p_i$ are in the orbit of a vertex of $C$. Each $F(p_i)$ is in the interior of a face of $F$, let $\Pi_i$ be the plane containing the face. All $\Pi_i$ belong to the orbit of a plane $\Pi$. Up to subsequence, $\Pi_i$ converge to a plane $\Pi'$. Since $F(p_i)$ converge to $s$, we have $\Pi' \ni s$. On the other hand, let $\tilde L_K^+$ be the $K$-surface tangent to $\Pi$. Then it is tangent to all $\Pi_i$. By Corollary~\ref{ksurfunion}, any subsequence of $\Pi_i$ can converge only to planes tangent to $\pt\A^3$ at points of $\Lambda$. Hence, $\Pi'$ is the tangent plane to $\pt \A^3$ at $s$. Then Lemma~\ref{subconvtoconv} implies that the initial sequence $\Pi_i$ converges to $\Pi'$. It follows that $G$ is continuous.

Because $G$ extends continuously to $\tilde S^\bdia$, which is compact, and $G(\Lambda) =J$, we see that $G|_{\tilde S}$ is proper as a map to $J^+$. Since $G|_{\tilde S}$ is a local homeomorphism, $G|_{\tilde S}$ is a covering map onto $J^+$. But since $J^+$ is simply connected, it is a homeomorphism. 

We now claim that for every $s \in \Lambda$ and every globally supporting plane $\Pi$ to $F(\tilde S) \cup \Lambda$, its exterior normal $\nu$ is in $J\cup J^-$. Indeed, suppose that $\nu \in J^+$. Then there exists $p \in F(\tilde S)$ such that the parallel plane $\Pi'$ to $\Pi$ at $p$ is locally supporting $F(\tilde S)$. Let $\tilde L^+_K$ be the $K$-surface tangent to $\Pi'$. Due to Corollary~\ref{ksurfunion}, $\Pi'$ belongs to the side of $\nu$ from $\Pi$. Then $p$ belongs to the wrong side of $\Pi$, which is a contradiction.

We now claim that every locally supporting plane to $F(\tilde S)$ is globally supporting. Indeed, pick such a plane $\Pi$ with an exterior normal $\nu \in J^+$. Thus, there is corresponding $q \in \tilde S$ such that $G(q)=\nu$. Suppose that there are points of $F(\tilde S)$ from both sides of $\Pi$. Consider points $p_1$, $p_2\in \cl(F(\tilde S))=F(\tilde S) \cup \Lambda$ on each side from $\Pi$ that are at the maximal Euclidean distance from $\Pi$. The planes that are parallel to $\Pi$ through $p_1$, $p_2$ are globally supporting $F(\tilde S) \cup \Lambda$. Then the exterior normals to these planes are in the opposite directions. One of them is $\nu$, let it be at $p_1$. Then $p_1\notin \Lambda$. But then there is $q_1 \in \tilde S$, $q_1\neq q$, such that $G(q_1)=G(q)=\nu$. It follows that $G$ is not injective, which is a contradiction. In turn, this implies that $F$ is injective and convex, i.e., $F(\tilde S) \subset \pt\clconv(F(\tilde S))=\pt\clconv(F(\tilde V))$, and is equal to a component of $(\pt\clconv(F(\tilde V))) \backslash \Lambda$. Note that it means that $F(\tilde S)$ is future-convex for $\A^3$, so $F(\tilde S)=\Sigma(\tilde f_t)$. In particular, $\tilde f_t$ is strictly polyhedral and its face celluation is $\ms T$.

Now suppose that the face celluation of $\tilde f$ is not simplicial, denote it by $\ms C$. Let $\ms T_1, \ldots, \ms T_r$ be representatives of all weak equivalence classes of $\pi_1 S$-invariant triangulations subdividing $\ms C$. We choose $t$ small enough so that assumptions (1) and (2) work for all triangles of every $\ms T_j$, $j=1,\ldots,r$, but the part of assumption (2) on the adjacent faces is meant only for the adjacent faces along the edges of $\ms C$. For a fixed $t$ and a given non-triangular face $Q$ of $\ms C$ with vertex set $\tilde V_Q$ we look at $\clconv(\tilde f_t(\tilde V_Q))$. Its future-convex part provides a decomposition of $Q$. We do this for a representative of the $\pi_1S$-orbit of every face. We obtain a celluation $\ms C_t$ subdividing $\ms C$. Then the same argument as for the triangulation case shows that $\ms C_t$ is the face celluation of $\tilde f_t$.
\end{proof}

Fix an affine connection on $\mc T$, lift it to a $G$-invariant connection on $\mc R$.  Together with the standard connection on $\R^{2,1}$, this produces an affine connection on $\tilde{\mc P}_-$. Denote by $\mc E: \tilde{\mc P}_0 \ra \tilde{\mc P}_-$ the exponential map. It produces a diffeomorphism between a pierced neighborhood $U$ of $o_-$ in $\tilde{\mc P}_-$ with a lower cone in $T_{o_-} \tilde{\mc P}_-\cong\tilde{\mc P}_0$. We say that this induces a lower-cone structure on $U$ (based at $o_-$). We then make the blow-up on $U$, and glue it with the rest of $\tilde{\mc P}_-$. Denote the resulting manifold with boundary by $\tilde{\mc P}_\vee$. It is independent on the choice of connection on $\tilde{\mc P}_-$. We have an identification $\pt \tilde{\mc P}_\vee \cong \S(\tilde{\mc P}_0)$. Define $\tilde{\mc P}_\vee^w:= \tilde{\mc P}_-^w \cup \S(\tilde{\mc P}_0^w) \subset \tilde{\mc P}_\vee$ and in the same way define $\tilde{\mc P}_\vee^c$, $\tilde{\mc P}_\vee^s$. We have

\begin{lm}
\label{subman}
$\tilde{\mc P}_\vee^w$, $\tilde{\mc P}_\vee^s$ are submanifolds with boundary of $\tilde{\mc P}_\vee$.
\end{lm}

This follows from Lemmas~\ref{weaktransition},~\ref{transition} and 

\begin{lm}
\label{techn5}
In $\R^m$ let $A$ be an open subset of the upper half-space $\{x: x_m>0\}$, and $B$ be an open subset of the hyperplane $\{x: x_m=0\}$. For $x \in B$ let $r_x: [0, +\infty) \ra \R^m$ be the vertical positive ray based in $x$ and parameterized by the $m$-th coordinate. Assume that for every $x \in B$ there exists $\e=\e(x)$ such that for all $t \in (0,\e)$ we have $r_x(t) \in A$. Then $A \cup B$ is a submanifold with boundary of $\R^m$. 
\end{lm}

\begin{proof}
Without loss of generality, we suppose that $A \subset \{x: 0<x_m<1\}$. Define $C:=\{x: 0\leq x_m \leq 1\}$. Define a function $f: B \ra [0,1]$ by
\[f(x):=\inf\{y_m: y\in r_x, y \notin A, y\neq x\}.\]
We have $0<f(x)\leq 1$. It is easy to see that $f$ is lower semi-continuous. Indeed, if $x^i \ra x$ and $a=\liminf f(x^i)$, then, up to subsequence, there exists $y^i$ such that $y^i \in C\backslash A$, $y^i_m \ra a$, $y^i \in r_{x^i}$. Then $y^i \ra y$ such that $y \in C\backslash A$, $y_m=a$ and $y \in r_x$. Hence, $f(x)\leq a$.

This means that for every $x \in B$ there exist a neighborhood $U$ of $x$ in $B$ and $a>0$ such that for all $x' \in U$ we have $f(x')\geq a$. Hence, the subset 
\[\{y: y \in r_{x'}\textrm{~for~some~}x'\in U, 0\leq y_m< a\}\]
belongs to $A \cup B$. Thus, $A \cup B$ is a submanifold with boundary.
\end{proof}

We also define $\tilde{\mc P}_{\vee, sp}^s:= \tilde{\mc P}_{-, sp}^s \cup \S(\tilde{\mc P}_0^s) \subset \tilde{\mc P}_\vee^s$. We can specify Lemma~\ref{subman}.

\begin{lm}
\label{submansp}
$\tilde{\mc P}_{\vee, sp}^s$ is a submanifold with boundary of $\tilde{\mc P}_\vee$.
\end{lm}

The group $G$ acts on $\tilde{\mc P}_-$ smoothly, freely and properly, preserving the bundle structure $\tilde{\mc P}_-\cong \mc R \times (\A^3)^V$. Denote the quotient by $\tilde{\mc P}_-' \cong \mc T \times (\A^3)^V$, and denote the image of $o_-$ by $o_-'$. The subspace $B$ of coboundaries in $T_{\rho_\circ}\mc R$ is naturally isomorphic to $\R^{2,1}$ and is tangent to the $G$-orbits in $\mc R$. Denote by $\tilde{\mc P}'_0$ the quotient of $\tilde{\mc P}_0$ by $B$. Using the exponential map $\mc E':\tilde{\mc P}_0' \ra \tilde{\mc P}_-'$ we define the blow-up $\tilde{\mc P}'_\vee$ with $\pt \tilde{\mc P}'_\vee \cong \S(\tilde{\mc P}_0')$. It is easy to show that the projection $\tilde{\mc P}_\vee \ra \tilde{\mc P}'_\vee$ is a submersion. (Note that $G$ does not act on $\tilde{\mc P}_\vee$, thereby this statement is not immediate, but it is easy to see that the submersions $\tilde{\mc P}_- \ra \tilde{\mc P}_-'$ and $\S(\tilde{\mc P}_0) \ra \S(\tilde{\mc P}_0')$ glue together to a submersion.) The subset $\tilde{\mc P}_-^w \subset \tilde{\mc P}_-$ is $G$-invariant and projects to $\tilde{\mc P}_-^{w'} \subset \tilde{\mc P}_-'$. Similarly, the subset $\tilde{\mc P}_0^w \subset \tilde{\mc P}_0$ is $B$-invariant and projects to $\tilde{\mc P}_0^{w'} \subset \tilde{\mc P}_0'$. Define $\tilde{\mc P}_\vee^{w'}:= \tilde{\mc P}_-^{w'} \cup \S(\tilde{\mc P}_0^{w'}) \subset \tilde{\mc P}_\vee'$, which is a submanifold with boundary.


The group $\pi_1S$ acts on $\tilde{\mc P}_0^{w'}$ by coned maps, freely and properly discontinuously. Hence, the quotient $\mc P_0^w$ is a coned manifold. Next, $\pi_1S$ acts on both $\tilde{\mc P}_-^{w'}$, $\S(\tilde{\mc P}_0^{w'})$ by diffeomorphisms, freely and properly discontinuously. We claim that the action is smooth as the action on $\tilde{\mc P}_\vee^{w'}$. Indeed, let $\phi: \tilde{\mc P}_-' \ra \tilde{\mc P}_-'$ be an action by an element $\gamma\in \pi_1S$. Note that it is smooth, and is a self-diffeomorphism on $\tilde{\mc P}_-^{w'}$. Due to Lemma~\ref{weaktransition}, the exponential map $\mc E'$ sends diffeomorphically some lower cone in $\tilde{\mc P}_0^{w'}$ spanning $\tilde{\mc P}_0^{w'}$ onto a set $X \subset \tilde{\mc P}_-^{w'}$, which then inherits a lower-cone structure based at $o'_-$. We have $\pt X_\vee \cong \S(\tilde{\mc P}_0^{w'})$. Pick $x \in \pt X_\vee$. Since $\phi$ is smooth and fixes $o_-'$, there exists a lower cone $Y \subset X$ such that $\pt Y_\vee \ni x$ and $\phi(Y) \subset X$. Since $\phi$ is smooth on $\tilde{\mc P}_-'$, the map $\grave \phi$ on $Y \times [0, +\infty)$ is smooth. Lemma~\ref{smooth} yields that $\phi$ has a smooth blow-up $\phi_\vee$ on $Y_\vee$. It is easy to check that the restriction of $\phi_\vee$ on $\pt Y_\vee$ coincides with the restriction of the action of $\gamma$ on $\S(\tilde{\mc P}_0^{w'})$. Hence, the action of $\pi_1S$ on $\tilde{\mc P}_\vee^{w'}$ is smooth, and thereby it is a free properly discontinuous action by diffeomorphisms. Denote the quotient by ${\mc P}_\vee^w$. Its interior is identified with $\mc P_-^w$ and its boundary is identified with $\S(\mc P_0^w)$. The map $\tilde{\mc P}_\vee^w \ra \mc P_\vee^w$ is a submersion. We define the subsets $\mc P_\vee^s \subset \mc P_\vee^c \subset \mc P_\vee^w$ in the obvious manner.

\subsection{The intrinsic metric map}
\label{sec:map}

Recall the notion of \emph{intrinsic metric} of spacelike convex surfaces in $\A^3$ from Appendix~\ref{sec:intmet}.

\begin{lm}
\label{intrmetrcone}
Consider $(\rho, f) \in \mc P_-^c$. The intrinsic metric $d$ of $\Sigma(f)$ is a concave hyperbolic cone-metric and $V(d)=V(f)$.
\end{lm}

\begin{proof}
Pick a lift $(\rho, \tilde f) \in \tilde{\mc P}_-^c$. Let $p\in \Sigma(\tilde f)$, $p \notin V(\tilde f)$. We will prove that the intrinsic metric of $\Sigma(\tilde f)$ is locally isometric to $\H^2$ at $p$. Pick a supporting plane $\Pi$ to $\Sigma(\tilde f)$ at $p$ and a compact convex set $K \subset \Pi$ such that $\Sigma(\tilde f)$ is locally a graph over $K$ via the exponential normal map $\mc E_\Pi$ from $\Pi$, that this graph over $K$ does not contain points of $V(\tilde f)$ and that $p \in \inter(K)$. Consider a sequence $\psi_i \subset K$ of closed convex polygonal curves converging to $\pt K$ uniformly so that they bound compact convex sets $K_i \subset K$ with $p \in \inter(K_i)$ for all $i$. There exists a compact convex set $K'$ such that for all $i$ we have $K' \subset K_i$ and $p \in \inter(K')$. 

Let $\psi'_i$ be the polygonal curves in $\A^3$ obtained from connecting the vertices of $\psi_i$ lifted to $\Sigma(\tilde f)$. Let $\Sigma_i$ be the future-convex parts of $\conv(\psi'_i)$. We consider the pull-backs to $K'$ of the intrinsic metrics of the parts of $\Sigma_i$ and $\Sigma(\tilde f)$ that are cut out by $\mc E_\Pi(K')$. Denote the obtained metrics on $K'$ by $d_i$, $d'$. We note that a priori $d'$ might be not the same as the pull-back of $d$, the intrinsic metric of $\Sigma(\tilde f)$, as some of the shortest paths for $d$ with endpoints in $K'$ can escape $K'$. However, due to Lemma~\ref{localize}, there exists a neighborhood $U_p \ni p$ in $K'$ such that for all large enough $i$ the shortest paths for $d_i$ and $d'$ with endpoints in $U_p$ belong to $K'$. Hence, over $U_p$ we have $d'=d$. Furthermore, by construction, there are local isometries $\phi_i: (U_p, d_i) \ra \H^2$. We normalize them so that $\phi_i(p)=q\in \H^2$ for all $i$. Due to Lemma~\ref{uniformconv}, $d_i \ra d$ uniformly on $K'$. 
We apply the Arzel\`a--Ascoli theorem and obtain that, up to subsequence, $\phi_i$ converge uniformly to a local isometry $\phi: (U_p, d) \ra \H^2$. 

Now suppose that $p \in V(\tilde f)$. There exists its neighborhood $U$ in $\Sigma(\tilde f)$ such that no edges pass through $U$ except those that have an end in $p$. A proof is the same as in the hyperbolic case, see~\cite[Lemma 3.10]{Pro2}. We can assume that $U$ is the intersection of $\Sigma(\tilde f)$ with a convex body. Then every point $q \in U$ is connected to $p$ by a segment that belongs to $U$. This means that $\Sigma(\tilde f)$ contains a piece of the boundary of a convex cone based at $p$. Since $p$ does not belong to the relative interior of a segment that belongs to $\Sigma(\tilde f)$, the curvature of this cone is non-zero. Hence, $d$ is a concave hyperbolic cone-metric and $V(d)=V(f)$. 
\end{proof}

The discussion from Section~\ref{sec:bentsurf} allows to consider $d$ as a hyperbolic cone-metric on $(S, V)$, defined up to an element of $\mc H$. Hence, we have the \emph{intrinsic metric map}
\[\mc I_-: \mc P_-^c \ra \mc D_-^c.\]
Moreover, the restriction of $\mc I_-$ to $\mc P_-^s$ has values in $\mc D_-^s$. We will denote it by $\mc I_-^s$ as a map $\mc P_-^s \ra \mc D_-^s$.

\begin{lm}
\label{contin}
The map $\mc I_-$ is continuous.
\end{lm}

\begin{proof}
Pick $x \in \mc P_-^c$, define $d:=\mc I_-(x) \in \mc D_-^c$. Take a geodesic triangulation $\ms T$ of $(S, V, d)$. Since $d$ is CAT(0), all edges are unique shortest paths between the vertices. Let $x_i \ra x$ in $\mc P_-^c$. Pick their lifts $(\rho_i, \tilde f_i)$ converging to a lift $(\rho, \tilde f)$ of $x$ in $\tilde{\mc P}_-^c$. Lemmas~\ref{convconv} and~\ref{convsurf} imply that $\clconv(\tilde f_i) \ra \clconv(\tilde f)$ and $\cl(\Sigma(\tilde f_i))\ra \cl(\Sigma(\tilde f))$. Furthermore, due to Lemma~\ref{convtemp1-}, $\pt\clconv(\tilde f_i)\backslash \Sigma(\tilde f_i) \ra \pt\clconv(\tilde f)\backslash \Sigma(\tilde f)$. Hence, we are in the setting of Lemma~\ref{distconverge0}. Let $\tilde{\ms T}$ be the preimage of $\ms T$ on $\tilde S$. We pick a fundamental domain for $\tilde{\ms T}$. Due to Lemma~\ref{distconverge0}, up to subsequence, for every edge in the fundamental domain, the respective shortest paths on $\Sigma(\tilde f_i)$ converge in the Hausdorff sense to the respective shortest paths on $\Sigma(\tilde f_i)$, and also their lengths converge. This means that $d_i$ are weakly $\ms T$-triangulable and converge to $d$ in $\mc D_-^c$. Due to Lemma~\ref{subconvtoconv}, this holds then for the initial sequence.
\end{proof}

\begin{lm}
\label{c1}
The map $\mc I_-$ is $C^1$ on $\mc P_{-, sp}^s$. 
\end{lm}

The proof is just the same as the proofs of similar statements in other settings~\cite[Lemma 3.33]{Pro2}, \cite[Lemma 2.18]{Pro3},~\cite[Lemma 2.14]{FP}. 

We also have the intrinsic metric map
\[\mc I_0: \mc P_0^c \ra \mc D_0^c.\]
It is clear that $\mc I_0$ is coned. Since the elements of $\mc P_0^c$ are strictly polyhedral and for $(\tau, f) \in \mc P_0^c$ the face celluation of nearby elements of $\mc P_0^c$ is a subdivision of the face celluation of $f$, it is easy to see that it is continuous. We similarly consider its restriction $\mc I_0^s: \mc P_0^s \ra \mc D_0^s$. Then~\cite[Theorem 1.4]{FP} states that

\begin{repthmM}
{minkowski}
The map $\mc I_0^s$ is a $C^1$-diffeomorphism. 
\end{repthmM}

We glue together the maps $\mc I_-$ and $\S(\mc I_0)$ into the map
\[{\mc I}_\vee: \mc P_\vee^c \ra \mc D_\vee^c.\]

Now we show

\begin{lm}
\label{gluedsmooth}
The map ${\mc I}_\vee$ is continuous and is $C^1$ on $\mc P_{\vee, sp}^s$.
\end{lm}

We will need the following lemmas.

\begin{lm}
\label{techn4}
Let $p_t, q_t: [0,1] \ra \A^3$ be two $C^1$-curves with $p_0=q_0=o$ such that $\dot p_0,\dot q_0 \in \R^{2,1}$ are in a spacelike position. Then for all small enough $t$, $p_t$ and $q_t$ are in a spacelike position, and $\dot d_A(p_t,q_t)$ is the Minkowski distance between $\dot p_0$ and $\dot q_0$.
\end{lm}

This is a routine computation using the Taylor expression of the anti-de Sitter metric tensor in the normal coordinates with respect to $o$.

\begin{lm}
\label{techn3}
Let $U \subset \R^{m_1}$ be a domain, $f_1, \ldots, f_r: U \ra \R^{m_2}$ be $C^1$-maps, and $\xi: U \ra 2^{[r]}$ be a function with the following properties\\
(1) if for $x \in U$ we have $\xi(x)=\{i_1, \ldots, i_p\}$, then $f_{i_1}(x)=\ldots=f_{i_p}(x)$ and $df_{i_1, x}=\ldots=df_{i_p,x}$ (the latter means that the differentials coincide as maps);\\
(2) if a sequence $x_i \ra x$ and $j \in \xi(x_i)$ for all $i$, then $j \in \xi(x)$.\\
Define a map $f: U \ra \R^{m_2}$ such that for every $x \in U$ we have $f(x)=f_j(x)$ for $j \in \xi(x)$. (Due to condition (1), this is well-defined.) Then $f$ is a $C^1$-function on $U$.
\end{lm}

\begin{proof}
By considering the coordinate functions, it is enough to verify this for $m_2=1$.
First, we check the case $m_1=1$. Let us see that for every $x$, $f$ is differentiable at $x$ with the derivative $f'_j(x)$, $j \in \xi(x)$. Indeed, for every sequence $t_i \ra 0$, the sequence $x+t_i$ can be divided into finitely many subsequences such that $f(x+t_i)=f_j(x+t_i)$, $j \in \xi(x)$, provided that $t_i$ are small enough. For every such subsequence the limit 
\[\lim_{i \ra +\infty} \frac{f(x+t_i)-f(x)}{t_i}\]
exists and is equal to $f'_j(x)$, which are equal for all $j \in \xi(x)$. Hence, $f$ is differentiable at $x$. The continuity of the derivative is immediate.

For general $m_1$ we show that $f$ is differentiable at $x$ with differential $df_{j, x}$, $j \in \xi(x)$. Indeed, for every differentiable curve $x_t$, $x_0=x$, the function $f(x_t)$ is differentiable at $t=0$ with the derivative given by $df_{j,x}(\dot x_0)$, $j \in \xi(x)$. The continuity of the differentials is immediate. 
\end{proof}

\begin{proof}[Proof of Lemma~\ref{gluedsmooth}.]
Pick $x \in \pt \mc P_\vee^c \cong \S(\mc P_0^c)$ and let $(\tau, f) \in \mc P_0^c$ be a representative of $x$. Let $\ms C$ be the face celluation of $f$. Let $\ms T_1, \ldots, \ms T_r$ be representatives of all weak equivalence classes of triangulations of $(S, V)$ subdividing $\ms C$. For every $(\tau', f')\in \mc P_0^w$ sufficiently close to $(\tau, f)$ and every $j=1,\ldots, r$ the position of $f'(V)$ in $\Omega_{\tau'}$ and the triangulation $\ms T_j$ determine a (possibly non-convex) simplicial Cauchy surface $\Sigma_j(f')\subset \Omega_{\tau'}$. Let $Y$ be a small neighborhood of $(\tau, f)$ so that for all $(\tau', f') \in Y$ we have (1) all the triangles of all $\Sigma_j(f')$ are spacelike, and (2) all the angles of the cells of $\ms C$ in the intrinsic metric of $\Sigma_j(f')$ are smaller than $\pi$. These properties continue to hold for the lower cone spanned by $Y$ in $\mc P_0^w$, which we will denote by $Y$ from now on. Let $\mc I_{0, j}: Y \ra \mc D_0$ be the map sending $(\tau', f') \in Y$ to the intrinsic metric of $\Sigma_j(f')$.

Let $(\tau, \tilde f) \in \tilde{\mc P}_0^c$ be a lift of $(\tau, f)$ and $\tilde Y \subset \tilde{\mc P}_0^w$ be a lift of $Y$ containing $(\tau, \tilde f)$. We possibly reduce $\tilde Y$ to a smaller one, which we will still denote by $\tilde Y$, the restriction of the exponential map $\mc E$ to $\tilde Y$ is a diffeomorphism onto the image $\tilde X \subset \tilde{\mc P}_-^w$ with the following property. For every $(\rho', \tilde f') \in \tilde X$ projecting to $(\rho', f') \in \mc P_-^w$ and for every $j=1,\ldots, r$ let $\Sigma_j(f') \subset \Omega_{\rho'}$ be the simplicial surface determined by $f'(V)$ and $\ms T_j$. We require $\tilde Y$ to be small enough so that all such $\Sigma_j(f')$ have only spacelike triangles. The subset $\tilde X$ inherits a lower-cone structure based at $o_-$, hence it has a blow-up $\tilde X_\vee$. The maps $\mc I_{0, j}$ lift to the maps $\tilde{\mc I}_{0, j}: \tilde Y \ra \mc D_0$. Let $\tilde {\mc I}_{-, j}: \tilde X \ra \mc D_-$ be the map sending $(\rho', \tilde f') \in \tilde X$ to the intrinsic metric of $\Sigma_j(f')$. The map $\mc D^\sharp_\vee \ra \mc D_\vee$ is a covering. We can assume that $\tilde Y$ is so small that the images of all $\tilde I_{-,j}$, $\S(\tilde I_{0,j})$ belong to a simply-connected subset of $\mc D_\vee$. We fix a lift of this subset to $\mc D^\sharp_\vee$ and using this lift we consider $\tilde I_{-,j}$, $\tilde I_{0,j}$ valued in $\mc D_-^\sharp$, $\mc D_0^\sharp$. Identify $\mc D_-^\sharp(\ms T_j)$ with the cone $\Phi^{\ms T_j} \subset \R^{E(\ms T_j)}$, and consider $\tilde{\mc I}_{-, j}$ valued there. The corresponding map $\grave{\tilde{\mc I}}_{-, j}$ on $\tilde X \times [0,1)$ is smooth, thereby by Lemma~\ref{smooth} the map $\tilde{\mc I}_{-, j}$ has a smooth blow-up $\tilde{\mc I}_{\vee, j}: \tilde X_\vee \ra \Phi^{\ms T_j}_\vee$. From Lemma~\ref{techn4}, its restriction to $\pt \tilde X_\vee \cong \S(\tilde Y)$ coincides with $\S(\tilde{\mc I}_{0, j})$, where $\pt\Phi^{\ms T_j}_\vee$ is identified with $\S(\mc D_0^\sharp(\ms T_j))$.

In particular, this means that the intrinsic angles of the triangles of the surfaces $\Sigma_j(f_i)$, $(\rho, f_i) \in X$, where $X$ is the projection of $\tilde X$ to $\mc P_-^w$, converge to those of $\Sigma_j(f)$ as $(\rho, f_i) \ra x$. This implies that, possibly after again reducing $\tilde Y$ and $\tilde X$, the condition (2) on the angles of $\ms C$ is also true for all $\Sigma_j(f')$, $(\rho', f') \in X$. 
Pick an arbitrary triangulation $\ms T$ from $\ms T_1, \ldots, \ms T_r$. Condition (2) means that the intrinsic metrics of all $\Sigma_j(f')$, $(\rho', f') \in X$, are weakly $\ms T$-triangulable. From now on we consider $\tilde{\mc I}_{-, j}$, $\tilde{\mc I}_{0, j}$ valued in $\mc D_-^\sharp(\ms T)$, $\mc D_0^\sharp(\ms T)$, which we identify with the cone $\Phi^{\ms T} \subset \R^{E(\ms T)}$. Then we can consider the blow-ups $\tilde{\mc I}_{\vee, j}$ valued in $\Phi^{\ms T}_\vee$. Now we define $X_\vee:=X \cup \S(Y) \subset \mc P_\vee^w$ and define maps ${\mc I}_{\vee, j}:  X_\vee \ra \Phi^{\ms T}_\vee$ by means of the commutative diagram
\begin{center}
\begin{tikzcd}
\tilde X_\vee \arrow[r,"\tilde{\mc I}_{\vee,j}"] \arrow[rd]
& \Phi^{\ms T}_\vee  \\
& X_\vee \arrow[u,"{\mc I}_{\vee,j}"]
\end{tikzcd}
\end{center}

The maps ${\mc I}_{\vee,j}$ are smooth. Define $X_{\vee,j} \subset X_\vee$ as the subset of those $(\rho', f') \in X$ and the classes of those $(\tau', f') \in Y$ that $\Sigma_j(f')$ is convex. In such case $\Sigma_j(f')$ coincides with $\Sigma(f')$. Let $\xi: X_\vee \ra 2^{[r]}$ be the corresponding partition function of $X_\vee$. At the common points, the values of ${\mc I}_{\vee,j}$ coincide, and coincide with ${\mc I}_{\vee}$. The differentials of ${\mc I}_{\vee,j}$ also coincide at the common points. For the Minkowski case this is Claim 2 in the proof of~\cite[Lemma 2.13]{FP}, for the anti-de Sitter case the proof is just the same. Clearly, the partition function satisfies condition (2) from Lemma~\ref{techn3}. Thus, if $x \in \S(\mc P_0^s)\cong \pt\mc P_{\vee, sp}$, then Lemma~\ref{techn3} implies that ${\mc I}_\vee$ is $C^1$ on $X_\vee$, particularly it is $C^1$ at $x$. Otherwise, ${\mc I}_\vee$ is continuous at $x$. This finishes the proof.
\end{proof}

Now we obtain

\begin{replmA}
{ml}
The differential of $\mc I_\vee$ is non-degenerate on $\pt\mc P_\vee^s$.
\end{replmA}

\begin{proof}
Due to Theorem~\ref{minkowski}, the vectors tangent to $\pt\mc P_\vee^s$ do not belong to the kernel of $d\mc I_\vee$. For every $x \in \pt\mc P_\vee^s$ we now need to check any tangent vector at $x$ transversal to $\pt\mc P_\vee^s$. 

For $\tilde y \in \tilde{\mc P}_0^s$ projecting to $y\in\mc P_0^s$ let $\tilde y_t$, $t \in [0,1]$, be the curve $t\tilde y$ in $\tilde{\mc P}_0$. Let $\tilde x_t:=\mc E(\tilde y_t)$. From Lemma~\ref{transition}, for all small enough $t>0$ we have $\tilde x_t \in \tilde{\mc P}_-^s$. Let $x_t \in \mc P_\vee^s$ be its projection. This is a curve emanating from $x_0=x \in \pt \mc P_\vee$ with a non-zero tangent vector transversal to $\pt\mc P_\vee^s$. Suppose that the face decomposition of $y$ is a triangulation $\ms T$. From Lemma~\ref{transition} $\ms T$ is the face decomposition of $x_t$ for all small enough $t$. Consider the map $\phi_\vee^\ms T \circ \mc I_\vee$ on $x_t$ with values in the cone $\Phi^{\ms T} \subset \R^{E(\ms T)}$. This gives a curve $d_t$, $t>0$. From Lemma~\ref{techn4}, it extends at $t=0$ to a curve $d_t$ with $\dot d_0=\phi_{0}^\ms T\circ \mc I_0(y)$. Thereby, $d{\mc I}_\vee(\dot x_0) \neq 0$.

When the face decomposition of $y$ is not a triangulation, one has to use the maps $\mc I_{-, j}$ from the proof of Lemma~\ref{gluedsmooth}. The argument from the paragraph above shows that $d{\mc I}_{\vee,j}(\dot x_0)\neq 0$ for all $j$. By using Lemma~\ref{techn3}, we deduce that this is true for ${\mc I}_\vee$ restricted to the curve $x_t$.
It follows that $d {\mc I}_\vee$ is injective at $x$. 
\end{proof}

\section{Proof}
\label{sec:proof}

Denote by $\mc I_\vee^s:\mc P_\vee^s \ra \mc D_\vee^s$ the restriction of $\mc I_\vee$. Note that $\mc I_\vee^s$ is the main object of our study, but for several steps of the proof we will need to employ $\mc I_\vee$.
In the next section we will prove

\begin{replmA}
{ml2}
The map $\mc I_\vee^s$ is proper.
\end{replmA}

From this we can establish Theorem~\ref{main}.

\begin{proof}[Proof of Theorem~\ref{main}]
From Lemma~\ref{contin} and Lemma~\ref{ml2}, $\mc I_\vee^s$ is a proper continuous map between manifolds with boundary $\mc P_\vee^s$ and $\mc D_\vee^s$, which have the same dimension. By construction, $\mc I_\vee^s$ sends the boundary to the boundary. Hence $\mc I_\vee^s$ has a well-defined degree. From Lemma~\ref{gluedsmooth} and Lemma~\ref{ml}, $\mc I_\vee^s$ is $C^1$ at $\pt\mc P_\vee^s$ and has a non-degenerate differential. Furthermore, from Theorem~\ref{minkowski}, when restricted to $\pt\mc P_\vee^s$, it is a diffeomorphism onto $\pt\mc D_\vee^s$. Thereby, it has degree one, which implies the surjectivity.

Next, every $y \in \pt\mc D_\vee^s$ has a neighborhood $U_\vee(y)$ in $\mc D_\vee^c$ such that every element of $U_\vee(y)$ has only one $\mc I_\vee^s$-preimage. Indeed, otherwise there exists a sequence $y_i$ converging to $y$ such that every $y_i$ has at least 2 preimages. Pick two of them for each $y_i$, denote them by $x_i$ and $x'_i$. From Lemma~\ref{ml2}, up to subsequence, they converge to $x, x' \in \mc P_\vee^s$. Due to continuity of $\mc I_\vee^s$ and Theorem~\ref{minkowski}, we have $x=x' \in \pt\mc P_\vee^s$. But Lemma~\ref{ml} and the inverse function theorem imply that $\mc I_\vee^s$ is locally injective around $x=x'$, which is a contradiction.

The union of all $U_\vee(y)$ for $y \in \pt\mc D_\vee^s$ is an open set $U_\vee$ such that every element of $U_\vee$ has only one $\mc I_\vee^s$-preimage. It follows that $U:=U_\vee \cap \mc D_-^s$ is the desired set.
\end{proof}

\begin{rmk}
We note that it is natural to call $U$ a ``neighborhood of zero'' in $\mc D_-^s$. It is helpful, however, to distinguish between ``strong'' and ``weak'' neighborhoods of zero. Consider the space $\mc D_-^\bullet$, obtained from $\mc D_\vee^s$ by contracting the boundary to a single point $\bullet$, ``the origin'', endowed with the quotient topology. We call this topology \emph{strong}. In the sense of this topology, the set $U \cup \{\bullet\}$ is indeed a neighborhood of $\bullet$ in $\mc D_-^\bullet$.

There is, however, a natural weaker topology on $\mc D_-^\bullet$. Recall the atlas $\{\phi^\ms T_-\}$ on $\mc D_-^s$ from Section~\ref{sec:metr}, given by all triangulations $\ms T$ of $(S, V)$. Every map $\phi^\ms T_-$ can be naturally extended to $\bullet$ by sending it to the origin of $\R^{E(\ms T)}$. We now call $Y \subset \mc D_-^\bullet$ open if and only if for every $\ms T$ the intersection of $Y$ with the domain of $\phi^\ms T_-$ (extended to $\bullet$) is open as a subset of the image of $\phi^\ms T_-$ in the induced topology from $\R^{E(\ms T)}$. We call the obtained topology on $\mc D_-^\bullet$ \emph{weak}. One can observe that the weak topology is indeed strictly weaker than the strong topology. To this purpose one may consider a single cone $C$ in a vector space $X$ with the origin $o$. The topology on $C \cup \{o\}$ induced from $X$ is strictly weaker than the topology obtained from contracting the boundary of $C_\vee$.

It would be tempting to interpret the uniqueness part of Theorem~\ref{main}, e.g., as that there exists $r=r(h)$ such that if for $d \in \mc D_-^s$ its diameter is $<r$, then the realization of $d$ is unique. However, this is not true, because the sets of metrics with diameter less than given $r$ form a base of the neighborhoods of $\bullet$ in the weak topology, but not in the strong topology.
\end{rmk}

\section{Properness}
\label{sec:properness}

The goal of this section is to prove Lemma~\ref{ml2}. Let us reformulate it as

\begin{lm}
\label{ml2'}
Let $x_i \in \mc P_\vee^s$ be a sequence such that $y_i:={\mc I}_\vee(x_i)$ converge to $y \in \mc D_\vee^s$. Then, up to subsequence, $x_i$ converge to $x \in \mc P_\vee^s$. 
\end{lm}

The proof is quite different depending on whether $y \in \mc D_-^s$ or $y \in \pt\mc D_\vee^s$ with the latter case being more difficult. We note also that since the behavior of ${\mc I}_\vee$ is completely understood on $\pt\mc P_\vee^s$ due to Theorem~\ref{minkowski}, it is enough to assume that all $x_i \in \mc P_-^s$, hence we can denote them by $(\rho_i, f_i)$. Also then $y_i \in \mc D_-^s$, and we can denote them by $d_i$.


\subsection{Convergence of holonomies away from blow-up}
\label{sec:holonaway}

The goal of this subsection is to obtain

\begin{lm}
\label{ml2.1}
Under the conditions of Lemma~\ref{ml2'}, let $y$ be in $\mc D_-^s$ (and we denote it by $d$). Then, up to subsequence, $\rho_i$ converge to $\rho \in \mc T$.
\end{lm}

For a measured geodesic lamination $\lambda$ on $S$ we denote by $E^+_\lambda: \mc T \ra \mc T$ the right earthquake map and by $E^-_\lambda: \mc T \ra \mc T$ the left earthquake map, see, e.g.,~\cite[Section 7.2]{Mar} for a definition. Recall that the earthquake maps are continuous both with respect to the laminations and the metrics. We will need to employ the Kerckhoff--Thurston earthquake theorem~\cite{Ker}:

\begin{thm}
\label{earthquake}
For every $\rho_0, \rho_1 \in \mc T$ there exist unique $\lambda^+$, $\lambda^-$ such that $E^+_{\lambda^+}(\rho_0)=\rho_1$ and $E^-_{\lambda^-}(\rho_0)=\rho_1$.
\end{thm}

Fix $\rho \in \mc T$. Let $\rho^+=\rho^+(\rho), \rho^-=\rho^-(\rho)$ be the points in $\mc T$ corresponding to the intrinsic metrics of $\pt^+C_\rho$ and $\pt^-C_\rho$. Let $\lambda^+=\lambda^+(\rho)$ and $\lambda^-=\lambda^-(\rho)$ be the bending laminations of $\pt^+C_\rho$ and $\pt^-C_\rho$, introduced in Section~\ref{sec:bentsurf}. Recall the fundamental result of Mess~\cite{Mes}:

\begin{thm}
\label{adsearthq}
We have
$$E^+_{\lambda^+}(\rho_\circ)=\rho^+,~~~~~E^+_{\lambda^+}(\rho^+)=\rho,$$
$$E^-_{\lambda^-}(\rho)=\rho^-,~~~~~E^-_{\lambda^-}(\rho^-)=\rho_\circ.$$
\end{thm}

Recall from Appendix~\ref{sec:intmet} that for a convex Cauchy surface $\Sigma \subset \Omega_\rho$ and $\gamma \in \pi_1S$, $l_\Sigma(\gamma)$ is the infimum of lengths of curves in the free homotopy class of $\gamma$ on $\Sigma$. Similarly, for $\rho \in \mc T$ and $\gamma \in \pi_1S$ define $l_\rho(\gamma)$ to be the length of the closed geodesic in the class of $\gamma$ in the hyperbolic metric on $S$ determined by $\rho$. We can show

\begin{lm}
\label{syst}
For every future-convex Cauchy surface $\Sigma \subset \Omega_\rho$ and every $\gamma \in \pi_1S$ we have $l_\Sigma(\gamma) \leq l_{\rho^+}(\gamma)$.
\end{lm}

\begin{proof}
Pick $\gamma\in\pi_1S$.
By Lemma~\ref{smoothing}, $\Sigma$ can be approximated by smooth strictly future-convex surfaces. Due to Lemma~\ref{lensystconv}, for arbitrary $\e$ there exists such a surface $\Sigma'$ with $l_{\Sigma'}(\gamma)\geq l_\Sigma-\e$. A strictly future-convex surface necessarily belongs to the strict past of $\pt^+C_\rho$. Hence, Lemma~\ref{comparison} implies that $l_{\rho^+}(\gamma)\geq l_{\Sigma'}(\gamma)\geq l_\Sigma-\e$. Since $\e$ is arbitrary, this finishes the proof.
\end{proof}

Now for a metric $d$ on $S$ we define its length function $l_d: \pi_1S \ra \R$ in the obvious way.

\begin{lm}
\label{thurstmetric}
Let $d_i \in \mc D_-$ be a precompact sequence and $\rho_i \in \mc T$ be such that for every $i$ and every $\gamma\in \pi_1S$ we have 
\[l_{d_i}(\gamma)\leq l_{\rho_i}(\gamma).\]
Then the sequence $\rho_i$ is precompact in $\mc T$. 
\end{lm}

\begin{proof}
Since $\{d_i\}$ is precompact, for any $\rho \in \mc T$ we can choose representatives of the metrics $d_i$ on $S$ and a hyperbolic metric $h$ representing $\rho$ so that there exists $C>0$ and the identity map $(S, h)\ra (S, d_i)$ is $C$-Lipschitz. Hence, for every $\gamma\in \pi_1S$ we have 
\[l_{d_i}(\gamma)\leq Cl_{\rho}(\gamma).\]
We can assume that $C>1$. Recall that for $\rho, \rho' \in \mc T$ the asymmetric Thurston distance is defined as
\[d_{{\rm Th}}(\rho',\rho):=\sup_{\gamma \in \pi_1S}\ln\frac{l_{\rho'}(\gamma)}{l_{\rho}(\gamma)},\] see~\cite{Thu2}. Hence, $d_{{\rm Th}}(\rho,\rho_i)\leq \ln C$. Thus, by~\cite{PT}, $\{\rho_i\}$ is precompact.
\end{proof}

Now we have all in hands to prove Lemma~\ref{ml2.1}.

\begin{proof}[Proof of Lemma~\ref{ml2.1}.]
By Lemma~\ref{syst}, for every $i$ and every $\gamma \in \pi_1S$ we have $l_{d_i}(\gamma)\leq l_{\rho^+_i}(\gamma)$, where $\rho^+_i=\rho^+(\rho_i)$. Hence, Lemma~\ref{thurstmetric} implies that $\rho^+_i$ belong to a compact subset of $\mc T$. Thereby, up to subsequence, $\rho^+_i$ converge to $\rho^+ \in \mc T$. Due to Theorem~\ref{earthquake} and the continuity of the earthquake map, the bending laminations $\lambda^+_i=\lambda^+ (\rho_i)$ converge to a measured lamination $\lambda^+$. Due to Theorem~\ref{adsearthq}, $\rho_i$ converge to $\rho \in \mc T$.	
\end{proof}

\subsection{Convergence of holonomies at the blow-up}
\label{sec:convholblow}

Here we show

\begin{lm}
\label{ml2.2}
Under the conditions of Lemma~\ref{ml2'}, let $y$ be in $\pt \mc D_\vee^s$. Then $\rho_i$ converge to $\rho_\circ$.
\end{lm}

First we make a quick excursion into group actions on (metric) trees. Set $\Gamma:=\pi_1S$. A pair of a metric tree $\Psi$ and of a $\Gamma$-action on $\Psi$ by isometries is called a $\Gamma$-tree. We will denote the pair by $\Psi$, assuming implicitly some $\Gamma$-action. A $\Gamma$-tree is \emph{minimal} if it does not contain a proper $\Gamma$-invariant subtree. It is called \emph{small} if the stabilizer of each arc is cyclic. For a $\Gamma$-tree $\Psi$ we denote by $l_\Psi: \Gamma \ra \R$ the \emph{length function} of $\Psi$, i.e., for $\gamma \in \Gamma$, $l_\Psi(\gamma)$ is equal to 
\[\inf_{p \in \Psi} d(p, \gamma p).\]
The $\Gamma$-equivariant isometry class of a small minimal $\Gamma$-tree $\Psi$ is determined by $l_\Psi$~\cite{Sko} (furthermore, finitely many $\gamma$ are enough to distinguish it). We topologize the space of such classes by its embedding into $\R^\Gamma$ via the length functions. Denote the resulting space by $\mc{MT}$. Note that there is a natural $\R_{>0}$-action on $\mc{MT}$ by multiplication. There is a distinguished degenerate tree $\Psi_0 \in \mc{MT}$ consisting of a single point, which serves as an origin. Consequently, $\Psi_0$ is distinguished by that $l_{\Psi_0}=0$.

Given a measured geodesic lamination $\lambda$ on $S$, there is a natural construction of a dual $\Gamma$-tree to $\lambda$, see, e.g.,~\cite[Chapter 11]{Kap}. In particular, the degenerate tree $\Psi_0$ corresponds to the empty lamination. Recall that measured geodesic laminations form a space $\mc{ML}$ with a natural $\R_{>0}$-action. Skora showed in~\cite{Sko}

\begin{thm}
\label{skora}
This construction provides a $\R_{>0}$-equivariant homeomorphism $\mc{ML}\cong\mc{MT}$.
\end{thm}

Consider now $\rho \in \mc R$ and the domain $\tilde \Omega_\rho \subset \A^3$. Recall that we divide its anti-de Sitter boundary into two components, the future-convex one $\pt^+\tilde\Omega_\rho$ and the past-convex one $\pt^-\tilde\Omega_\rho$. They are, however, not spacelike. Let us say that the set of points of $\pt^\pm\tilde\Omega_\rho$ that admit a spacelike supporting plane is the \emph{spacelike part} of $\pt^\pm\tilde\Omega_\rho$. The spacelike part of each component is a $\Gamma$-tree. We denote the one on $\pt^+\tilde\Omega_\rho$ by $\Psi^+(\rho)$. A geometric observation shows that it is dual to $\lambda^-(\rho)$ via the Skora duality, see~\cite{Bel}.


\begin{lm}
\label{treecomparison}
Let $\rho \in \mc T$, $\Sigma \subset \Omega_\rho$ be a future-convex Cauchy surface. Then for every $\gamma \in \pi_1S$, we have $l_{\Psi^+(\rho)}(\gamma) \leq l_\Sigma(\gamma)$.
\end{lm}

\begin{proof}
By Theorem~\ref{BBZ}, there exists a $K$-surface $L_K$ in the strict past of $\Sigma$. From Lemma~\ref{comparison}, $l_{L_K}(\gamma)\leq l_\Sigma(\gamma)$. Also from Lemma~\ref{comparison}, the function $l_{L_K}(\gamma)$ is decreasing in $K$. It follows from the results of Belaraouti~\cite[Theorem 2.5]{Bel} that $l_{\Sigma_K} \ra l_{\Psi^+(\rho)}$ as $K \ra \infty$. This implies that $l_{\Psi^+(\rho)}(\gamma) \leq l_\Sigma(\gamma)$.
\end{proof}

\begin{proof}[Proof of Lemma~\ref{ml2.2}.]

Since $d_i \ra y \in \pt \mc D_\vee^s$, for all $\gamma \in \pi_1S$ we have $l_{d_i}(\gamma) \ra 0$. Thereby, Lemma~\ref{treecomparison} implies that $l_{\Psi^+(\rho_i)}\ra 0$. Hence, $\Psi^+(\rho_i) \ra \Psi_0$ and, due to Theorem~\ref{skora}, $\lambda^-(\rho_i)$ converge to the empty lamination. It follows from Theorem~\ref{adsearthq} that $\rho_i$ converge to $\rho_\circ$.
\end{proof}

\subsection{Cosmological time and canonical decomposition}
\label{sec:ct}

In order to finish the proof of Lemma~\ref{ml2'}, we need an important tool, namely \emph{the cosmological time}. It was introduced in~\cite{AGH}, its significance for the study of Minkowski spacetimes was demonstrated in~\cite{BG, Bon}. For a simultaneous treatment of the cosmological time in Minkowski and anti-de Sitter geometries we refer to~\cite{BB}.

Pick $\rho \in \mc R$. As we discussed in Section~\ref{sec:convholblow}, the spacelike part of $\pt^+\tilde\Omega_\rho$ is a metric tree $\Psi^+$ dual to $\lambda^-$. Denote this spacelike part by $\pt^+_s \tilde\Omega_\rho$. The cosmological time of $\tilde\Omega_\rho$ is a function $\ct_\rho: \tilde\Omega_\rho \ra \R_{>0}$ that assigns to $p \in \tilde\Omega_\rho$ the supremum of lengths of the timelike segments $pq$ with $q \in \pt\tilde\Omega_\rho$ where $\vec{pq}$ is directed to the past of $p$. Note that there exists a unique $q$ such that $pq$ realizes $\ct_\rho(p)$. We denote such $q$ by $\eta_\rho(p)$. By construction of $\tilde\Omega_\rho$, note that $\eta_\rho(p)\in \pt^+\tilde\Omega_\rho$. Furthermore, $\eta_\rho(p) \in \pt^+_s\tilde \Omega_\rho$. See~\cite[Proposition 6.3.7]{BB}. The function $\ct_\rho$ is $\pi_1S$-invariant and projects to a function on $\Omega_\rho$, which we continue to denote by $\ct_\rho$. Similarly for $\tau\in T_{\rho_\circ}\mc R$ one defines $\pt^+_s\tilde\Omega_\tau$, $\ct_\tau$ and $\eta_\tau$ for $\tilde\Omega^+_\tau$. We will denote the $r$-level surface of $\ct_\rho$ or $\ct_\tau$ by $L_r(\rho)$, $L_r(\tau)$ respectively. 
When $\rho$ or $\tau$ is clear from the context, we sometimes write just $\ct$, $\eta$, $L_r$.

Recall that $\pt^-C_\rho$ is the past-convex boundary component of $C_\rho$. It turns out that $\pt^-C_\rho=L_{\pi/2}(\rho)$. The function $\ct_\rho$ is $C^{1,1}$ on the past of $\pt^-C_\rho$, see, e.g.,~\cite[Lemma 4.3]{BS4}. On the other hand, $\ct_\tau$ is $C^{1,1}$ everywhere on $\Omega^+_\tau$, see, e.g.,~\cite[Proposition 3.3.3]{BB}.
We will need the following result of Bonsante~\cite[Theorem 6.7]{Bon}:

\begin{lm}
\label{varylevel0}
Consider $\tau \in T_{\rho_\circ}\mc R$. There exists a neighborhood $U \ni \tau$ in $T_{\rho_\circ}\mc R$ and a continuous map $\Phi: U \times \tilde S \times \R_{>0} \ra \R^{2,1}$ such that\\
(1) for every $\tau' \in U$, $r \in \R_{>0}$, the map $\Phi(\tau',.,r): \tilde S \ra \R^{2,1}$ is a $\theta_{\tau'}$-equivariant map onto $L_r(\tau')$;\\
(2) for every $\tau' \in U$, $p \in \tilde S$, the set $\Phi(\tau', p, \R_{>0})$ is a gradient line of $\ct_{\tau'}$.
\end{lm}

Note that Bonsante does not formulate explicitly these properties of his map, but they follow from his proof. We will also need an anti-de Sitter version of this. The proof of Bonsante mildly uses some special features of Minkowski geometry, thus we give our account of the proof, following the ideas of Bonsante.

\begin{lm}
\label{ctconvbase}
Let $\rho_i \ra \rho$ in $\mc R$ and $p_i \ra p$ in $\A^3$, where $p \in \tilde \Omega_\rho$. Define $\eta_i:=\eta_{\rho_i}$, $\ct_i:=\ct_{\rho_i}$. Then $\eta_i(p_i) \ra \eta(p)$ and $\ct_i(p_i) \ra \ct(p)$. (Note that due to Corollary~\ref{convdombase}, $\eta_i(p_i)$ and $\ct_i(p_i)$ are well-defined for all large enough $i$.)
\end{lm}

\begin{proof} 
Let $C_i$ and $C$ be the sets of points in $\A^3$ that are in causal relation to $p_i$, $p$ respectively. Then $\cl(C_i)$ converge to $\cl(C)$ as subsets of $\RP^3$. Note that $\pt^+\tilde\Omega_\rho\cap C$ is compact. Let $K$ be its compact neighborhood in $\A^3$.
Pass to a subsequence realizing $\limsup\ct_i(p_i)$.
By Lemma~\ref{convtemp2-}, $\cl(\pt^+\tilde \Omega_{\rho_i})$ converge to $\cl(\pt^+\Omega_\rho)$. Hence for all large enough $i$ the points $\eta_i(p_i)$ belong to $K$. Thereby, up to subsequence, they converge to a point $q \in \pt^+\tilde\Omega_\rho$. We get 
\[\ct(p) \geq \limsup \ct_i(p_i).\]

On the other hand, since $\cl(\pt^+\tilde \Omega_{\rho_i})$ converge in the Hausdorff sense to $\cl(\pt^+\tilde \Omega_\rho)$, there exists a sequence $q_i \in \cl(\pt^+\tilde \Omega_{\rho_i})$ such that $q_i$ converge to $\eta(p)$. Thereby,
\[\ct(p) \leq \liminf \ct_i(p_i).\]
Thus, $\lim \ct_i(p_i)=\ct(p)$ and $q=\eta(p)$. The latter means that $\eta_i(p_i) \ra \eta(p)$.
\end{proof}

\begin{lm}
\label{varylevelbase}
Consider $\rho \in \mc R$. There exists a neighborhood $U \ni \rho$ in $\mc R$ and a continuous map $\Phi: U \times \tilde S \times (0, \frac{\pi}{2}) \ra \A^3$ such that\\
(1) for every $\rho' \in U$, $r \in (0, \frac{\pi}{2})$, the map $\Phi(\rho',.,r): \tilde S \ra \A^3$ is a $\theta_{\rho'}$-equivariant map onto $L_r(\rho')$;\\
(2) for every $\rho' \in U$, $p \in \tilde S$, the set $\Phi(\rho', p, (0, \frac{\pi}{2}))$ is a gradient line of $\ct_{\rho'}$.
\end{lm}

\begin{proof}
Pick a smooth $\theta_{\rho}$-invariant Cauchy surface in $\tilde\Omega_\rho$ and parameterize it as the image of a $\theta_{\rho}$-equivariant embedding $\phi(\rho,.): \tilde S \ra \A^3$.
By the Ehresmann--Thurston theorem~\cite[Theorem 1.7.1]{CEG}, there exists a neighborhood $U$ of $\rho$ in $\mc R$ and a smooth map $\phi: U \times \tilde S \ra \A^3$ such that for every $\rho' \in U$ the map $\phi(\rho', .)$ is a $\theta_{\rho'}$-equivariant embedding. Provided that $U$ is small enough, the resulting surfaces are spacelike. Thus, their images in $\Omega_{\rho'}$ are embedded compact spacelike surfaces, hence they are Cauchy. Now for every $\rho' \in U$, $p \in \tilde S$ and $r \in (0, \frac{\pi}{2})$ we define $\Phi(\rho', p, r)$ to be the point on the $r$-level surface of $\ct_{\rho'}$ that is on the same gradient line of $\ct_{\rho'}$ as $\phi(\rho', p)$.  One can check that due to Lemma~\ref{ctconvbase}, the map $\Phi$ is continuous. By construction, it satisfies the desired properties.
\end{proof}

We will need a compactification of some ends of $\mc P_-^c$. Pick a compact neighborhood $U_-$ of $\rho_\circ$ in $\mc T$, let $\tilde U_-$ be its lift to a compact neighborhood of $\rho_\circ$ in $\mc R$. Let $\mc P_-^c(U_-)$ be the subset of $(\rho, f)\in\mc P_-^c$ where $\rho \in U_-$. Let $\tilde{\mc P}_-^c(\tilde U_-)$ be the similarly defined subset of $\tilde{\mc P}_-$. We need to describe a compactification of $\mc P_-^c(U_-)$. First, notice a natural compactification coming from the closure of $\tilde{\mc P}_-^c(\tilde U_-)$ in $\mc R \times (\RP^3)^V$: it consists in adding to $\tilde{\mc P}_-^c(\tilde U_-)$ the configurations $(\rho, \tilde f)$ such that (1) $\tilde f$ is not necessarily injective and values in $\cl(\tilde \Omega^+_\rho)$, where $\cl(.)$ is the closure in $\RP^3$, and (2) $\tilde f$ is in a convex position. This is a compactification of $\tilde{\mc P}_-^c(\tilde U_-)$ and projects to a compactification of $\mc P_-^c(U)$. However, we will need a rougher one. Namely, let $\pt_s^+ \tilde \Omega_\rho$ be the spacelike part of $\pt^+\tilde\Omega_\rho$. To obtain a compactification of $\mc P_-^c(U)$, it is enough to replace (1) by the condition that $\tilde f$ values in $\tilde \Omega^+_\rho \cup \pt^+\tilde C_\rho\cup \pt^+_s\tilde \Omega_\rho$. Heuristically, this is because the action of $\pi_1S$ on $\pt^+\tilde\Omega_\rho$ is not proper and the quotient space is not Hausdorff. Let us add such configurations to $\tilde{\mc P}_-^c(\tilde U_-)$ and denote the obtained topological space by $\tilde{\mc P}_-^\bdia(\tilde U_-)$. Denote the quotient space by ${\mc P}_-^\bdia(U_-)$.

\begin{lm}
\label{compact-}
The space ${\mc P}_-^\bdia(U_-)$ is compact.
\end{lm}

\begin{proof}
It is enough to consider the case of $|V|=1$, so $V=\{v\}.$ 
Consider a sequence $(\rho_i, f_i) \in {\mc P}_-(U_-)$. Lift $\rho_i$ to $\tilde U_-$,  assume that, up to subsequence, they converge to $\rho \in \tilde U_-$. Pick a neighborhood $U$ of $\rho$ and a map $\Phi:U \times \tilde S \times (0, \frac{\pi}{2})$ from Lemma~\ref{varylevelbase}. Pick a compact fundamental domain $D \subset \tilde S$ for the $\pi_1S$-action on $\tilde S$. Then for every $\rho' \in U$, $\Phi(\rho', D,  (0, \frac{\pi}{2}))$ is a fundamental domain for the $\theta_{\rho'}$-action on the past of $\pt^-\tilde C_{\rho'}$ in $\tilde\Omega_{\rho'}$. Denote this domain by $D_{\rho'}$. Using Lemma~\ref{ctconvbase}, one sees that $\cl(D_{\rho_i})$ converge in the Hausdorff sense to $\cl(D_\rho)$ as subsets of $\RP^3$. Pick a representative $\tilde f_i$ of $f_i$ such that $\tilde f_i(v) \in D_{\rho_i}$. Then, up to subsequence, $\tilde f_i(v)$ converge to $\tilde f(v)\in \cl(D_\rho)$. But $(\cl(D_\rho)\backslash \tilde\Omega_\rho)\subset \pt^+_s\tilde \Omega_\rho$, which finishes the proof. 
\end{proof}

Now for $\rho \in \mc R$ we describe the the canonical decomposition of $\tilde\Omega_\rho$, following~\cite[Section 5.4]{BBZ}. More exactly, it is the decomposition of the strict past of $\pt^-\tilde C_\rho$. Every $p \in \pt^+_s\tilde\Omega_\rho$ determines a block defined as the intersection of $\eta_\rho^{-1}(p)$ with the strict past of $\pt^-\tilde C_\rho$. If $p$ corresponds to a vertex of the respective metric tree, the block is called \emph{solid}, otherwise it is called \emph{thin}. For every edge of the metric tree  the union of the respective thin blocks 
is called a \emph{Misner block}. In turn, this projects to a decomposition of $\Omega_\rho$. 

Let $\lambda^-$ be the bending lamination of $\pt^- C_\rho$. From the viewpoint of the geometry of $\pt^- C_\rho$, the isolated components of $\lambda^-$ correspond to the Misner blocks, every non-isolated component of $\lambda^-$ determines a thin block, which does not belong to any Misner block, and the components of $\pt^- C_\rho\backslash \lambda^-$ correspond to the thick blocks.

\subsection{Convergence of marked points away from blow-up}
\label{sec:convmp1}

In this subsection we prove Lemma~\ref{ml2'}, provided that $y \in \mc D_-$ (and hence we denote it by $d$).
The main technical tool is the following result.


\begin{lm}
\label{ctinf}
For every $\e>0$ and every compact set $U \subset \mc T$ there exists $\alpha>0$ such that for every future-convex Cauchy surface $\Sigma \subset \Omega_\rho$, $\rho \in U$, if $\sys(\Sigma)\geq \e$, then 
\[\inf_{p \in \Sigma} \ct_\rho(p) \geq \alpha.\]
\end{lm}

Note that when $U$ is a point, this is the anti-de Sitter part of~\cite[Theorem 3.5]{BBZ}. There seem to be small inaccuracies in their proof in the anti-de Sitter case, which, however, are easy to fix. Namely,~\cite[Proposition 6.1]{BBZ} relies on the fact that, for a given $\rho\in \mc R$, the maps $\zeta_{r_1,r_2}:L_{r_1} \ra L_{r_2}$, $r_1>r_2$, along the gradient flow of $\ct$ are 1-Lipschitz. The paper~\cite{Bon} is cited, which, however, tackles only the Minkowski situation. This claim is actually wrong in anti-de Sitter geometry, as one can see by considering a Misner block, for which the metric can be written explicitly and the tangent vectors transverse to the foliation expand in the wrong direction. Nevertheless, the correction is

\begin{lm}
\label{lipscontrol}
The map $\zeta_{r_1,r_2}$ is $\cos^{-1}(r_1)$-Lipschitz.
\end{lm}

This is shown in~\cite[Proposition 6.13]{Bel}. This helps us with a corrected version of~\cite[Proposition 6.1]{BBZ}.

\begin{lm}
\label{ctcompare}
Let $0<r<\pi/2$, $\chi: [0,1] \ra \Omega_\rho$ be a spacelike rectifiable curve in the past of $L_r$ and $\chi'$ be its projection to $L_r$ along the gradient flow of $\ct$. Then \[\len(\chi) \leq \len(\chi')\cos^{-1}(r);\] 
\[|\ct(\chi(0))-\ct(\chi(1))|\leq \len(\chi')\cos^{-1}(r).\]
\end{lm}

The next lemma is~\cite[Proposition 6.2]{BBZ}. Despite its proof in~\cite{BBZ} relies on~\cite[Proposition 6.1]{BBZ}, which, as we mentioned, should be corrected, the proof of~\cite[Proposition 6.2]{BBZ} is correct as it is because it uses~\cite[Proposition 6.1]{BBZ} only inside solid and thin blocks, where the claim of~\cite[Proposition 6.1]{BBZ} actually holds as stated.

\begin{lm}
\label{timeestimate}
Let $0<r<\pi/2$, $\chi: [0,1] \ra \Omega_\rho$ be a spacelike rectifiable curve in the past of $L_r$ and $\chi'$ be its projection to $L_r$ along the gradient flow of $\ct$. Assume that $\chi$ belongs to a single block of the canonical decomposition. Then for an absolute constant $C>1$ we have
\[ C^{-1}\exp^{-1}(\len(\chi')) \leq \frac{\ct(\chi(0))}{\ct(\chi(1))} \leq C\exp(\len(\chi')),\]
\[\len(\chi)\leq C\len(\chi')\exp(\len(\chi'))\ct(\chi(0)).\]
\end{lm}

From these we deduce

\begin{lm}
\label{sysct}
Let $\rho \in \mc T$, $0<r<\pi/2$ and let $\Sigma \subset \Omega_\rho$ be a future-convex Cauchy surface in the past of $L_r$. Denote by $A_r$ the area of $L_r$ and denote by $\sys(\Sigma)$ the systole of $\Sigma$. Let $\delta$ be the diameter of the hyperbolic metric on $S$ given by $\rho^-$. Then for an absolute constant $C>0$ we have
\[\inf_{p \in \Sigma} \ct(p) \geq \frac{C\exp(-\delta)\sys^2(\Sigma)\cos(r)}{A_r\exp\left(\frac{2A_r}{\sys(\Sigma)\cos(r)}\right)}.\]
\end{lm}

A proof is identical to the proof of the counterpart in Minkowski geometry given in~\cite[Section 3.4.2]{FP}, provided that we use Lemma~\ref{ctcompare} instead of~\cite[Lemma 3.21]{FP}, which is different only in multiplication by $\cos^{-1}(r)$. We will also need an expression for $A_r$, see~\cite[p. 188]{BB}. To state it, recall that the length functions of hyperbolic metrics extend from $\pi_1S$ to the length functions over $\mc{ML}$.

\begin{lm}
\label{areaexpr}
We have
\[A_r=-2\pi\sin^2(r)\chi(S)+l_{\rho^-}(\lambda^-)\sin(r)\cos(r).\]
\end{lm}

We can now prove Lemma~\ref{ctinf}.

\begin{proof}[Proof of Lemma~\ref{ctinf}.]
We show that every $\rho \in \mc T$ has a neighborhood $Z$, for which the statement of the lemma holds. First suppose that $\rho \neq \rho_\circ$. Define the width of the convex core $C_\rho$ as the supremum of the lengths of timelike segments inside $C_\rho$. The width is positive if and only if $\rho \neq \rho_\circ$. One can see that by Corollary~\ref{convccorebase}, the width is continuous in $\rho$. Hence, there exists a compact neighborhood $Z$ of $\rho$ over which the width is at least $w_0>0$. Thereby, for $r=\frac{\pi-w_0}{2}$ and for all $\rho' \in Z$ the level surface $L_r(\rho')\subset \Omega_{\rho'}$ belongs to $C_{\rho'}$. Hence, for every $\rho' \in Z$ every future-convex Cauchy surface in $\Omega_{\rho'}$ is in the past of $L_r(\rho')$. Furthermore, note that from Lemma~\ref{areaexpr} and Theorem~\ref{adsearthq} the area of $L_r(\rho')$ is continuous in $\rho'$. Hence, there exists an upper bound on the area of $L_r(\rho')$ over $Z$, and we get the desired result from applying Lemma~\ref{sysct}.

Now we treat $\rho=\rho_\circ$. We claim that for every $r>0$ there exist a compact neighborhood $Z$ of $\rho_\circ$ in $\mc T$ and $\alpha>0$ such that for a future-convex Cauchy surface $\Sigma \subset \Omega_{\rho'}$, $\rho' \in Z$, if the infimum of $\ct_{\rho'}$ over $\Sigma$ is $\leq \alpha$, then the supremum is $\leq r$. Indeed, otherwise for some $r$ there exists a sequence $\rho_i$ converging to $\rho_\circ$ and $\Sigma_i \subset \Omega_{\rho_i}$ such that the infima of $\ct_{\rho_i}$ over $\Sigma_i$ go to zero, but the suprema are at least $r$. Lift the universal covers of $\Sigma_i$ to $\A^3$. From Lemma~\ref{compact-}, one can choose the lifts so that there exist $p_i, q_i \in \tilde\Sigma_i$ such that $p_i \ra o$, and $q_i \ra q \in \tilde \Omega_{\rho_\circ}$. The segment $oq$ is timelike, but is the limit of segments $p_iq_i$, which are spacelike since $\tilde \Sigma_i$ are Cauchy surfaces in $\tilde\Omega_{\rho_i}$. This is a contradiction. 

Thereby, for some neighborhood $Z$ of $\rho_\circ$ and some $r>0$, for every $\rho' \in Z$ every future-convex Cauchy surface in $\Omega_{\rho'}$ with the infimum of $\ct_{\rho'}$ at most $\alpha$ belongs to the past of $L_r(\rho')$ and we can apply Lemma~\ref{sysct} to it, as well as the bound on $A_r$.
\end{proof}


\begin{proof}[Proof of Lemma~\ref{ml2'} for $y \in \mc D_-$.]
By Lemma~\ref{ml2.1}, up to subsequence, $\rho_i$ converge to $\rho$. By Lemma~\ref{ctinf}, there exists $\alpha>0$ such that for every $v \in V$ we have $\ct_{\rho_i}(f_i(v))\geq \alpha$. Hence, from Lemma~\ref{varylevelbase}, up to subsequence, for every $v \in V$ the sequence $f_i(v)$ converges to some $f(v) \in \Omega^+(\rho)\cup\pt^+C_\rho$. We need to see that $f$ is injective. Suppose that for $v\neq w \in V$ we have $f(v)=f(w)$. Lift all to $\tilde{\mc P}_-^w$, suppose that $\tilde v, \tilde w \in \tilde W$ are lifts of $v,w$ such that $\tilde f(\tilde v)=\tilde f(\tilde w)$. For every $i$, the segment $\tilde f_i(\tilde v)\tilde f_i(\tilde w)$ is spacelike. We can pick an arbitrary timelike plane containing $\tilde f_i(\tilde v)\tilde f_i(\tilde w)$. Using the reverse triangle inequality in the timelike plane we see that $d_i(v, w)\leq d_A(\tilde f_i(\tilde v), \tilde f_i(\tilde w))$. Thus $d(v, w) \leq \liminf d_A(\tilde f_i(\tilde v), \tilde f_i(\tilde w))=0$, which is a contradiction, so $\tilde f(\tilde v)\neq \tilde f(\tilde w)$. Thus, $(\rho_i, f_i)$ converge to $(\rho, f) \in \mc P_-^c$. But if $(\rho, f) \notin \mc P_-^s$, then Lemma~\ref{intrmetrcone} implies that $d \notin \mc D_-^s$, which is a contradiction.
\end{proof}

\subsection{Convergence of marked points at the blow-up}
\label{sec:convmp2}

\subsubsection{Compactification at the blow-up}
\label{sec:compact}

We will require a compactification of some ends of $\mc P_\vee^c$, similar to the one for $\mc P_-^c$ in Section~\ref{sec:ct}. First we need to describe a compactification of $\S(\mc P_0^c)$. Pick a compact neighborhood $U_0$ of zero in $T_{\rho_\circ}\mc T$, its compact lift to a neighborhood $\tilde U_0$ of zero in $T_{\rho_\circ}\mc R$ and pick $\alpha>0$. We define $\tilde{\mc P}_0^c(\tilde U_0, \alpha)$ as the subset of $(\tau, \tilde f)\in \tilde{\mc P}_0^c$ where $\tau \in \tilde U_0$ and for all $v \in V$ we have $\ct_\tau(\tilde f(v)) \leq \alpha$. Now to conditions (1) and (2) used in the definition of $\tilde{\mc P}_-^\bdia(\tilde U_-)$ in Section~\ref{sec:ct} we also add (3): if for $v \in V$ we have $\tilde f(v) \in \tilde\Omega^+_\tau$, then $\ct_\tau(\tilde f(v)) \leq \alpha$. This produces $\tilde{\mc P}_0^\bdia(\tilde U_0, \alpha)$. We define $\mc P_0^\bdia(U_0, \alpha)$ to be its quotient.
Then the space $\mc P_0^\bdia(U_0, \alpha)$ is a compactification of ${\mc P}_0^c(U_0, \alpha)$. A proof goes the same way as the proof of Lemma~\ref{compact-}, just instead of Lemma~\ref{varylevelbase} we use Lemma~\ref{varylevel0} and instead of Lemma~\ref{ctconvbase} we use~\cite[Propositions 6.2 and 6.5]{Bon}. 
By applying scaling to $\mc P_0^\bdia(U_0, \alpha)$, we obtain the space $\mc P_0^\bdia$, independent on the choices of $U_0$ and $\alpha$, which is not a compactification of $\mc P_0^c$. However, it is easy to see that $\S(\mc P_0^\bdia)$ is a compactification of $\S(\mc P_0^c)$.  

As in Section~\ref{sec:ct}, let $U_-$ be a compact neighborhood of $\rho_\circ$ in $\mc T$ and $\tilde U_-$ be its lift to $\mc R$. Define $\tilde{\mc P}_\vee^\bdia(\tilde U_-):=\tilde{\mc P}_-^\bdia(\tilde U_-)\cup \S(\tilde{\mc P}_0^\bdia) \subset \tilde{\mc P}_\vee$, ${\mc P}_\vee^\bdia(U_-):={\mc P}_-^\bdia(U_-)\cup \S({\mc P}_0^\bdia)$. We plan to show that ${\mc P}_\vee^\bdia(U_-)$ is compact. To this purpose we need a ``blown-up'' analogue of the argument from Section~\ref{sec:ct}. We will rely on the following elementary fact.

\begin{lm}
\label{rosental}
In $\R^m$, $m \geq 2$, let $x^i$ be a sequence converging to the origin $o$. Then there exists a $C^1$-curve $\chi:[0,1]\ra \R^m$ with $\chi(0)=o$ containing infinitely many of $x^i$. Furthermore, if for a chosen coordinate system the coordinates $x^i_m$ are monotonously decreasing, then we can choose $\chi$ so that the projection $\chi_m$ to the $m$-th axis is monotonously increasing.
\end{lm}

\begin{proof}
A proof basically follows from~\cite[Theorem 3]{Ros} of Rosenthal. We only need to check the second claim. In~\cite{Ros} the author constructs a $C^1$-curve $\chi$ that he calls \emph{primitive}. The definition is inductive. For $m=2$ a curve is primitive at $o$ if it is locally convex. Suppose that $m>2$. Assume that a Euclidean metric is chosen so that the coordinate system is orthogonal. Consider the orthogonal projection $\chi'$ of $\chi$ to the orthogonal plane to the tangent direction at $o$. Then $\chi$ is primitive at $o$ if $\chi'$ is either locally constant or is primitive at $o$. 

Now we pass to our claim, which we prove by induction on $m$. If $m=2$ the claim follows from convexity. Suppose that $m>2$. If $\dot\chi_m(0) \neq 0$, the claim is obvious. If $\dot\chi_m(0)=0$, we pass to the orthogonal projection $\chi'$ to the orthogonal plane to the tangent direction at zero. Clearly, $\chi'$ is not locally constant at zero. By induction, the claim is true for the projection $\chi'_m$ at the $m$-th axis. But $\chi'_m=\chi_m$. 
\end{proof}

Recall that in Section~\ref{sec:blowupconstr} we chose a $G$-invariant affine connection on $\mc R$. Consider the associated exponential map $\mc E_{\mc R}$ from $\rho_\circ$. Assume that it sends homeomorphically a compact neighborhood $\tilde U_0$ of zero in $T_{\rho_\circ}\mc R$ onto $\tilde U_-$.
Let $t_i \in \R_{>0}$ be a sequence converging to zero, $\tau_i$ be a sequence converging to $\tau$ in $T_{\rho_\circ} \mc R$, define $\rho_i:=\mc E_\mc R(t_i\tau_i)$, $g_i:=g_{t_i}$, $\tilde\Omega_i:=\tilde\Omega_{\rho_i}$.

\begin{lm}
\label{domainsec}
The sets $g_{i}\cl(\tilde\Omega_{i})$ converge to $\cl(\tilde\Omega_\tau)$ as subsets of $\RP^3$.
\end{lm}

\begin{proof}
Pass to a subsequence such that $t_i$ decreases monotonously. Pick $\e>0$, define $X:=(\frac{1}{\e}\tilde U_0) \times [0, \e)$. By Lemma~\ref{rosental}, there exists a $C^1$-curve $(\tau_s, t_s): [0,1] \ra X$ with $(\tau_0, t_0)=(\tau, 0)$ such that it contains a subsequence of $(\tau_i, t_i)$ and $t_s$ is an increasing function. Thus, we can reparameterize it as $(\tau_t, t)$. 
Define $\rho_t:=\mc E_\mc R(t\tau_t)$. This is a continuous curve, differentiable at $t=0$ with $\dot\rho_0=\tau$. Now the claim follows from Corollary~\ref{convdom}. However, we proved it up to subsequence, which is not a problem because of Lemma~\ref{subconvtoconv} applied to the space of closed subsets of $\RP^3$ endowed with the topology of Hausdorff convergence.
\end{proof}

By the same argument, using Corollary~\ref{convdom2} instead of Lemma~\ref{convdom}, we show


\begin{lm}
\label{convtemp2+}
The sets $g_i\cl(\pt^+\tilde \Omega_{i})$ converge to $\cl(\pt^+\tilde \Omega_\tau)$.
\end{lm}


Any point $p \in \R^{2,1}$ belongs to $g_t\A^3$ for all small enough $t$. We need now to consider simultaneously the Minkowski metric and the rescaled anti-de Sitter metrics. Let $\xi_0: \R^{2,1} \times \R^{2,1} \ra \R_{\geq 0}$ be the \emph{absolute Minkowski distance function}. It is equal to the spacelike distance on the pairs of points in spacelike Minkowski relation, the timelike distance on the pairs in timelike Minkowski relation and zero on the pairs in lightlike relation and on the diagonal. Similarly, let $\xi_t$ be the absolute distance function of the rescaled anti-de Sitter metric: for $p,q \in (g_t\A^3 \cap \R^{2,1})$, $\xi_t(p,q)$ is equal to the respective length of the segment between $p$ and $q$ that belongs to $\R^{2,1}$, provided that this segment belong to $g_t\A^3 \cap \R^{2,1}$. Hence, $\xi_t$ is defined on the respective subset of $(g_t\A^3 \cap \R^{2,1}) \times (g_t\A^3 \cap \R^{2,1})$ of the pairs of points that can be connected by such segment. We consider all $\xi_t$ as a single function $\xi$ defined on the respective subset $Z \subset [0,1] \times \R^{2,1} \times \R^{2,1}$.

\begin{lm}
\label{timeconv}
The function $\xi$ is continuous on $Z$.
\end{lm}

\begin{proof}
It is enough to check the continuity as $t \ra 0$. For $(t,p,q) \in Z$,
\[\xi_t(p,q)=\frac{\xi_1(g_t^{-1}p, g_t^{-1}q)}{t}.\]
Now the Taylor expression of the anti-de Sitter metric tensor in the normal coordinates with respect to $o$ implies that the right-hand expression converges to $\xi_0(p,q)$ as $(t_i, p_i,q_i) \in Z$ converge to $(0,p,q)$.
\end{proof}

Note that if $p,q \in \R^{2,1}$ are in spacelike (resp. timelike) relation for the Minkowski metric, then for all small enough $t$ they are in spacelike (resp. timelike) relation for the rescaled anti-de Sitter metrics.

Consider $p \in \tilde\Omega^+_\tau$. Due to Lemma~\ref{domainsec}, $p \in g_i\tilde\Omega_{i}$ for all large enough $i$. Define $\ct(p):=\ct_\tau(p)$, $\eta(p):=\eta_\tau(p)$, $\ct^\vee_i(p):=\frac{1}{t_i}\ct_{\rho_i}(p)$, $\eta^\vee_i(p):=g_i\eta_{\rho_i}(p)$. We can now establish a ``blown-up'' analogue of Lemma~\ref{ctconvbase}.

\begin{lm}
\label{ctconv}
Let $p_i \in \R^{2,1}$ be a sequence of points converging to $p\in \tilde\Omega^+_\tau$.
Then $\eta^\vee_i(p_i) \ra \eta(p)$ and $\ct^\vee_i(p_i) \ra \ct(p)$. (Note that $\eta^\vee_i(p_i)$ and $\ct^\vee_i(p_i)$ are well-defined for all large enough $i$.)
\end{lm}

\begin{proof} 
Let $K_p$ be a compact neighborhood of $p$ in $\tilde\Omega^+_\tau$. For all large enough $i$, we have $K_p \subset g_i\A^3$. Denote by $C$ the causal past of $K_p$ in $\R^{2,1}$ with respect to the Minkowski metric and denote by $C_i$ the intersections of the causal past of $K_p$ with respect to the rescaled anti-de Sitter metrics with $\R^{2,1}$. We claim that $\cl(C_i)$ converge to $\cl(C)$ as subsets of $\RP^3$. Indeed, observe that any causal segment with respect to any rescaled anti-de Sitter metric that belongs to $\R^{2,1}$ is timelike for the Minkowski metric. Hence, all $C_i \subset C$. On the other hand, any interior point of $C$ is an interior point of $C_i$ for all large enough $i$. Since $\cl(C)$ is the closure of its interior, we derive the desired claim. Let $K$ be a compact neighborhood of $C \cap \pt^+\tilde \Omega_\tau$.

Pass to a subsequence realizing $\limsup\ct^\vee_i(p_i)$.
For all large enough $i,$ the points $p_i$ belong to $K_p$. 
Thus, up to subsequence, the closures of the intersections of the causal past of $p_i$ with respect to the rescaled anti-de Sitter metrics with $\R^{2,1}$ converge to a subset of $C$. 
From Lemma~\ref{convtemp2+}, $g_i\cl(\pt^+\tilde \Omega_{\rho_i})$ converge to $\cl(\pt^+\Omega_\tau)$. Hence, for all large enough $i$ the points $\eta_i^\vee(p_i)$ belong to $K$. Thereby, up to subsequence, they converge to a point $q \in \pt^+\tilde\Omega_\tau$. From Lemma~\ref{timeconv}, we get 
\[\ct(p) \geq \limsup \ct^\vee_i(p_i).\]

On the other hand, since $g_i\cl(\pt^+\tilde \Omega_{\rho_i})$ converge in the Hausdorff sense to $\cl(\pt^+\tilde \Omega_\tau)$, there exists a sequence $q_i \in g_i\cl(\pt^+\tilde \Omega_{\rho_i})$ such that $q_i$ converge to $\eta(p)$. Thereby, from Lemma~\ref{timeconv},
\[\ct(p) \leq \liminf \ct^\vee_i(p_i).\]
Thus, $\lim \ct^\vee_i(p_i)=\ct(p)$ and $q=\eta(p)$. The latter means that $\eta^\vee_i(p_i) \ra \eta(p)$.
\end{proof}

Now we can establish the ``blown-up'' analogue of Lemma~\ref{varylevelbase}.

\begin{lm}
\label{varylevel}
Consider $\tau \in T_{\rho_\circ}\mc R$. There exists a neighborhood $U \ni \tau$ in $T_{\rho_\circ}\mc R$, $\e>0$ and a continuous map $\Phi: U \times [0, \e) \times \tilde S \times (0, \frac{\pi}{2}) \ra \RP^3$ such that\\
(1.1) for every $\tau \in U$, $t\in(0,\e)$, $r \in (0, \frac{\pi}{2})$, the map $\Phi(\tau,t,.,r): \tilde S \ra \RP^3$ is a $g_{t}\theta_{\mc E_\mc R(t\tau)} g_{t}^{-1}$-equivariant map onto $g_{t}L_{tr}(\mc E_\mc R(t\tau))$;\\
(1.2) for every $\tau \in U$, $t\in(0,\e)$, $p \in \tilde S$, the set $\Phi(\tau, t, p, (0, \frac{\pi}{2}))$ is the $g_t$-image of a gradient line of $\ct_{\mc E_\mc R(t\tau)}$;\\
(2.1) for every $\tau \in U$, $r \in (0, \frac{\pi}{2})$, the map $\Phi(\tau, 0,.,r): \tilde S \ra \RP^3$ is a $\theta_{\tau}$-equivariant map onto $L_r(\tau)$;\\
(1.2) for every $\tau \in U$, $p \in \tilde S$, the set $\Phi(\tau, 0, p, (0, \frac{\pi}{2}))$ is a gradient line of $\ct_{\tau}$.
\end{lm}

With Lemma~\ref{ctconv} established, the proof of this Lemma is basically a repetition of the proof of Lemma~\ref{varylevelbase}. The necessary changes are the following. We apply the Ehresmann--Thurston theorem to the representation variety of $\pi_1S$ into ${\rm PGL}(4, \R)$ to get a neighborhood of $\theta_{\tau}$ and a varying equivariant surface. We also observe that if a surface is spacelike for the Minkowski metric then it is spacelike for all the rescaled anti-de Sitter metrics. (Note that any plane in $\RP^3$ intersects any $g_t\A^3$, which allows to define the notion of spacelikeness even when a surface is not fully contained in $g_t\A^3$.) Modulo these details, the rest of the proof is the same.

Finally, the proof of the next lemma now follows the same way as Lemma~\ref{compact-}, using Lemmas~\ref{ctconv} and~\ref{varylevel} in the appropriate places.

\begin{lm}
\label{compact}
The space $\mc P_\vee^\bdia(U_-)$ is compact.
\end{lm}

\subsubsection{End of the proof}
\label{sec:end1}

In this section we prove Lemma~\ref{ml2'} for $y \in \pt\mc D_\vee^s$, which finishes the proof of Lemma~\ref{ml2'}. We consider a sequence $(\rho_i, f_i) \in \mc P_-^s$ from Lemma~\ref{ml2'}. Pick a compact neighborhood $\tilde U_-$ of $\rho_\circ$ in $\mc R$ projecting to $U_-$ in $\mc T$ as in Section~\ref{sec:compact}. Due to Lemma~\ref{ml2.2}, up to subsequence, $\rho_i$ converge to $\rho_\circ$ in $\mc R$. Hence, we may assume that $\rho_i \in \tilde U_-$. Define $\ct_i:=\ct_{\rho_i}$.

\begin{lm}
\label{ctdown}
Up to subsequence, $(\rho_i, f_i)$ converge in $\mc P_\vee^\bdia(U_-)$ to $x \in \S({\mc P}_0^\bdia)$.
\end{lm}

\begin{proof}
Note first that, up to subsequence, there exists $v \in V$ such that $\ct_{i}(f_i(v)) \ra 0$. Indeed, otherwise from Lemma~\ref{varylevelbase}, up to subsequence, $(\rho_i, f_i)$ converge to $(\rho_\circ, f) \in \mc P_-^w$. Recall from Appendix~\ref{sec:intmet} that $l_{f, v}:\pi_1S \ra \R$ sends $\gamma$ to the infimum of lengths of closed curves in $\Sigma(f)$ based at $f(v)$ in the class of $\gamma$. Lemma~\ref{notalldegen} shows that $l_{f_i,v}\ra l_{f,v}$. Clearly, $l_{f,v}$ values in $\R_{>0}$. On the other hand, if $d_i$ converge to $\pt\mc D_\vee$, $l_{f_i,v}$ must converge to zero pointwise. This is a contradiction.

Let $v \in V$ be such that $\ct_{i}(f_i(v)) \ra 0$. Suppose that there exists $w \in V$ such that $\ct_{i}(f_i(v))$ does not converge to zero. Due to Lemma~\ref{compact-}, up to subsequence, we can pick lifts $(\rho_i, \tilde f_i) \in \tilde{\mc P}_-$ such that $\rho_i \ra \rho_\circ$, $\tilde f_i(v) \ra o$ and $\tilde f_i(w) \ra p \in \tilde \Omega_{\rho_\circ}$. Thereby, the segments $\tilde f_i(v)\tilde f_i(w)$ converge to a timelike segment $op$. However, the segments $\tilde f_i(v)\tilde f_i(w)$ are spacelike. This is a contradiction. 

Hence, for every $v \in V$ we have $\ct_{i}(f_i(v)) \ra 0$. Due to Lemma~\ref{compact-}, this means that, up to subsequence, $(\rho_i, f_i)$ lift to $(\rho_i, \tilde f_i)$ that converge to $o_-$ in $\tilde{\mc P}_-$. Due to Lemma~\ref{compact}, up to subsequence, $(\rho_i, f_i)$ converge to $x \in \mc P_\vee^\bdia(U_-)$. Altogether, this means that $x \in \S({\mc P}_0^\bdia)$.
\end{proof}

Consider $x \in \S({\mc P}_0^\bdia)$ from Lemma~\ref{ctdown}. If $x \in \S(\mc P_0^s)$, we are done. If $x \in \S(\mc P_0^c\backslash \mc P_0^s)$, then due to Lemma~\ref{gluedsmooth}, $y$ must be in $\S(\mc D_0^c\backslash\mc D_0^s)$ and we get a contradiction. Otherwise, pick a representative $(\tau, \tilde f) \in \tilde{\mc P}_0$ of a lift of $x$ to $\S(\tilde{\mc P}_0^\bdia)$. We have a dichotomy. The first option is that $(\tau, \tilde f) \in \tilde{\mc P}_0^w\backslash \tilde{\mc P}_0^c$. Then $\tilde f(V) \subset \Sigma(\tilde f)$, but, since $(\tau, \tilde f) \notin \tilde{\mc P}_0^c$, $\tilde f$ is not injective. The second option is that $(\tau, \tilde f) \notin \tilde{\mc P}_0^w$. By construction, we have $\tilde f(\tilde V) \subset (\tilde\Omega^+_\tau\cup \pt^+_s\tilde\Omega_\tau)$. Since $(\tau, \tilde f) \notin \tilde{\mc P}_0^w$, we get $\tilde f(\tilde V) \cap \pt^+_s\tilde\Omega_\tau \neq \emptyset$.

Recall that in Section~\ref{sec:blowupconstr} we chose an affine connection on $\tilde{\mc P}_-$. Consider the associated exponential map $\mc E: \tilde{\mc P}_0 \ra \tilde{\mc P}_-$. It is a homeomorphism from a neighborhood $\tilde X_0$ of $o_0$ in $\tilde{\mc P}_0$ onto a neighborhood $\tilde X_-$ of $o_-$ in $\tilde{\mc P}_-$. Pick a section $\kappa: \S(\tilde{\mc P}_0) \ra \tilde{\mc P}_0$. Then every $x \in \tilde X_-$ is uniquely represented by a pair $(z_x, t_x)$ so that $z_x \in \tilde{\mc P}_0$ is in the image of $\kappa$, $t_x \in \R_{>0}$ and $\mc E(t_xz_x)=x$. We may assume that our $(\tau, \tilde f)$ is in the image of $\kappa$ and that $(\rho_i, \tilde f_i)=:x_i$ are in $\tilde X_-$. We define $t_i:=t_{x_i}$, $g_i:=g_{t_{i}}$, $\tilde\Omega_i:=\tilde\Omega_{\rho_i}$. Also let $\tau_i$ come from $z_{x_i}$. By construction, $g_i\tilde f_i \ra \tilde f$ as elements of $(\RP^3)^V$.

\begin{lm}
\label{hausdconv}
Let $P_i \subset \tilde \Omega^+_{i}$ and $P \subset \tilde\Omega^+_\tau$ be finite sets such that $g_iP_i$ converge to $P$ in $\RP^3$. Let $\tilde P_i$ and $\tilde P$ be their $\theta_{\rho_i}$- and $\iota_\tau$-orbits respectively. Let $\Sigma_i$ and $\Sigma$ be the future-convex boundary components of $\clconv(\tilde P_i)$ and $\clconv(\tilde P)$. Then $g_i\cl(\Sigma_i)$ converge to $\cl(\Sigma)$ as subsets of $\RP^3$.
\end{lm}

\begin{proof}
From Lemma~\ref{convlimset2}, $g_i\Lambda_{\rho_i}$ converge to $\Lambda_{\rho_\circ}$. The set $\Lambda_{\rho_i}$ is the limit set for $\tilde P_i$. Due to~\cite[Lemma 2.4]{FP}, $\Lambda_{\rho_\circ}$ is the limit set for $\tilde P$. Hence, $\clconv(\tilde P_i)$ converges to $\clconv(\tilde P)$. The rest of the proof goes the same way as the proof of Lemma~\ref{convtemp1-} for $\rho\neq\rho_\circ$.
\end{proof}

In particular, in the both cases $g_i\cl(\Sigma(\tilde f_i))\ra \cl(\Sigma(\tilde f))$. We now deal with the first case of the dichotomy. Then there are $\tilde v\neq \tilde w\in \tilde V$ such that $\tilde f(\tilde v)=\tilde f(\tilde w)$. Let $v, w \in V$ be their projections and $a_i:=d_i(v,w)$. Note that $v\neq w$. Recall that $\xi_{t_i}$ is the rescaled anti-de Sitter absolute distance function. From Lemma~\ref{timeconv}, we have $\xi_{t_i}(g_i\tilde f_i(\tilde v), g_i\tilde f_i(\tilde w)) \ra 0$. By applying the reversed triangle inequality to timelike anti-de Sitter planes containing the segments $\tilde f_i(\tilde v)\tilde f_i(\tilde w)$, we obtain 
\begin{equation}
\label{collapseblow}
\frac{a_i}{t_i} \leq \xi_{t_i}(g_i\tilde f_i(\tilde v),g_i\tilde f_i(\tilde w))\ra 0.
\end{equation}
We need two more ingredients to obtain a contradiction.

\begin{lm}
\label{edgeremains}
Let $e$ be an edge of $\Sigma(\tilde f)$ between $\tilde f(\tilde v_1)$ and $\tilde f(\tilde v_2)$. Let $\tilde V_1, \tilde V_2 \subset \tilde V$ be defined as maximal subsets of $\tilde V$ such that $\tilde f(\tilde V_1)=\tilde f(\tilde v_1)$, $\tilde f(\tilde V_2)=\tilde f(\tilde v_2)$. Then for all large enough $i$ there exist $\tilde v_1' \in \tilde V_1$, $\tilde v_2' \in \tilde V_2$, depending on $i$, such that the segment $\tilde f_i(\tilde v_1') \tilde f_i(\tilde v_2')$ belongs to $\Sigma(\tilde f_i)$. 
\end{lm}

\begin{proof}
Suppose the converse, pass to a subsequence for which the claim does not hold. 
Let $p$ be the midpoint of $e$. Since $e$ is an edge of $\Sigma(\tilde f)$ and $\Sigma(\tilde f)$ is strictly polyhedral, there exists a neighborhood $X$ of $p$ in $\RP^3$ such that $X$ is disjoint from $\clconv(\tilde f(\tilde V \backslash \tilde V_1))\cup\clconv(\tilde f(\tilde V \backslash \tilde V_2))$. Clearly, $g_i(\clconv(\tilde f_i(\tilde V \backslash \tilde V_1))\cup\clconv(\tilde f_i(\tilde V \backslash \tilde V_2)))$ converge to $\clconv(\tilde f(\tilde V \backslash \tilde V_1))\cup\clconv(\tilde f(\tilde V \backslash \tilde V_2))$. Hence $X$ is also disjoint from $g_i(\clconv(\tilde f_i(\tilde V \backslash \tilde V_1))\cup\clconv(\tilde f_i(\tilde V \backslash \tilde V_2)))$ for all large enough $i$. On the other hand, $g_i\cl(\Sigma(\tilde f_i)) \ra \cl(\Sigma(\tilde f))$. Hence for all large enough $i$ there exists $p_i \in X \cap g_i\Sigma(\tilde f_i)$. Then $p_i$ belongs to the convex hull of three points from $g_i(\tilde f_i(\tilde V) \cup \Lambda_i)$. Since the claim does not hold, these three points are either from $g_i(\tilde f_i(\tilde V \backslash \tilde V_1) \cup \Lambda_i)$ or from $g_i(\tilde f_i(\tilde V \backslash \tilde V_2) \cup \Lambda_i)$. In any case, $p_i \in g_i(\clconv(\tilde f_i(\tilde V \backslash \tilde V_1))\cup\clconv(\tilde f_i(\tilde V \backslash \tilde V_2)))$, which is a contradiction. 
\end{proof}

Note that the subsets $\tilde V_1, \tilde V_2$ are finite. By passing to a subsequence, we assume that $\tilde v_1', \tilde v_2'$ from Lemma~\ref{edgeremains} are fixed and denote them from now by $\tilde v_1, \tilde v_2$. Note that the segment $\tilde f_i(\tilde v_1)\tilde f_i(\tilde v_2)$ is geodesic in the intrinsic metric of $\Sigma(\tilde f_i)$. Hence, it projects to a geodesic arc $\chi_i$ in $(S, d_i)$ between some $v_1, v_2 \in V$, which are possibly coinciding. Then $\chi_i$ are representatives of the same class $\gamma$ of arcs on $S \backslash V$ up to isotopy. Let $b_i$ be its length in $d_i$. We now apply Lemma~\ref{timeconv} to $g_i\tilde f_i(\tilde v_1), g_i\tilde f_i(\tilde v_2)$ and see that there is $b>0$ such that

\begin{equation}
\label{edgerembound}
\frac{b_i}{t_i} \geq b.
\end{equation}

We claim that equations~(\ref{collapseblow}) and~(\ref{edgerembound}) are in contradiction with the fact that $y \in \pt \mc D_\vee^s$. The idea is that the latter fact means that for $d_i$ all the metric quantities (lengths of shortest curves in free homotopy classes, distances between marked points, diametes, etc) must go to zero with the same speed. To see this, we need

\begin{lm}
\label{metricblow}
Let $y_t:[0,1] \ra\mc D_\vee^s $ be a $C^1$-curve with $y_0 \in \pt \mc D_\vee^s$ and $y_t=d_t \in \mc D^s_-$ for $t>0$. Then there exists a representative metric $d_0 \in \mc D_0^s$ of $y_0$ such that $\frac{d_t}{t}\ra d_0$ in the Lipschitz sense as $t \ra 0$.
\end{lm}

Indeed, we lift $y_t$ to $\mc D_\vee^\sharp$. We can assume that there is a triangulation $\ms T$ such that $y_t \in \mc D_\vee^\sharp(\ms T)$. Now Lemma~\ref{metricblow} follows from the construction of $\mc D_\vee^\sharp(\ms T)$ in Section~\ref{sec:metr}. The derivatives of the triangle lengths at $t=0$ define $d_0 \in \mc D_0^\sharp(\ms T)$. To construct Lipschitz maps from $\frac{d_i}{t_i}$ to $d_0$, pick a triangle $T$ of $\ms T$. It is a hyperbolic triangle in $d_t$, we realize it on the hyperboloid $\H^2 \subset \R^{2,1}$ and send it radially to the Euclidean triangle in the spacelike plane that subtends the vertices in the realization. Then we send it to the Euclidean triangle of the realization of $T$ in $d_0$ by the respective affine map. This defines Lipschitz maps from $\frac{d_i}{t_i}$ to $d_0$ with Lipschitz constants converging to 1. Lemma~\ref{rosental} yields

\begin{crl}
\label{lipschmetr}
Let $d_i \in \mc D^c_-$ be a sequence converging to $y \in \pt \mc D_\vee^s$ in $\mc D_\vee^s$. Then there exist a representative metric $d \in \mc D_0^s$ of $y$ and a sequence $t_i \ra 0$ such that $\frac{d_i}{t_i}\ra d$ in the Lipschitz sense.
\end{crl}

Corollary~\ref{lipschmetr} implies that there is a sequence $t'_i \ra 0$ such that
\[\frac{a_i}{t_i'} \ra a'>0,~~~~~\frac{b_i}{t_i'} \ra b'>0.\]
The second inequality and~(\ref{edgerembound}) imply that $t'_i/t_i \ra b/b' >0$, while the first and~(\ref{collapseblow}) imply that $t'_i/t_i \ra 0$, which is a contradiction.

Hence, it remains to consider $(\tau, \tilde f) \notin \tilde{\mc P}_0^w$. Then there is $\tilde v \in \tilde V$ such that $\tilde f(\tilde v) \in \pt^+_s\tilde\Omega_\tau$. Note that due to Lemma~\ref{ctconv}, this means that 
\begin{equation}
\label{ctdownblow}
\frac{\ct_i(\tilde f_i(\tilde v))}{t_i} \ra 0.
\end{equation}

We claim that this implies that $\tau \neq 0$. Indeed, otherwise we get a similar contradiction as in the proof of Lemma~\ref{ctdown}. Namely, then we have $\tilde f (\tilde v)=o$. Since $(\tau, \tilde f)$ is a representative of $x \in \S({\mc P}_0)$, we have $(\tau, \tilde f) \neq o_0$. Thereby, there exists $\tilde w \in \tilde V$ such that $\tilde f(\tilde w) \neq o$, i.e., $\tilde f(\tilde w) \in \tilde \Omega^+_\tau$. This means that the segment $\tilde f(\tilde v)\tilde f(\tilde w)$ is timelike for the Minkowski metric. Then for all large enough $i$, $\tilde f_i(\tilde v)\tilde f_i(\tilde w)$ are timelike for the anti-de Sitter metric. This is a contradiction, since they must be spacelike. Hence, $\tau \neq 0$.

Once again we invoke the theory of metric graphs, which we already employed in Section~\ref{sec:convholblow}. Recall that $\pt^+_s\tilde\Omega_\tau$ is the spacelike part of $\pt^+\tilde\Omega_\tau$ and is a metric $\Gamma$-tree for $\Gamma=\pi_1S$. We denote it by $\Psi^+(\tau)$. Denote by $\lambda^-(\tau)$ the measured lamination dual to $\Psi^+(\tau)$ via the Skora duality. See~\cite{BB, BG} for more details. In contrast to the anti-de Sitter case, $\lambda^-(\tau)$ does not have a direct geometric interpretation in $\tilde\Omega_\tau$.

Recall that since $\mc{ML}$ admits a natural PL-structure, but no natural differentiable structure, one can define its tangent spaces, which are not vector spaces, but cones. See details in~\cite{Bon2}.
Let $\lambda_0 \in \mc{ML}$ be the empty measured lamination. Since $\mc{ML}$ also has the structure of a cone based at $\lambda_0$, the tangent space $T_{\lambda_0}\mc{ML}$ can be identified with $\mc{ML}$ itself. We have

\begin{lm}
\label{convlam}
Let $\rho_t: [0,1] \ra \mc R$ be a continuous curve with $\rho_0=\rho_\circ$, differentiable at $t=0$ with $\dot\rho_0=\tau \in T_{\rho_\circ}\mc T$. Then $\dot\lambda^-_0:=\frac{d}{dt}\lambda^-(\rho_t)|_{t=0}=\lambda^-(\tau)$.
\end{lm}

This is basically shown by Bonsante--Schlenker in~\cite[Appendix B]{BS3}. Recall first the notion of infinitesimal earthquake. For every nonzero $\lambda \in \mc{ML}$ the curves $E^\pm_{t\lambda}(\rho_\circ)$ are $C^1$ (in fact, analytic, see~\cite{Ker2}). We denote by $e^\pm_\lambda$ the resulting vector fields on $\mc T$. Due to Theorem~\ref{adsearthq}, we have $\tau=e^-_{\dot\lambda^-_0}(\rho_\circ)$. On the other hand, it is shown in the proof of~\cite[Proposition B.3]{BS3} that $\tau=e^-_{\lambda^-(\tau)}(\rho_\circ)$. Since the earthquake map from $\rho_\circ$ is a PL-homeomorphism, we get $\dot\lambda^-_0=\lambda^-(\tau)$. 

In the same way as in the proof of Lemma~\ref{domainsec} we deduce from Lemma~\ref{rosental}

\begin{lm}
\label{lamblowup}
In our setting, $\frac{1}{t_i}\lambda^-(\rho_i)$ converge to $\lambda^-(\tau)$ in $\mc{ML}$. 
\end{lm}

Since $\tau \neq 0$, the Skora duality, Theorem~\ref{skora}, implies

\begin{crl}
\label{treeblowup}
The trees $\frac{1}{t_i}\Psi^+(\rho_i)$ converge to non-trivial $\Psi^+(\tau)$ in $\mc{MT}$.
\end{crl}

Define $\Sigma_i:=\Sigma(\tilde f_i)$, $\alpha_i:=\inf_{\Sigma_i}\ct_i(p)$, $\beta_i:=\sup_{\Sigma_i}\ct_i(p)$.

\begin{lm}
\label{supblowup}
The sequence $\frac{\beta_i}{t_i}$ is bounded.
\end{lm}

\begin{proof}
Suppose the converse. Then, up to subsequence, $\frac{\beta_i}{t_i}$ increases to infinity. Due to~(\ref{ctdownblow}), $\frac{\alpha_t}{t_i} \ra 0$. Hence for all large enough $i$ and some $\alpha>0$, $\Sigma_i$ intersects $\tilde L_{\alpha t_i}(\rho_i)$, where the latter is a level surface of $\ct_i$. We apply now Lemma~\ref{varylevel} and get a neighborhood $U \ni \tau$ in $T_{\rho_\circ}\mc T$, $\e>0$ and the map \[\Phi: U \times [0,\e) \times \tilde S \times \left(0, \frac{\pi}{2}\right) \ra \RP^3.\]
We define $\phi: U \times [0,\e) \times \tilde S \ra \RP^3$ by $\phi(\tau,t, p):=\Phi(\tau, t, p, \alpha)$. Pick an intersection point of $g_i\tilde L_{\alpha t_i}(\rho_i)$ and $g_i\tilde \Sigma_i$. These points determine a sequence $p_i \in S$ via the map $\phi$. Up to subsequence, they converge to $p \in S$. Pick lifts $\tilde p_i \in \tilde S$ and $\tilde p \in \tilde S$ so that $\tilde p_i$ converge to $\tilde p$.
Then $\phi(\tau_i, t_i, \tilde p_i)$ converge to $\phi(\tau, 0, \tilde p)$.

Denote by $\tilde P_i$ the $\theta_{\rho_i}$-orbit of $g_i^{-1}\phi(\tau_i, t_i, \tilde p_i)$ and by $\tilde P$ the $\iota_\tau$-orbit of $\phi(\tau, 0, \tilde p)$. Denote by $\Sigma'_i$ and $\Sigma'$ the future-convex boundaries of the closed convex hull of $\tilde P_i$ and of $\tilde P$ respectively. Due to Lemma~\ref{hausdconv}, $g_i\cl(\Sigma_i')$ converge to $\cl(\Sigma')$. Let $\beta$ be the supremum of $\ct_\tau$ over $\Sigma'$ and $\beta'_i$ be the suprema of $\ct_i$ over $\Sigma_i'$. We first claim that $\limsup \beta'_i \leq \beta.$

Pass to a subsequence realizing $\limsup\beta_i'$. Pick a point $q_i$ on $\Sigma_i'$ that realizes the supremum of $\ct_i$. We project $q_i$ along the $g_i$-images of gradient lines of $\ct_i$ to $g_i\tilde L_{\alpha t_i}(\rho_i)$ and then to $S$ via the map $\phi$. This gives us a sequence $s_i \in S$. Up to subsequence, it converges to $s \in S$. Let $\tilde s_i \in \tilde S$, $\tilde s\in \tilde S$ be their lifts such that $\tilde s_i$ converge to $\tilde s$. We may assume that $q_i$ project to $\phi(\tau_i, t_i, \tilde s_i)$.

Because $\cl(\Sigma'_i)$ converges to $\cl(\Sigma)$, up to subsequence, $q_i$ converges to $q \in \cl(\Sigma)$. Because the projections of $q_i$ converge to $\phi(\tau, 0, \tilde s)$, we have $q \in \Sigma$. Thus $\limsup \beta'_i \leq \beta$. Note, however, that every $\Sigma'_i$ is in the future of $\Sigma_i$. Hence, $\limsup \frac{\beta_i}{t_i}\leq\limsup \beta'_i\leq\beta$.
\end{proof}

Let $A_i$ be the area of $L_{\beta_i}(\rho_i)$.

\begin{lm}
\label{areabound}
The sequence $\frac{A_i}{t_i^2}$ is bounded.
\end{lm}

\begin{proof}
From Lemma~\ref{areaexpr}, we have
\[A_i=-2\pi\sin^2(\beta_i)\chi(S)+l_{\rho^-(\rho_i)}(\lambda^-(\rho_i))\sin(\beta_i)\cos(\beta_i).\]
The claim follows from Lemma~\ref{lamblowup} and~\ref{supblowup}.
\end{proof}

Now we denote the systole of $d_i$ by $a_i$. 

\begin{lm}
\label{sysblowup}
We have $\frac{a_i}{t_i} \ra 0$.
\end{lm}

\begin{proof}
Suppose the converse, that $\limsup \frac{a_i}{t_i} \geq a>0$. From Lemma~\ref{sysct} we get
\[\alpha_i\geq \frac{Ca_i^2}{A_i\exp\left(\frac{2A_i}{a_i}\right)}.\]
From Lemma~\ref{areabound}, there exists $A>0$ such that for all $i$ we have $\frac{A_i}{t_i^2} \leq A$. Hence, up to subsequence, we get
\[\alpha_i \geq \frac{Ca^2}{A\exp\left(\frac{2At_i}{a}\right)}\geq \frac{Ca^2}{2A}.\]

It follows that $\frac{\alpha_i}{t_i}$ does not converge to zero, which contradicts~(\ref{ctdownblow}).
\end{proof}

%


Now we are ready to define $b_i$ in this case and finish the proof.
Due to Corollary~\ref{treeblowup}, the trees $\frac{1}{t_i}\Psi^+(\rho_i)$ converge to $\Psi^+(\tau)$, which is non-trivial. Thus there exists $\gamma \in \Gamma$ such that $l_{\Psi^+(\tau)}(\gamma)>0$. Define $b_i:=l_{d_i}(\gamma)$. From Lemma~\ref{treecomparison}, it follows that $\frac{1}{t_i}\liminf b_i>0$. Up to subsequence, $\frac{b_i}{t_i} \ra b>0$. On the other hand, from Corollary~\ref{sysblowup}, $\frac{a_i}{t_i} \ra 0$.

Now from Corollary~\ref{lipschmetr}, there exists a sequence $t'_i \ra 0$ such that 
\[\frac{a_i}{t_i'} \ra a'>0,~~~~~\frac{b_i}{t_i'} \ra b'>0.\]
As in the previous case, the second inequality implies that $t'_i/t_i \ra b/b' >0$, while the first implies $t'_i/t_i \ra 0$, which is a final contradiction.

\vskip+0.8cm
\begin{center}
{\LARGE \textit{Part 2}}
\end{center}

The goal of this part is to prove Theorem~\ref{main2}.

\section{Changes in the setup}
\label{sec:changes}

Our proof of Theorem~\ref{main2} follows the same path as the proof of Theorem~\ref{main}. Let us describe the necessary changes in the setup. 

We now have two finite non-empty sets $V^\pm \subset S$. 
As for the spaces of cone-metrics, we first define the spaces $\mc D\mc D_-^\sharp:=\mc D_-^\sharp(V^+) \times \mc D_-^\sharp(V^-)$ and $\mc D\mc D_0^\sharp:=\mc D_0^\sharp(V^+) \times \mc D_0^\sharp(V^-)$. These manifolds have natural atlases with charts valued in cones in $\R^{E(\ms T^+)}\times \R^{E(\ms T^-)}$ for pairs of triangulations $(\ms T^+,\ms T^-)$ of $(S, V^\pm)$ respectively. By blowing-up these cones as in Section~\ref{sec:metr}, we obtain the blown-up space $\mc D\mc D_\vee^\sharp$, which is the union of $\mc D\mc D_-^\sharp\cup\S(\mc D\mc D_0^\sharp)$. The pure braid group $\mc B$ acts on $\mc D\mc D_-^\sharp$, $\mc D\mc D_0^\sharp$ and $\mc D\mc D_\vee^\sharp$. The quotients are denoted by $\mc D\mc D_-$, $\mc D\mc D_0$ and $\mc D\mc D_\vee$. We denote the elements of the first two by $(d^+, d^-)$. As before, we denote the respective subsets corresponding to the concave cone-metrics and to the concave cone-metrics with $V(d^\pm)=V^\pm$ by the superscripts $c$ and $s$ respectively.

As for the spaces of polyhedral surfaces, first we need the concept of \emph{coned bundle} over a manifold. It is defined the same as a vector bundle, but the fiber is isomorphic to a cone in some vector space. A vector bundle minus the zero section is an example of a coned bundle (as our standard convention is to delete the origin from a cone). When we say that we consider a vector bundle as a coned bundle, we mean that we delete the zero section. One can define the spherization and the blow-up of a coned bundle in the same way how it was done for cones in Section~\ref{sec:cone}. Furthermore, one can define the blow-ups of pierced tubular neighborhoods of the zero section, similarly as we defined it for lower-cones.

We mostly denote the elements of $\mc R \times \mc R$ by $\theta$, considered as representations of $\pi_1S$ to $G_-$, though sometimes we write them as $(\rho^l, \rho^r)$. Following our standard convention, we denote the elements of $\mc T \times \mc T$ in the same way. Similarly, we denote the elements of $T\mc R$ or $T\mc T$ by $\iota$, considered as representations of $\pi_1S$ to $G_0$, but sometimes we write them as $(\rho, \tau)$. 

Consider the space
\[\tilde{\mc P}\tilde{\mc P}_0:=T\mc R  \times (\R^{2,1})^{V^+}\times (\R^{2,1})^{V^-}.\]
This is a vector bundle over $\mc R$. (Actually, a trivial one, since $\mc R$ is homeomorphic to a ball.) 
We consider it as a coned bundle. Denote its elements by $(\iota, \tilde f^+, \tilde f^-)$, where $\tilde f^\pm: V^\pm \ra \R^{2,1}$. 

Consider also the space
\[\tilde{\mc P}\tilde{\mc P}_-:=\mc R \times \mc R \times (\A^3)^{V^+}\times (\A^3)^{V^-}.\] We denote its elements by $(\theta, \tilde f^+, \tilde f^-)$, where $\tilde f^\pm: V^\pm \ra \A^3$.
Let $O_-$ be the submanifold of $\tilde{\mc P}\tilde{\mc P}_-$ consisting of all $(\theta, \tilde f^+, \tilde f^-)$ with $\theta=(\rho, \rho)$ for $\rho \in \mc R$ and $\tilde f^+(\tilde V^+)=\tilde f^-(\tilde V^-)=o$. Let $N$ be the subbundle of $T\tilde{\mc P}\tilde{\mc P}_-$ defined over $O_-$ that is tangent to the product of all but the first factor of $\tilde{\mc P}\tilde{\mc P}_-$. Then $N$ complements $TO_-$ in $T\tilde{\mc P}\tilde{\mc P}_-$. There is a natural isomorphism $N \cong \tilde{\mc P}\tilde{\mc P}_0$ as of vector bundles over $\mc R$. We have chosen an affine connection on $\mc R$, which, together with the Levi--Civita connection on $\R^{2,1}$, gives a connection on $\tilde{\mc P}\tilde{\mc P}_-$. Consider the exponential map along $N$. It produces a diffeomorphism between a pierced tubular neighborhood of the zero section in $N$ onto a pierced tubular neighborhood of $O_-$ in $\tilde{\mc P}\tilde{\mc P}_-$. We use it to construct the blow-up 
\[\tilde{\mc P}\tilde{\mc P}_\vee:= \tilde{\mc P}\tilde{\mc P}_- \cup \S(\tilde{\mc P}\tilde{\mc P}_0).\]

We consider domains of discontinuity for $\theta \in \mc R \times \mc R$ or $\iota \in T\mc R$ described in Section~\ref{sec:domain}, but now we use the subscripts $\theta$ or $\iota$ instead of $\rho$ or $\tau$, because now we vary all the holonomy. When $\theta \in \mc R \times \mc R$, we use the respective notation $\tilde C_\theta$ for the convex core of $\tilde\Omega_\theta$, $\tilde\Omega_\theta^\pm$ for the respective connected components of $\tilde\Omega_\theta^\pm \backslash \tilde C_\theta$, etc., following the conventions of Section~\ref{sec:domain}. 

Similarly to Section~\ref{sec:bentsurf} we denote the subspaces of $\tilde{\mc P}\tilde{\mc P}_-$ and $\tilde{\mc P}\tilde{\mc P}_0$ with $\tilde f^\pm(V^\pm) \subset \tilde\Omega_\theta^\pm$ and $\tilde f^\pm(V^\pm) \subset \tilde\Omega_\iota^\pm$ by $\tilde{\mc P}\tilde{\mc P}_-^w$ and $\tilde{\mc P}\tilde{\mc P}_0^w$ respectively. The space $\tilde{\mc P}\tilde{\mc P}_0^w$ is a coned subbundle of $\tilde{\mc P}\tilde{\mc P}_0$. The notions of convex and strictly convex position for maps $\tilde f^\pm$ are defined the same as in Section~\ref{sec:bentsurf} with the only difference that now we need to distinguish between future-convex and past-convex positions. The future-convex position was defined in Section~\ref{sec:bentsurf}. The past convex position for a map $\tilde f^-: V^- \ra \tilde \Omega_\theta$ or $\tilde f^-: V^- \ra \tilde \Omega_\iota$ is defined the same, but we require $\tilde f^-(\tilde V^-) \subset \tilde \Omega_\theta^-$ or $\tilde f^-(\tilde V^-) \subset \tilde \Omega_\iota^-$ respectively. Then a triple $(\theta, \tilde f^+, \tilde f^-)$ is said to be in a (strictly) convex position if $\tilde f^+$ is in (strictly) future-convex and $\tilde f^-$ is in (strictly) past-convex positions. The respective subsets of $\tilde{\mc P}\tilde{\mc P}_-$ and $\tilde{\mc P}\tilde{\mc P}_0$ are denoted by $\tilde{\mc P}\tilde{\mc P}_-^c$, $\tilde{\mc P}\tilde{\mc P}_0^c$, $\tilde{\mc P}\tilde{\mc P}_-^s$ and $\tilde{\mc P}\tilde{\mc P}_0^s$. The space $\tilde{\mc P}\tilde{\mc P}_0^s$ is a coned subbundle. We define the spaces $\tilde{\mc P}\tilde{\mc P}_\vee^w$, $\tilde{\mc P}\tilde{\mc P}_\vee^c$, $\tilde{\mc P}\tilde{\mc P}_\vee^s$ in an obvious way. The first and the third are submanifolds of $\tilde{\mc P}\tilde{\mc P}_\vee$. 

The same argument as in Section~\ref{sec:blowupconstr} shows that after taking the quotients by the action of $G_-$ on $\tilde{\mc P}\tilde{\mc P}_-^w$, of $G_0$ on $\tilde{\mc P}\tilde{\mc P}_0^w$ and of $\pi_1S$ on both of them, we obtain manifolds $\mc P\mc P_-^w$, $\mc P\mc P_0^w$ and $\mc P\mc P_\vee^w=\mc P\mc P_-^w \cup \S(\mc P\mc P_0^w)$, where the latter is a manifold with boundary. The quotients of the subsets with the superscripts $c$ and $s$ are denoted by the same superscripts. The elements of these spaces are triples $(\theta, f^+, f^-)$ or $(\iota, f^+, f^-)$. We denote the respective surfaces defined by $f^\pm$ by $\Sigma(f^\pm)$. We have the respective induced metric maps $\mc I\mc I_-$, $\mc I\mc I_0$, $\mc I\mc I_\vee$ defined on $\mc P\mc P_-^c$, $\mc P\mc P_0^c$ and $\mc P\mc P_\vee^c$. Here, e.g., $\mc I\mc I_\vee: \mc P\mc P_\vee^c \ra \mc D\mc D_\vee^c$. If we consider its restriction to $\mc P\mc P_\vee^s$, then it values in $\mc D\mc D_\vee^s$, and we write $\mc I\mc I_\vee^s$ for the respective map $\mc P\mc P_\vee^s \ra \mc D\mc D_\vee^s$, the same convention holds for the other maps.
We have

\begin{thmM}
\label{minkowski2}
The map
\[\mc I\mc I_0^s: \mc P\mc P_0^s \ra \mc D\mc D_0^s\]
is a $C^1$-diffeomorphism.
\end{thmM}

This is a direct corollary of the results of~\cite{FP}, though it is not formulated there explicitly. Hence, we now derive it from the results of~\cite{FP}. Some ideas were used in~\cite[Section 5]{PS}.

\begin{proof}
The fact that $\mc I\mc I_0^s$ is injective and surjective is a direct reformulation of~\cite[Theorem II']{FP}. The fact that it is $C^1$ can be proven the same as such fact for the map~$\mc I_0^s$, which is~\cite[Lemma 2.14]{FP}. The only fact that requires a bit of work is that its differential is nowhere degenerate. 

Define the space
\[\tilde{\mc P}_0^{+}:= T\mc T \times (\R^{2,1})^{V^+}.\]
There is its subset $\tilde{\mc P}_0^{s,+}$ consisting of the configurations in the strictly future-convex positions. By taking its $G_0$- and $\pi_1S$-quotients we obtain the space $\mc P_0^{s,+}$, which is a manifold. Denote also $\mc D_0^{s,+}:=\mc D_0^{s}(V^+)$. We have an induced metric map
\[\mc I_0^{s, +}:\mc P_0^{s,+} \ra \mc D_0^{s,+}.\]
Theorem~\ref{minkowski} implies that it is a $C^1$-submersion. We similarly define the space $\mc P_0^{s,-}$, however we consider for it the configurations in the strictly \textbf{past-convex} positions. Define $\mc D_0^{s,-}:=\mc D_0^{s}(V^-)$ and an induced metric map 
\[ \mc I_0^{s, -}:\mc P_0^{s,-} \ra \mc D_0^{s,-}.\]
The theory of past-convex embeddings is obtained from the theory of future-convex embeddings by a change of the time-orientation. Hence, Theorem~\ref{minkowski} implies that $\mc I_0^{s, -}$ is also a $C^1$-submersion.

For $d \in \mc D_0^{s,\pm}$ we define $\mc P_0^\pm(d):=(\mc I_0^{s,\pm})^{-1}(d) \subset \mc P_0^{s,\pm}(d)$. Since $\mc I_0^{s,\pm}$ are $C^1$-submersions, these sets are $C^1$-submanifolds. Consider the maps $\phi_d^\pm: \mc P_0^\pm(d)\ra T\mc T$ sending an element of $\mc P_0^\pm(d)$ to its holonomy. These are $C^1$-maps, see~\cite[Section 4.1]{FP}. In turn, Theorem~\ref{minkowski} also implies that the compositions of these maps with the projections to $\mc T$ are $C^1$-diffeomorphisms. In particular, it follows that $\mc \phi_d^\pm$ are $C^1$-immersions. Now pick $(d^+, d^-) \in \mc D\mc D_0^s$. Then \cite[Theorem II]{FP} implies that the intersection of the images of $\phi_{d^+}^+$ and $\phi_{d^-}^-$ is unique and transverse.

Suppose that for some $x \in \mc P\mc P_0^s$ there exists $y \in T_x\mc P\mc P_0^s$ such that $d\mc I\mc I_0^s(y)=0$. Let $\mc I\mc I_0^s(x)=(d^+, d^-)$. There are natural smooth submersions $\mc P\mc P_0^s \ra \mc P\mc P_0^{s, \pm}$. Let $y^\pm$ be the images of $y$ by these submersions. Since $d\mc I\mc I_0^s(y)=0$ and since the maps $\mc I_0^{s,\pm}$ are $C^1$-submersions, the vectors $y^\pm$ are tangent to $\mc P_0^\pm(d^\pm)$ respectively. There is also the natural projection $\mc P\mc P_0^s \ra T\mc T$. Let $y_0$ be the image of $y$ under this projection. Then \[d\phi_{d^+}^+(y^+)=d\phi_{d^-}^-(y^-)=y_0.\] Hence, the transversality of the images implies that $y_0=0$. But $\phi_{d^\pm}^\pm$ are $C^1$-immersions, thus $y^+=y^-=0$. From this it follows that $y=0$, which finishes the proof.
\end{proof}

Similarly as in Section~\ref{sec:map} one shows that $\mc I\mc I_\vee^s$ is continuous and is $C^1$ around $\pt\mc P\mc P_\vee^s$. Thus, if we show that it is proper, we can finish the proof of Theorem~\ref{main2} exactly the same as we proved Theorem~\ref{main}. 
The proof of the properness, which occupies the next section, is the main part, where we require new ideas, compared to the proof of Theorem~\ref{main}.

\section{Properness for pairs of metrics}
\label{sec:prop2}

We need to show

\begin{lmA}
\label{ml3}
Let $x_i \in \mc P\mc P_\vee^s$ be a sequence such that $y_i:=\mc I\mc I_\vee(x_i)$ converge to $y \in \mc D\mc D_\vee^s$. Then, up to subsequence, $x_i$ converge to $x \in \mc P\mc P_\vee^s$.
\end{lmA}

Again, since the behavior of $\mc I\mc I_\vee$ is completely understood on $\pt\mc D\mc D_\vee^s$, thanks to Theorem~\ref{minkowski2}, we can assume that $x_i \in \mc P\mc P_-^s$, so they can be denoted by $(\theta_i, f_i^+, f_i^-)$. Then $y_i \in \mc D\mc D_-^s$ and can be denoted by $(d_i^+, d_i^-)$. Define $\Omega_i:=\Omega_{\theta_i}$, $C_i:=C_{\theta_i}$. Let $\ct_i$ be the cosmological time of $\Omega_i$, defined on the past of $\pt^- C_i$ in $\Omega_i$.


%
%

\subsection{Convergence of holonomies}
\label{sec:convhol3}

First we show

\begin{lm}
\label{ml3.1}
Under the conditions of Lemma~\ref{ml3}, let $y$ be in $\mc D\mc D_-^s$. Then, up to subsequence, $\theta_i$ converge to $\theta \in \mc T \times \mc T$.
\end{lm}

Similarly to the proof of Lemma~\ref{ml2.1}, it follows from Lemma~\ref{syst} and~\ref{thurstmetric} that, up to subsequence, $\rho_i^\pm \ra \rho^\pm \in \mc T$, where $\rho_i^\pm$ are the holonomies of the intrinsic metrics of $\pt^\pm C_i$. It is a result of Diallo that in such case, up to subsequence, $\theta_i \ra \theta$. See~\cite[Lemma A.2]{BDMS} or the proof of~\cite[Proposition 6.13]{Tam} by Tamburelli.

Now suppose that under the conditions of Lemma~\ref{ml3}, $y$ is in $\pt \mc D\mc D_\vee^s$. We would like to prove that then $\theta_i$ converge to $\theta \in \mc T \times \mc T$ and $\theta$ belongs to the diagonal, so $\theta=(\rho, \rho)$ for some $\rho \in \mc T$. However, we will do it here only under a stronger assumption. Later we will see how we can overcome this assumption by a density argument. To introduce the stronger assumption we first need few preliminaries.

Due to our construction of the blow-up, analogously to Lemma~\ref{metricblow}, we have

\begin{lm}
\label{metricblow2}
Let $(d_i^+, d_i^-) \in \mc D\mc D^c_-$ be a sequence converging to $y \in \pt \mc D\mc D_\vee^s$ in $\mc D\mc D_\vee^s$. Then there exist representative metrics $(d^+, d^-) \in \mc D\mc D_0^s$ of $y$ and a sequence $t_i \ra 0$ such that $\frac{d_i^\pm}{t_i}\ra d^\pm$ in the Lipschitz sense.
\end{lm}

Let $\lambda_i^\pm$, $\Psi_i^\pm$ be the bending laminations and the metric trees associated to $\Omega_i$.
Lemma~\ref{treecomparison}, which bounds the length functions of the trees $\Psi_i^\pm$ by the length functions of Cauchy surfaces, implies that in our setting $\Psi^\pm_i \ra 0$ in $\mc{MT}$. Due to the Skora duality, Theorem~\ref{skora}, this means that $\lambda^\pm_i \ra 0$. The first option is that, up to subsequence, $\lambda_i^\pm=0$. Suppose the other case. Then there exists a sequence $t_i' \ra 0$ such that, up to subsequence, at least one of the sequences $\frac{1}{t_i'}\lambda_i^\pm$ has a nonzero limit in $\mc{ML}$ and the second sequence has a limit, which is possibly zero. We assume that $\frac{1}{t_i'}\lambda_i^- \ra \lambda^- \neq 0$. 

In this case we will additionally assume that all $\lambda^-_i$ are supported on simple closed curves. Afterwards we will show how to overcome it. In other words, we will now prove

\begin{lm}
\label{ml3.2}
Under the conditions of Lemma~\ref{ml3}, let $y$ be in $\pt \mc D\mc D_\vee^s$. Assume that either $\lambda_i^\pm=0$ or that for a sequence $t_i' \ra 0$ we have $\frac{1}{t_i'}\lambda^\pm_i \ra \lambda^\pm \in \mc{ML}$ with $\lambda^- \neq 0$ and all $\lambda^-_i$ are supported on simple closed curves. Then, up to subsequence, $\theta_i$ converge to $\theta=(\rho, \rho) \in \mc T \times \mc T$ for some $\rho \in \mc T$.
\end{lm}

The first proof ingredient is a result of Bonsante--Schlenker~\cite[Corollary 4.10]{BS3}:

\begin{lm}
\label{lengthbound}
There are constants $C, h_0>0$, depending only on the genus of $S$, such that for every $\theta \in \mc R \times \mc R$, in $\Omega_\theta$ we have
\[i(\lambda^+, \lambda^-)\geq C l_{\rho^-}(\lambda^-)\min\{h_0, l_{\rho^-}(\lambda^-)\}.\]
\end{lm}

Here $i: \mc{ML\times ML}\ra \R_{\geq 0}$ is the geometric intersection number, see~\cite{Bon2}.
Our second assumption in Lemma~\ref{ml3.2} implies that ${i(\lambda_i^+, \lambda_i^-)}/{t_i'^2}$ is bounded. Let $\mu_i$ be the simple closed curve supporting $\lambda_i$, so $\lambda_i^-=s_i\mu_i$ for some $s_i>0$. Since no sequence of simple closed curves in $\mc{ML}$ (with weights one) converges to zero and $\frac{s_i}{t_i'}\mu_i$ is bounded in $\mc{ML}$, then $s_i/t_i'$ is bounded. Hence, Lemma~\ref{lengthbound} implies that there exists a constant $C>0$, depending only on $\theta_i$, such that
\begin{equation}
\label{lengthbound2}
l_{\rho_i^-}(\mu_i) \leq C.
\end{equation}
In the case of the first assumption of Lemma~\ref{ml3.2}, we just consider $\mu_i$ empty.
We can now obtain the key ingredient for a proof of Lemma~\ref{ml3.2}.

\begin{lm}
\label{ctbound}
There exists a constant $C>0$, depending only on $\theta_i$ and $y$, such that 
\[\inf_{p \in \Sigma(f^+_i)} \ct_i(p)\geq Ct_i.\]
\end{lm}

\begin{proof}
We first claim that there exists $C_0>0$ with the following property.
Pick $p \in \Sigma(f^+_i)$. 
Then there exists a homotopically non-trivial curve $\chi: [0,1] \ra \Sigma(f^+_i)$ such that it belongs to a single block of the canonical decomposition, $\chi(0)=\chi(1)=p$ and for the projections $\psi$ of $\chi$ to $\pt^- C_i$ along the gradient of the cosmological time we have $\len(\psi)\leq C_0$.

Actually, we first construct $\psi$ and then lift it to $\chi$. Suppose that $p$ is projected to $q \in \pt^- C_i$ along the gradient of $\ct_i$. 
If $q \in \mu_i$, then we can take $\mu_i$ as $\psi$ and consider as $\chi$ the intersection of $\Sigma(f^+_i)$ with any thin block associated to $\psi$. Then~(\ref{lengthbound2}) shows that there exists $C_1>0$ such that $\len(\psi)\leq C_1$. 

If $q \in \pt^-C_\theta \backslash \mu_i$, then we consider a small embedded disk in $\pt^-C_\theta \backslash \mu_i$ around $q$ of radius $r$ in the intrinsic metric of $\pt^-C_\theta$. We start to increase $r$ until either the disk bumps into itself or bumps into $\mu_i$. Since $\pt^-C_\theta$ intrinsically is a hyperbolic surface, due to the Gauss--Bonnet theorem, there exists a constant $C_2>0$, depending only on the genus, such that this happens for $r \leq C_2$. In the first case we obtain $\psi$ as a homotopically non-trivial curve based at $p$ and contained in the relative interior of a component of $\pt^-C_\theta \backslash \mu_i$. In the second case we follow the shortest path from $p$ to $\mu_i$, do one turn around $\mu_i$ and take the shortest path back. The length of the obtained curve is $\leq C_1+2C_2$. We then lift $\psi$ to $\chi$ along the gradient of the cosmological time in a straightforward way. 

We have $\len(\chi)\geq \sys(d_i^+)$. By Lemma~\ref{metricblow2}, for all large enough $i$ we have $\sys(d_i^+)\geq t_i\sys(d^+)/2$. Now we apply Lemma~\ref{timeestimate} and get
\[\ct_i(p) \geq \frac{\len(\chi)}{\len(\psi)\exp(\len(\psi)) } \geq \frac{t_i\sys(d^+)}{2(C_1+2C_2)\exp(C_1+2C_2)},\]
which finishes the proof.
\end{proof}

We need an insight into the intrinsic geometry of the level surfaces of the cosmological time. They are instances of what is called a \emph{grafted metric}. The initial interest to the grafted metrics comes from the study of $\mathbb{CP}^1$-structures on $S$, see, e.g.,~\cite{Dum, KT} for details.

The canonical decomposition of $\Omega_i$ in our case is just one or two thick blocks and zero or one Misner blocks. Looking at the metric inside the blocks (see, e.g.,~\cite[Section 5.4]{BBZ}), we observe that the intrinsic metric of the $r$-level surface $L_{i,r}$ of $\ct_i$, $0<r<\pi/2$, is obtained as follows. Let $m_i^-$ be the intrinsic metric of $\pt^- C_i$, considered as a hyperbolic metric on $S$. Scale it by $\sin r$, cut along $\mu_i$ and glue there the Euclidean cylinder of length equal to the length of $\mu_i$ in the scaled metric and of width equal to $s_i \cos r$. The Euclidean cylinder corresponds to the intersection of $L_{i,r}$ with the Misner block and the rest arises from the intersection with the thick blocks. Now we can prove Lemma~\ref{ml3.2}.

\begin{proof}[Proof of Lemma~\ref{ml3.2}]
Set $r_i:=Ct_i/2$, where $C$ is the constant from Lemma~\ref{ctbound}. Let $L_i$ be the $r_i$-level surface of $\ct_i$ in $\Omega_i$. By Lemma~\ref{ctbound}, $L_i$ belongs to the strict past of $\Sigma(f^+_i)$. By Lemma~\ref{comparison}, for every $\gamma \in \pi_1S$ we have $\len_{L_i}(\gamma)\leq\len_{\Sigma(f^+_i)}(\gamma)=\len_{d_i^+}(\gamma)$. 

Since the intrinsic metric of $L_i$ is obtained by the grafting process from the intrinsic metric of $\pt^- C_i$ scaled by $\sin r_i$, for every $\gamma \in \pi_1 S$ we have $l_{L_i}(\gamma) \geq \sin r_i l_{\rho_i^-}(\gamma)$. Altogether, this and Lemma~\ref{metricblow2} imply that for some constant $C_0>0$ and all large enough $i$ we have 
\[l_{\rho_i^-}(\gamma) \leq C_0l_{d^+}(\gamma).\]
It is a standard compactness criterion for the Teichm\"uller space that there are finitely many classes $\gamma$ in $\pi_1S$ such that if $l_{\rho_i^-}(\gamma)$ are uniformly bounded, then the sequence is precompact in $\mc T$. See, e.g.,~\cite[Lemmas 7.10-7.11]{FLP}. Thus, up to subsequence, $\rho_i^-$ converge to $\rho \in \mc T$. 

We also have $\lambda_i^\pm \ra 0$. By the Mess Theorem, Theorem~\ref{adsearthq}, we get $\rho^l_i, \rho^r_i \ra \rho$. Hence, $\theta_i \ra (\rho, \rho)$. 
\end{proof}

\subsection{End of the proof}
\label{sec:end2}

Convergence of the marked points follows from 

\begin{lm}
\label{ml3.3}
Under the conditions of Lemma~\ref{ml3}, assume that $\theta_i$ converge to $\theta \in \mc T \times \mc T$. Furthermore, if $y \in \pt \mc D\mc D_\vee^s$, assume additionally that $\theta$ is on the diagonal. Then, up to subsequence, $x_i$ converge to $x \in \mc P\mc P_\vee^s$.
\end{lm}

We claim that this, again, follows basically from the same arguments as the respective proofs in Part 1, in Sections~\ref{sec:convmp1} and~\ref{sec:convmp2}. First, consider the case $y \in \mc D\mc D_-^s$. Our results from Section~\ref{sec:convmp1} were given for constant $\rho_i^l$. However, all the conclusions remain valid when instead it belongs to a compact set. This implies the convergence to $x \in \mc P\mc P_\vee^s$ in this case. 

In the case when $y \in \pt \mc D\mc D_\vee^s$, one just needs to modify our construction of the compactification from Section~\ref{sec:compact}. The spaces $\tilde{\mc P}\tilde{\mc P}_0^\bdia$ and $\mc P\mc P_0^\bdia$ are constructed exactly the same as $\tilde{\mc P}_0^\bdia$ and $\mc P_0^\bdia$. Now fix $\rho \in \mc R$, set $\theta=(\rho, \rho)$, and consider a compact neighborhood $\tilde U$ of $\rho$ in $\mc R$ projecting to a compact neighborhood $U$ of $\rho$ in $\mc T$. Then $\tilde{\mc P}\tilde{\mc P}_-^\bdia(\tilde U)$ is defined as a subset consisting of all $(\theta', \tilde f^+, \tilde f^-) \in \tilde{\mc P}\tilde{\mc P}_-$ such that (0) $\theta' \in \tilde U \times \tilde U$; (1) $\tilde f^\pm$ have values in $\tilde \Omega^\pm_{\theta'}\cup \pt^\pm \tilde C_{\theta'}\cup \pt^\pm_s\tilde \Omega^\pm_{\theta'}$ respectively; and (2) $\tilde f^\pm$ are in convex positions. Then $\mc P\mc P_-^\bdia(U)$ is defined as the quotient of $\tilde{\mc P}\tilde{\mc P}_-^\bdia(\tilde U)$ by the $G_-$- and $\pi_1 S$-actions. We define $\tilde{\mc P}\tilde{\mc P}_\vee^\bdia(\tilde U):=\tilde{\mc P}\tilde{\mc P}_-^\bdia(\tilde U)\cup\S(\tilde{\mc P}\tilde{\mc P}_0^\bdia)$, $\mc P\mc P_\vee^\bdia(U):=\mc P\mc P_-^\bdia(U)\cup \S(\mc P\mc P_0^\bdia)$. The first has topology from its inclusion to $\tilde{\mc P}\tilde{\mc P}_\vee$ and the topology on the second is induced as on a quotient of the first. Similarly as in Section~\ref{sec:compact} one can show that $\mc P\mc P_\vee^\bdia(U)$ is compact.

In the context of Lemma~\ref{ml3.3}, when $y \in \pt \mc D\mc D_\vee^s$, we have $\theta=(\rho, \rho)$. Then we pick a compact neighborhood $U$ of this $\rho$ to construct $\mc P\mc P_\vee^\bdia(U)$. The rest of the proof of this case continues exactly the same as in Section~\ref{sec:end1}.

To prove Lemma~\ref{ml3}, it remains only to overcome the additional assumption in Lemma~\ref{ml3.2} that $\lambda_i^-$ are supported on simple closed curves.

\begin{proof}[Proof of Lemma~\ref{ml3}.]
Due to Lemmas~\ref{ml3.1} and~\ref{ml3.3}, it remains to consider the case when $y \in \pt\mc D\mc D_\vee^s$. As it is explained in Section~\ref{sec:convhol3}, there exists a sequence $t_i' \ra 0$ such that, up to subsequence, $\frac{1}{t_i'}\lambda_i^\pm\ra \lambda^\pm$ in $\mc{ML}$ and one of $\lambda^\pm$ is nonzero. We assume that it is $\lambda^-$. The space $\mc P\mc P_\vee^s$ is metrizable, we pick a metric $D$ on it. 

The laminations that are supported on simple closed curves are dense in $\mc{ML}$. Due to this fact and the Mess Theorem, Theorem~\ref{adsearthq}, one can choose a sequence $\hat x_i$ such that for every $i$ we have $D(x_i, \hat x_i) \leq 2^{-i}$, $\mc I\mc I_\vee^s(\hat x_i) \ra y$, $\frac{1}{t_i'}\hat\lambda_i^\pm \ra \lambda^\pm$ and $\hat\lambda_i^-$ are all supported on simple closed curves. Here $\hat\lambda_i^\pm$ are the respective bending laminations of $\Omega_{\hat \theta_i}$, where $\hat \theta_i$ are the holonomies of $\hat x_i$. Hence, Lemma~\ref{ml3.2} implies that, up to subsequence, $\hat \theta_i$ converge to $\theta \in \mc T\times \mc T$. Furthermore, $\theta$ is on the diagonal of $\mc T \times \mc T$. Then Lemma~\ref{ml3.3} implies that, up to subsequence, $\hat x_i$ converge to $x \in \mc P\mc P_\vee^s$. However, the condition $D(x_i, \hat x_i)<2^{-i}$ then shows that $x_i$ also converge to $x$. The proof is finished.
\end{proof}

\begin{appendices}
\renewcommand\thesection{\Alph{section}.}
\section{Intrinsic metrics of convex surfaces}
\label{sec:intmet}
\renewcommand\thesection{\Alph{section}}

Denote by $d_A$ the spacelike distance on $\A^3$, which is defined on the pairs of points in spacelike relation. Let $\chi: [a,b] \ra \A^3$ be a $C^0$-curve. We call it \emph{spacelike} if for every $x \in [a,b]$ there exists its neighborhood $X\subset [a,b]$ such that every two points in $\chi(X)$ are in spacelike relation. Let $t_0=a < t_1 < \ldots < t_n=b$ be a partition of $[a,b]$. The diameter of a partition is $\sup|t_{i+1}-t_i|$. Since $\chi$ is spacelike, when the diameter is small enough, all pairs $\chi(t_i)$, $\chi(t_{i+1})$ are in spacelike relation. We call it a \emph{spacelike partition}. We say that it is \emph{spacelike rectifiable} if there exists
\[\len_A(\chi):=\limsup \sum d_A(\chi(t_i), \chi(t_{i+1})),\]
where the $\limsup$ is taken over spacelike partitions as their diameters tend to zero. Then $\len_A(\chi)$ is the \emph{(spacelike) length} of $\chi$. If $\chi$ is differentiable almost everywhere, then its tangent vectors are non-timelike and
\begin{equation}
\label{spacelikerect}
\len_A(\chi)=\int_a^b\|\dot\chi\|.
\end{equation}
Vice versa, a differentiable almost everywhere curve with non-timelike tangent vectors is spacelike rectifiable.

We say that a surface $\Sigma \subset \A^3$ is \emph{entirely convex} if it is a boundary component of the intersection of a convex subset of $\RP^3$ with $\A^3$. Let $\Sigma$ be an entirely convex spacelike surface. The \emph{intrinsic distance} between two points of $\Sigma$ is the infimum of lengths of all spacelike rectifiable curves in $\Sigma$ connecting the points. Clearly, at least one such curve exists between any pair of points. What is not immediate, however, that for distinct points the defined distance is positive, and, more generally, that the topology induced by the obtained intrinsic pseudo-metric is the same is the initial topology of $\Sigma$ as of a submanifold of $\A^3$.

%
%

\begin{lm}
\label{sametopology}
The intrinsic pseudo-metric $d$ is a metric and does not alter the topology $\Sigma$.
\end{lm}

\begin{proof}
Pick $p \in \Sigma$. Consider the Minkowski chart and the standard Euclidean metric $d_E$ on it. We may assume that $p=o$ and the horizontal plane $\Pi$ is supporting for $\Sigma$ at $p$. For a small enough neighborhood $U$ of $p$ on $\Sigma$ in the standard topology there exist $A_1, A_2>0$ such that for any $p,q \in U$ we have 
\[A_1d_A(p,q)\leq d_E(p,q)\leq A_2d_A(p,q).\]
Thus, spacelike rectifiable curves inside $U$ are Euclidean rectifiable and vice versa. It implies that $d$ is a metric and that every neighborhood of $p$ with respect to the standard topology contains a neighborhood with respect to $d$ and vice versa. Since both topologies are metric, it follows that $d$ does not alter the topology of $\Sigma$.
\end{proof}

We note that it can happen that the intrinsic metric is incomplete. Since $\Sigma$ is locally compact, when $d$ is complete, a standard application of the Arzel\`a--Ascoli theorem implies the existence of a shortest path between any pair of points on $\Sigma$.

%
%

A \emph{convex body} is a closed convex set $C \subset \RP^3$ with non-empty interior. We call it \emph{spacelike} if every plane supporting it at a point in $\pt C \cap \A^3$ is spacelike. Let $C_i$ be a sequence of spacelike convex bodies converging to a spacelike convex body $C$. Let $\Sigma_i$, $\Sigma$ be connected components of $\pt C_i \cap \A^3$, $\pt C \cap \A^3$. We assume that $\cl(\Sigma_i)$ converge to $\cl(\Sigma)$ and $\pt C_i \backslash \Sigma_i$ converge to $\pt C \backslash \Sigma$. Pick a point $p \in \inter(C)$, we assume that $p \in \inter(C_i)$ for all $i$. Let $S^2$ be the space of directions from $p$, $D_i$ and $D$ be the projections of $\Sigma_i$, $\Sigma$ to $S^2$. Then $\cl(D_i)$ converge to $\cl(D)$ and $S^2 \backslash D_i$ converge to $S^2 \backslash D$. We consider the intrinsic metrics $d_i$, $d$ of $\Sigma_i$, $\Sigma$ pushed forward to $D_i$, $D$. We assume that they are complete. Pick $p,q \in D$, let $p_i$, $q_i$ be two sequences converging to $p$ and $q$ respectively in $D$. We have $p_i,q_i \in D_i$ for all large enough $i$. We want to prove

\begin{lm}
\label{distconverge0}
We have $d_i(p_i,q_i) \ra d(p,q)$. Furthermore, up to subsequence, there are shortest paths between $p_i$ and $q_i$ for $d_i$ whose images converge in the Hausdorff sense to the image of a shortest path between $p$ and $q$ for $d$. 
\end{lm}

Note that one could show that $d_i$ and $d$ are CAT(0), hence the shortest paths are unique. We will need it only in particular, rather evident cases, so we will not prove it in full generality.

We need to make another digression first. Let $\Pi \subset \A^3$ be a spacelike plane. We denote its metric by $d_\Pi$ and its length function by $\len_\Pi$. The past-directed normal exponential map from $\Pi$ is a diffeomorphism $\mc E_\Pi$ from $\Pi \times [0, \pi_2)$ onto the image. Let $K \subset \Pi$ be a compact convex subset with nonempty interior. We call a function $s: K \ra [0, \pi/2)$ \emph{C-convex} if it is continuous and its graph with respect to $\mc E_\Pi$ is spacelike and future-convex. Here we say that a convex surface with boundary in $\A^3$ is spacelike if each supporting plane at the interior points is spacelike and those supporting planes at the boundary points that are the limits of supporting planes at intrinsic points are spacelike. We will follow the paper~\cite{Lab2} of Labeni, who treated the intrinsic geometry of graphs of C-convex functions. We note that Labeni works with the functions defined over $\Pi$, which does not matter for our context. Following the prior work~\cite{FS3} of Fillastre--Slutskiy on the Minkowski case, Labeni makes few technical assumptions on the functions he works with. To apply his work, we will need now to show that they are actually unnecessary, i.e., they hold for all C-convex functions.

Let $\chi: [a,b] \ra K$ be a Lipschitz curve for $d_\Pi$. Then $s\circ \chi$ is a Lipschitz function, see~\cite[Section 2.2]{Lab2}. In particular, it is differentiable almost everywhere and so is the respective curve $\chi_s:[a,b] \ra \A^3$ in the graph of $s$ obtained via $\mc E_\Pi$. The tangent vectors to $\chi_s$ are spacelike, so it is spacelike rectifiable. Labeni defines $L_s(\chi):=\int_a^b\|\dot \chi_s\|$. Due to~(\ref{spacelikerect}), we have $L_s(\chi)=\len_A(\chi_s)$. Labeni defines an intrinsic metric $d_s$ on $K$ from the length structure $L_s$.

On the other hand, let $\chi_s: [a,b] \ra \A^3$ be a spacelike rectifiable curve in the graph of $s$ and $\chi: [a,b] \ra K$ be its projection. We assume that $\Pi$ is the horizontal plane of the Minkowski chart and pick the standard Euclidean metric there. It is easy to see from compactness that there exists a constant $A>0$ such that the Euclidean length of every chord of the graph of $s$ is at most $A$ times the spacelike length. Hence, $\chi_s$ is rectifiable for the Euclidean metric. Hence, $\chi$ is rectifiable for the Euclidean metric on $\Pi$. But then, again due to compactness, $\chi$ is rectifiable for $d_\Pi$. 

Now we consider a curve $\chi: [a,b] \ra K$ that is rectifiable for $d_s$. Since the projections to $\Pi$ of spacelike segments in the image of $\mc E_\Pi$ are rectifiable for $d_\Pi$ and thus for $d_s$, one sees that the respective $\chi_s$ is spacelike rectifiable and $\len_A(\chi_s) \leq \len_{d_s}(\chi)$, where $\len_{d_s}$ is the length structure induced by $d_s$. By the argument above, $\chi$ is rectifiable for $d_\Pi$. 

Furthermore, if $\chi: [a,b] \ra K$ is rectifiable for $d_\Pi$, then trivially $\len_{d_s}(\chi)\leq L_s(\chi)=\len_A(\chi_s)$. Since $\len_A(\chi_s) \leq \len_{d_s}(\chi)$, we get 
\[\len_{d_s}(\chi)= L_s(\chi)=\len_A(\chi_s).\]

Let us sum it up. The rectifiable curves for $d_s$ are rectifiable for $d_\Pi$ and vice versa. Furthermore their $d_s$-lengths coincide with their $L_s$-lengths and with $\len_A$ of their images in the graph of $s$. We denote this length structure now by $\len_s$ on $U$. Just the same proof as of Lemma~\ref{sametopology} shows that $d_s$ does not alter the topology of $U$. Our conclusions allow us to apply the results of Labeni. In~\cite[Lemma 2.11]{Lab2}, Labeni showed

\begin{lm}
\label{labeni1}
Over $K$ we have $d_s \leq d_\Pi$.
\end{lm}

Let $s_i$ be a sequence of C-convex functions on $K$ converging uniformly to a C-convex function $s$. Define $d_i:=d_{s_i}$, $\len_i:=\len_{s_i}$. It follows from~\cite[Lemma 3.4]{Lab2} that

\begin{lm}
\label{labeni2}
There exists $A>0$ such that for all $i$ we have $d_i \geq A\cdot d_\Pi$ as well as $d_s \geq A\cdot d_\Pi$ over $K$.
\end{lm}

From Lemma~\ref{labeni1} and~\ref{labeni2} it follows

\begin{crl}
\label{labeni3}
There exist $A_1, A_2>0$ such that over $K$ for all $i$ we have 
\[A_1 d_s \leq d_i \leq A_2 d_s,\]
\[A_1 \len_s \leq \len_i \leq A_2 \len_s.\]
\end{crl}

Furthermore,~\cite[Proposition 2.9]{Lab2} gives us

\begin{lm}
\label{labeni4}
Let $\chi: [0,1] \ra K$ be a rectifiable curve. Then $\len_i(\chi) \ra \len_s(\chi)$.
\end{lm}

Note that Labeni states his result ``up to subsequence'', which one overcomes by applying Lemma~\ref{subconvtoconv}. Using~\cite[Lemma 2.5 and Lemma 2.1]{FS3}, we deduce

\begin{lm}
\label{uniformconv}
We have $d_i \ra d_s$ uniformly on $K$.
\end{lm}

Now we show

\begin{lm}
\label{localize}
For every $p \in \inter(K)$ there exists its neighborhood $U_p \subset \inter(K)$ such that for all large enough $i$ and all $q, q'  \in U_p$ every shortest path between $q$ and $q'$ for $d_i$ is contained in $\inter(K)$. The same claim holds for $d_s$.
\end{lm}

\begin{proof}
Consider the first claim.
Suppose the converse. Then, up to subsequence, there are $q_i \ra p$, $q_i' \ra p$ such that there is a shortest path $\chi_i: [0,1] \ra K$ for $d_i$ between $q_i$ and $q_i'$ that contains a point of $\pt K$. Because of Corollary~\ref{labeni3}, we have $\len_i(\chi_i)=d_i(q_i, q'_i)\ra 0$. By applying Corollary~\ref{labeni3} again, we get $\len_s(\chi_i)\ra 0$. After reparameterizing $\chi_i$ proportional to $\len_s$, by applying the Arzel\`a--Ascoli theorem, up to subsequence, $\chi_i$ converge to a curve $\chi: [0,1] \ra K$ that passes through $p$ as well as through a point of $\pt K$ and $\len_s(\chi)\leq \liminf \len_s(\chi_i)=0$. This is a contradiction to that $d_s$ induces the standard topology on $K$. The second claim is proven the same.
\end{proof}

Now we return to our previous setting. We denote the lengths structures of $d$, $d_i$ by $\len$, $\len_i$. Since spacelike surfaces are locally graphs over spacelike planes, Lemma~\ref{localize} together with Corollary~\ref{labeni3} and Lemma~\ref{labeni4} yield

\begin{crl}
\label{metricbilip}
For every $p \in D$ there exists its neighborhood $U_p \subset D$ and $A_1, A_2>0$, depending on $U_p$, such that over $U_p$ for all $i$ we have 
\[A_1 d \leq d_i \leq A_2 d,\]
\[A_1 \len \leq \len_i \leq A_2 \len.\]
\end{crl}

\begin{crl}
For every $p \in D$ there exists its neighborhood $U_p \subset D$ such that if $\chi: [0,1] \ra U_p$ is rectifiable curve, then $\len_i(\chi) \ra \len(\chi)$.
\end{crl}

In turn, these imply

\begin{crl}
\label{lengthsblilp}
For every compact $K \subset D$ there exist $A_1, A_2>0$, depending on $K$, such that for all $i$ over rectifiable curves in $K$ we have 
\[A_1 \len \leq \len_i \leq A_2 \len.\]
\end{crl}

\begin{crl}
\label{lengthsconv}
Let $\chi: [0,1] \ra D$ be a rectifiable curve. Then $\len_i(\chi) \ra \len(\chi)$.
\end{crl}

Furthermore, we will need the following technical results.

\begin{lm}
\label{lengthsdiagbound}
Let $\chi_i: [0,1] \ra D$ be a sequence of rectifiable curves converging uniformly for $d$ to a rectifiable curve $\chi: [0,1] \ra D$. Then $\len(\chi)\leq \liminf \len_i(\chi_i)$.
\end{lm}

\begin{proof}
We use some ideas from the proof of~\cite[Proposition 3.12]{FS3}. Pick a partition $t_0=0<t_1<\ldots<t_n=1$ and $\e>0$. By Corollary~\ref{metricbilip}, there exists $A>0$ and for every $j=0,\ldots, n$ there exists a neighborhood $U_j$ of $\chi(t_j)$ such that if $p\in U_j$, then \[d_i(\chi(t_j), p)< A\cdot d(\chi(t_j),p).\] Then for all large enough $i$ we have
\[d_i(\chi(t_j),\chi_i(t_j))\leq A\cdot d(\chi(t_j),\chi_i(t_{j}))\leq\frac{\e}{n+1}.\]
Hence $\sum_jd_i(\chi(t_j),\chi_i(t_j))\leq \e$. By the triangle inequality,
\[\sum_j d_i(\chi(t_j),\chi(t_{j+1}))\leq \sum_j d_i(\chi_i(t_j),\chi_i(t_{j+1}))+2\e.\]
By taking the suprema over partitions, we get $\len_i(\chi)\leq \len_i(\chi_i)+2\e.$ By Corollary~\ref{lengthsconv},  $\len_i(\chi)\ra\len(\chi)$. Since $\e>0$ is arbitrary, it follows that $\len(\chi)\leq \liminf\len_i(\chi_i)$. 
\end{proof}

\begin{lm}
\label{catchgeodesics}
For $p \in D_i$ and $r \in \R_{>0}$ denote by $B_i(p,r)$ the closed $r$-ball for $d_i$ around $p$. For every $p\in D$ and every $r \in \R_{>0}$ there exists compact $K \subset D$ such that for all large enough $i$ we have $B_i(p,r) \subset K$. 
\end{lm}

\begin{proof}
Suppose the converse. 
We fix $p$ and vary $r$. If the claim is true for some value of $r$, then trivially it is true for all smaller values. For all small enough $r$ the claim is true by Corollary~\ref{metricbilip}. Let $r_0>0$ be the supremum of those $r$ for which the claim is true. Suppose that the claim is true for $r_0$. Let $K\subset D$ be the respective compact set. Pick a strictly decreasing sequence $r_i$ converging to $r_0$. Then, up to subsequence, there exist a sequence $\chi_i:[0,r_i] \ra D_i$ of shortest paths for $d_i$ parameterized by lengths such that $\chi(0)=p$ and $q_i:=\chi_i(r_i)$ leave every compact subset of $D$. Define $x_i:=\chi_i(r_0)$. Then $x_i \in K$. Up to subsequence, $x_i$ converge to $x \in K$. Pick a compact neighborhood $U_x \ni x$ in $D$ from Corollary~\ref{metricbilip}. Pick a simple closed curve $Y$ around $x$ in $U_x$. Since $q_i$ leave every compact set, for all large enough $i$ we have $Y \cap \chi_i((r_0, r_i]) \neq \emptyset$. Pick $y_i$ in this intersection. Up to subsequence, $y_i$ converge to $y \in Y$. Thus $y\neq x$, so $d(x, y)>0$. On the other hand, we have $d_i(x_i, y_i)\leq r_i-r_0\ra 0$. By Corollary~\ref{metricbilip}, there exists $A>0$ such that for all large enough $i$ we have $d(x_i,y_i)\leq A\cdot d_i(x_i, y_i)$. Then $d(x_i, y_i)\ra 0$. This is a contradiction.

Now suppose that the claim is not true for $r_0$. Pick a strictly increasing sequence $r_j$ converging to $r_0$. Let $\chi_i: [0, r_0] \ra D_i$ be a sequence of shortest paths for $d_i$ parameterized by lengths such that $q_i=\chi_i(r_0)$ leave every compact subset of $D$. Denote by $\chi_i^j$ the restriction of $\chi_i$ to $[0, r_j]$. For every $j$ and all large enough $i$, by assumption, $\chi_i^j$ belong to compact $K_j \subset D$. By Corollary~\ref{lengthsblilp}, there exists $A_j>0$ such that for all $i$ we have $\len(\chi_i^j)\leq A_j\len_i(\chi_i^j)\leq A_jr_0$. By the Arzel\`a--Ascoli theorem, after a reparameterization, up to subsequence, $\chi_i^j$ converge in $i$ to $\chi^j: [0, r_j] \ra K_j$ uniformly for $d$. By Lemma~\ref{lengthsdiagbound}, we have $\len(\chi^j)\leq \liminf_i \len_i(\chi_i^j)\leq r_j$. We do this subsequently, passing to further subsequences, and construct a curve $\chi: [0, r_0) \ra D$ such that for every $r_j$ from the sequence we have $\len(\chi|_{[0, r_j]})\leq r_j$. On the other hand, $\chi$ leaves every compact subset of $D$. This contradicts to completness of $d$.
\end{proof}

\begin{proof}[Proof of Lemma~\ref{distconverge0}.]
We prove the first claim also up to subsequence, then we can get rid of it by Lemma~\ref{subconvtoconv}. Let $\chi$ be a shortest path between $p$ and $q$ for $d$, $\chi'_i$ be a shortest path between $p_i$ and $p$ for $d_i$ and $\chi''_i$ be a shortest path between $q$ and $q_i$ for $d_i$. Corollary~\ref{metricbilip} shows that $\len_i(\chi'_i)\ra 0$, $\len_i(\chi''_i)\ra 0$. Corollary~\ref{lengthsconv} says that $\len_i(\chi) \ra \len(\chi)$. By considering the concatenation of paths $\chi'_i$, $\chi$ and $\chi''_i$, we see that $\limsup d_i(p_i,q_i) \leq d(p,q)$. 

Due to Corollary~\ref{metricbilip} and Lemma~\ref{catchgeodesics}, there exists a compact $K\subset D$ such that for all large enough $i$ all shortest paths between $p_i$ and $q_i$ for $d_i$ belong to $K$.
Let $\chi_i: [0,1] \ra D_i$ be shortest paths between $p_i$ and $q_i$ for $d_i$. Then, for all large enough $i$ the images of $\chi_i$ are in $K$. Due to Lemma~\ref{lengthsblilp}, there exists $A>0$ such that $\len(\chi_i)\leq A\len_i(\chi_i)$ for all $i$. Since $\limsup d_i(p_i,q_i) \leq d(p,q)$, we have $\len(\chi_i)$ bounded. Hence, after a reparameterization, we can apply the Arzel\`a--Ascoli theorem and see that, up to subsequence, $\chi_i$ converge to a curve $\chi: [0,1] \ra K$ uniformly in $d$. By Lemma~\ref{lengthsdiagbound}, $\len(\chi) \leq \liminf \len_i(\chi_i)$. Thus $d(p,q)\leq \liminf d_i(p_i,q_i).$ This also means that $\chi$ is a shortest path between $p$ and $q$ for $d$, whose image is the Hausdorff limit of the images of $\chi_i$.
\end{proof}

Among others, we will need the following two applications of Lemma~\ref{distconverge0}. For $\rho \in \mc T$ and a convex Cauchy surface $\Sigma \subset \Omega_\rho$ let $l_\Sigma: \pi_1S \ra \R_{>0}$ be the \emph{length function} of $\Sigma$ sending $\gamma \in \pi_1S$ to the infimum of lengths of closed curves in the free homotopy class of $\gamma$. 

\begin{lm}
\label{lensystconv}
Let $\Sigma_i \subset \Omega_\rho$ be a sequence of future-convex Cauchy surfaces converging in the Hausdorff sense to a Cauchy surface $\Sigma$. Then $l_{\Sigma_i} \ra l_\Sigma$ pointwise.
\end{lm}

\begin{proof}
Pick $\tilde \Omega_\rho$ and let $\tilde \Sigma_i$, $\tilde \Sigma$ be the preimages of $\Sigma_i$, $\Sigma$. Then $\cl(\Sigma_i) \ra \cl(\Sigma)$. Using the convex bodies bounded by $\tilde\Sigma_i \cup \Lambda_\rho \cup \pt^-\tilde\Omega_\rho$ we get to the setting of Lemma~\ref{distconverge0}. Pick $\gamma \in \pi_1S$. From considering a representative curve in $D$ for $l_\Sigma(\gamma)$ and using Lemma~\ref{distconverge0}, we see that $\limsup l_{\Sigma_i}(\gamma)\leq l_\Sigma(\gamma)$. One can pick a compact fundamental domain $F \subset D$ for the action of $\pi_1S$ on $D$ coming from the action on $\tilde\Sigma$ and pick compact fundamental domains $F_i$ for the actions on $D_i$ coming from $\tilde\Sigma_i$ so that $F_i \ra F$ in the Hausdorff sense. Pick representative curves $\chi_i$ for $l_{\Sigma_i}(\gamma)$ in $D_i$ starting in $F_i$. Since $F_i \ra F$ and since $\limsup l_{\Sigma_i}(\gamma)\leq l_\Sigma(\gamma)$, for an arbitrary point $p \in F$ there exists $r>0$ such that the endpoints of $\chi_i$ are contained in $B_i(p,r)$, where $B_i(p,r)$ is the closed $r$-ball for $d_i$ around $p$. Lemma~\ref{catchgeodesics} states that then these endpoints belong to a compact set $K \subset D$. Then it follows from Lemma~\ref{distconverge0} that $l_\Sigma(\gamma) \leq \liminf l_{\Sigma_i}(\gamma)$.
\end{proof}

Next, recall the notation from Section~\ref{sec:bentsurf}. For $(\rho, f) \in \mc P_-^w$ and $v \in V$ such that $f(v) \in \Sigma(f)$, define $l_{f,v}: \pi_1S \ra \R_{>0}$ to send $\gamma \in \pi_1S$ to the infimum of lengths of closed curves on $\Sigma$ based at $f(v)$ in the homotopy class of $\gamma$.

\begin{lm}
\label{notalldegen}
Let $(\rho_i, f) \ra (\rho, f)$ in $\mc P_-^w$ and for $v \in V$ we have $f_i(v) \in \Sigma(f_i)$. Then also $f(v) \in \Sigma(f)$ and $l_{f_i,v}\ra l_{f,v}$ pointwise.
\end{lm}

\begin{proof}
The claim that $f(v) \in \Sigma(f)$ is clear from the topology of $\mc P_-^w$, we only need to show the second claim. We lift all to $\tilde{\mc P}_-^w$ so that we are in the setting of Lemma~\ref{distconverge0}. Clearly, $l_{f_i, v}(\gamma)=d_i(\tilde f_i(v), \tilde f_i(\gamma v))$, $l_{f,v}(\gamma)=d(\tilde f(v), \tilde f(\gamma v))$. Now the claim follows from Lemma~\ref{distconverge0}.
\end{proof}

%
%
%
%
\renewcommand\thesection{\Alph{section}.}
\section{A Busemann--Feller-type lemma}
\label{sec:busfel}
\renewcommand\thesection{\Alph{section}}

Here we prove a fact that seems important on its own. Pick arbitrary $\rho \in \mc T$. The goal of this section is to show

\begin{lm}
\label{comparison}
Let $\Sigma \subset \Omega_\rho$ be a future-convex Cauchy surface and $\Sigma_0 \subset \Omega_\rho$ be a $C^1$ Cauchy surface in the strict past of $\Sigma$. 
Then for any $\gamma \in \pi_1S$ we have $l_{\Sigma_0}(\gamma)\leq l_\Sigma(\gamma)$. 
\end{lm}


This lemma resembles a Lorentzian analogue of the renowned Busemann--Feller lemma. There are multiple results of this flavor in the literature: see, e.g.,~\cite[Proposition 6.1]{BBZ}, \cite[Proposition 4.1]{Bel} or~\cite[Lemma 5.3]{BSS}. However, they all do not apply to the generality that we require. We begin with

\begin{lm}
\label{expand}
Suppose that we have a $C^2$-foliation of a globally hyperbolic (2+1)-spacetime $\Omega$ by future-convex Cauchy surfaces.
Let $L \subset \Omega$ be a leaf of the foliation and $\chi: [0,1] \ra \Omega$ be a spacelike rectifiable curve in the past of $L$. Let $\psi$ be the projection of $\chi$ to $L$ along the normal flow of the foliation. Then $\len(\chi)\leq \len(\psi)$.
\end{lm}

This is shown in~\cite[Proposition 6.1]{BBZ} by Barbot--Beguin--Zeghib. Note that there the authors speak about a precise foliation, but they only use the convexity of the leaves.

To prove Lemma~\ref{comparison}, we need to construct a foliation. First, we will employ

\begin{lm}
\label{smoothing}
Let $\Sigma$ be a future-convex Cauchy surface in $\Omega_\rho$. Then it can be approximated in the Hausdorff sense by smooth strictly future-convex Cauchy surfaces.
\end{lm}

This is shown in~\cite[Lemma 4.2]{BS4} by Bonsante--Schlenker. Next, we need

\begin{lm}
\label{flow}
Let $\Sigma$ be a smooth strictly future-convex Cauchy surface in $\Omega_\rho$ and $L_K \subset \Omega_\rho$ be a future-convex Cauchy surface of constant Gauss curvature $K$ in the strict past of $\Sigma$ with $K$ greater than the supremum of the Gauss curvature of $\Sigma$. Then the domain between $\Sigma$ and $L_K$ in $\Omega_\rho$ is $C^3$-smoothly foliated by strictly future-convex Cauchy surfaces. The foliation is $C^3$-smooth on $\Sigma$, though not necessarily on $L_K$. 
\end{lm}

\begin{proof}
Let $\Sigma^* \subset \Omega$ be the dual surface. To obtain it, one considers the preimage $\tilde\Sigma \subset \tilde\Omega_\rho$, defines the dual surface $\tilde\Sigma^*$ as the set of the dual points to the supporting planes to $\tilde\Sigma$ and considers the $\theta_\rho$-quotient $\Sigma^*$ of $\tilde\Sigma^*$. Then $\Sigma^*$ is a smooth strictly past-convex Cauchy surface, whose Gauss curvature at a point is reciprocal to the Gauss curvature at the corresponding point of $\Sigma$, see~\cite[Section 11]{BBZ}. Similarly, one obtains the dual surface $L_K^* \subset \Omega_\rho$. It follows that it is a past-convex $K^*$-surface of curvature $K^*=1/K$.

Now we claim that there exists a $C^4$-smooth foliation of the domain between $\Sigma^*$ and $L^*_{K}$ by strictly past-convex Cauchy surfaces. That follows from results of Gerhardt~\cite{Ger2}. Namely, $K^*$ is less than the infimum of the Gauss curvature of $\Sigma^*$. In~\cite{Ger2} the author describes a curvature flow that starts from $\Sigma^*$, evolves into the past, exists for all time and converges to a surface of constant curvature $K^*$. For the long-time existence and convergence of the flow, however, it is required to know that there exists a lower barrier, i.e., a smooth Cauchy surface in the past of $\Sigma^*$, whose supremum of the curvature is at most $K^*$. We use $L_K^*$ for this purpose. Due to the maximum principle~\cite[Corollary 4.7]{BBZ}, the past-convex surface of curvature $K^*$ is unique in $\Omega_\rho$, so the flow converges to $L_{K}^*$. By dualizing the flow, we construct the desired flow from $\Sigma$ to $L_K$. 
\end{proof}

\begin{proof}[Proof of Lemma~\ref{comparison}.]
Pick $\gamma \in \pi_1S$. Due to Lemma~\ref{smoothing}, $\Sigma$ can be approximated in the Hausdorff sense by smooth strictly future-convex surfaces. Due to Lemma~\ref{lensystconv}, for arbitrary $\e>0$ there is such a surface $\Sigma'$ and a curve $\psi: [0,1] \ra \Sigma'$ such that $\len(\psi)\leq l_\Sigma(\gamma)+\e$. Additionally, we can pick $\Sigma'$ in the future of $\Sigma_0$. Due to Theorem~\ref{BBZ}, there exists a $K$-surface $L$ in the strict past of $\Sigma_0$. By Lemma~\ref{flow}, there exists a $C^3$-smooth foliation of the domain between $\Sigma'$ and $L$ by strictly future-convex Cauchy surfaces. Let $\chi: [0,1]\ra \Sigma_0$ be the projection of $\psi$ to $\Sigma_0$ along the normal flow of the foliation. By Lemma~\ref{expand}, $\len(\chi)\leq\len(\psi)\leq l_\Sigma(\gamma)+\e$.
Because $\e$ is arbitrary, we obtain the desired result. 
\end{proof}
\end{appendices}

\bibliographystyle{abbrv}
\bibliography{polyhedralAdS}

\end{document}